\documentclass[11pt]{article}

\usepackage{amsfonts,amssymb,amscd,xy}
\usepackage[leqno]{amsmath}
\usepackage{mathrsfs}
\usepackage{amssymb}
\usepackage{amsxtra}
\usepackage{amscd}
\usepackage{amsthm}
\usepackage[mathscr]{eucal}
\usepackage{color}
\usepackage{epic,eepic}
\usepackage{tabularx}
\usepackage{epsfig,amssymb}
\usepackage{bm} 

\usepackage[usernames,dvipsnames]{xcolor}

\theoremstyle{plain}
\newtheorem{lemma}{Lemma}[subsection]{\bf}{\it}
\newtheorem{theorem}[lemma]{Theorem}{\bf}{\it}
\newtheorem{proposition}[lemma]{Proposition}{\bf}{\it}
\newtheorem{corollary}[lemma]{Corollary}{\bf}{\it}
\newtheorem*{theorem*}{Theorem}  

\theoremstyle{definition}
\newtheorem{definition}[lemma]{Definition}{\bf}{\rm}
\newtheorem{remark}[lemma]{Remark}{\bf}{\rm}
{\bf}{\rm}
{\bf}{\rm}
{\bf}{\rm}

\newcommand{\G}{\mathbb{G}}
\newcommand{\PP}{\mathrm{P}}
\newcommand{\HH}{\mathrm{H}}

\newcommand{\Z}{{\mathbb Z}}
\newcommand{\Q}{{\mathbb Q}}
\newcommand\OO{\mathcal{O}}
\newcommand{\QQ}{\mathcal{Q}}
\newcommand{\R}{\mathrm{R}}

\newcommand{\Gr}{{\rm Gr}}
\newcommand{\GGr}{{\bf Gr}}

\newcommand{\spec}{{\rm Spec}}
\newcommand{\Pic}{{\rm Pic}}

\newcommand{\Sch}{\mathbf{Sch}}

\newcommand{\ab}{\mathrm{ab}}
\newcommand{\id}{\mathrm{id}}

\newcommand{\ft}{{\rm ft}}
\newcommand{\tor}{{\rm tor}}
\newcommand{\fr}{{\rm fr}}
\newcommand{\fl}{{\rm fl}}

\begin{document}
\input xy
\xyoption{all}

\title{On torsors under elliptic curves and Serre's pro-algebraic structures}
\author{Alessandra Bertapelle \thanks{Dipartimento di Matematica, via Trieste 63, I-35121 Padova, Italy, Email: alessandra.bertapelle@unipd.it, Tel.: ~+39 0498271412, Fax: +39 0498271479}
\and Jilong Tong \thanks{Institut de Math\'ematiques de Bordeaux UMR 5251, Universit\'e de Bordeaux 1, 351, cours de la Lib\'eration - F 33405 Talence Cedex, France, Email: jilong.tong@math.u-bordeaux1.fr}
}
\date{\today}

\maketitle

\begin{abstract} Let $\OO_K$ be a complete discrete valuation ring with algebraically closed residue field of positive characteristic and field of fractions $K$. Let $X_K$ be a torsor under an elliptic curve $A_K$ over $K$, $X$ the proper minimal regular model of $X_K$ over $S:=\spec(\OO_K)$, and $J$ the identity component of the N\'eron model of $\Pic_{X_K/K}^{0}$. We study the canonical morphism $q\colon \mathrm{Pic}^{0}_{X/S}\to J$ which extends the natural isomorphism on generic fibres. We show that $q$ is pro-algebraic in nature with a construction that recalls Serre's work on local class field theory. Furthermore, we interpret our results in relation to Shafarevich's duality theory for torsors under abelian varieties.
\end{abstract}
\medskip



\section*{Introduction}
This paper concerns some local studies of torsors under an elliptic curve, or more generally, under an abelian variety.
In the following, let $\OO_K$ be a complete discrete valuation ring with field of fractions $K$ and algebraically closed residue field $k$ of positive characteristic $p>0$. Let $\pi\in \OO_K$ be a uniformizer of $\OO_K$.
Let $S:=\spec(\OO_K)$, denote by $s$ its closed point and by $i\colon \spec(k)\to S$ the usual closed immersion.
Let $A_K$ be an elliptic curve over $K$ and $X_{K}$ a torsor under $A_{K}$ of order $d$. Then $X_K$ is a smooth projective $K$-curve whose jacobian is canonically isomorphic to $A_{K}'$, the dual of $A_K$ (see Lemma~\ref{J_K et A_K} (i)). Let $X$ be the proper minimal regular model of $X_K$ over $S$. In general $X$ is not cohomologically flat in dimension $0$ over $S$ (\emph{i.e.}, the canonical morphism $k\rightarrow \HH^{0}(X_{s},\mathcal{O}_{X_s})$ is not an isomorphism), and the relative Picard functor $\Pic^{0}_{X/S}$ is not representable, not even by an algebraic space.
Nevertheless, the functor $\mathrm{Pic}^{0}_{X/S}$ is not very far from being representable. For example, let $J$ denote the identity component of the N\'eron model of $J_K:=\Pic_{X_K/K}^{0}$ over $S$. In \cite{Raynaud}, 2.3.1, (see \S~\ref{Rappels Pic} for a brief summary) it is shown that there exists an epimorphism of fppf-sheaves
\begin{eqnarray*}
q\colon \Pic^{0}_{X/S}\longrightarrow J
\end{eqnarray*}
that extends the natural isomorphism on generic fibres.
This morphism plays a very important role in a recent work of Liu, Lorenzini and Raynaud where, by considering the induced map $\mathrm{Lie}(q)$ between the Lie algebras of $\mathrm{Pic}^{0}_{X/S}$ and $J$, together with a result of T.~Saito, the authors prove a beautiful result about the geometry of the scheme $X$, namely that the Kodaira type of the special fibre $X_s$ of $X$ is exactly $d$ times the Kodaira type of the special fibre of the minimal regular $S$-model of the elliptic curve $J_K$ (\cite{LLR}, Theorem 6.6).

One of the aims of this paper is to study the morphism $q$ in order to reveal other interesting properties.
More precisely, consider the surjective map induced by $q$ on the $S$-sections (see the end of \S~\ref{Rappels Pic} for the surjectivity of $q$):
\begin{eqnarray}\label{eq.morpfonsec}
q=q(S)\colon \mathrm{Pic}^{0}(X)=\mathrm{Pic}_{X/S}^{0}(S)\longrightarrow J(S).
\end{eqnarray}
Since the gcd of the multiplicities of the irreducible components of $X_s$ is $d$ (see Lemma \ref{d1=d2=d}), one finds that $D:=\frac{1}{d}X_s$ is a well-defined effective divisor of $X$, whose sheaf of ideals $\mathcal{I}:=\mathcal{O}_{X}(-D)\subset \mathcal{O}_{X}$ is invertible of order $d$, and generates the kernel of $q$.
With the help of Greenberg realization functors, one can show that the morphism $q$ is in fact \emph{pro-algebraic} in nature, and we get a short exact sequence of \emph{pro-algebraic groups} over $k$ (see \S~\ref{sec.proalg-gree} for a review on pro-algebraic groups and Greenberg realizations):
\begin{equation}\label{eq.ZPicJ}
\xymatrix{0\ar[r] & \Z/d\Z\ar[r] & \bm{\Pic^0(X)}\ar[r]^{\bm q} & \bm{ J(S)}\ar[r] & 0},
\end{equation}
where the second map is given by sending $\bar{1}\in \mathbb{Z}/d\mathbb{Z}$ to $\mathcal{O}_X(D)\in \bm{\Pic^{0}(X)}(k)=\Pic^{0}(X)$ (see Corollary~\ref{cor.finalresult} and \eqref{eq.d0}).

One of the main results of this paper shows that the morphism $q$ in \eqref{eq.morpfonsec} can be thought as an analogue of the norm map studied by Serre in his work on local class field theory \cite{SerreCFT}.
Let us first briefly review Serre's results.
Let $L/K$ be a finite Galois extension of $K$ with Galois group $\Gamma_{L/K}$, and let $U_K$ (respectively $U_L$) be the group of units of the valuation ring $\OO_K$ of $K$ (respectively $\OO_L$ of $L$).
Since the Brauer group $\mathrm{Br}(K)$ is trivial (\cite{BLR}, 8.1, p.~203), the usual norm map
\begin{equation}\label{eq.Norm}
\xymatrix{N_{L/K}\colon U_L\ar@{->>}[r] & U_K}
\end{equation}
 is surjective.
By using the Greenberg realization functors, one can show that the morphism $N_{L/K}$
 is pro-algebraic in nature. This means that $N_{L/K}$ is the morphism on $k$-rational points induced by a morphism of pro-algebraic groups:
\begin{equation}\label{eq.NormPro}
\xymatrix{\bm N_{L/K}\colon \bm{U_L}\ar@{->>}[r] & \bm{U_K}}.
\end{equation}
On the other hand, the two pro-algebraic groups in \eqref{eq.NormPro} are naturally filtered: for each $n\geq 1$, one can define a pro-algebraic subgroup $\bm{U_{K}^{n}}$ of $\bm{U_K}$ whose group of $k$-rational points is given by the group $U_{K}^{n}$ of $n$-units in $K$:
\[
\bm{U_{K}^{n}}(k)=U_{K}^{n}:=\ker\left(U_K\longrightarrow \left(\mathcal{O}_K/\pi^{n}\mathcal{O}_K\right)^{\times}\right).
\]
Hence $\bm{U_K}$ has the following filtration by pro-algebraic subgroups:
\[
\ldots \subset \bm{U_{K}^{n+1}}\subset \bm{U_{K}^{n}}\subset \ldots \subset \bm{U_{K}^{1}}\subset \bm{U_{K}^{0}}:=\bm{U_K}.
\]
Similarly, $\bm{U_L}$ has the following filtration
\[
\ldots \subset \bm{U_{L}^{n+1}}\subset \bm{U_{L}^{n}}\subset \ldots \subset \bm{U_{L}^{1}}\subset \bm{U_{L}^{0}} :=\bm{U_L}.
\]
In \cite{SerreCFT}, {3.4}, Serre proved that these two filtrations are in fact compatible with respect to the norm map \eqref{eq.NormPro}.
More precisely, for each $n\in \mathbb{Z}_{\geq 0}$, the map $\bm N_{L/K}$ in \eqref{eq.NormPro} sends $\bm{U_{L}^{\psi(n)}}$ onto $\bm{U_{K}^{n}}$, where $\psi=\psi_{L/K}\colon \mathbb{Z}_{\geq 0}\rightarrow \mathbb{Z}_{\geq 0}$ is the \emph{Herbrand function} attached to the extension $L/K$ and used to define the upper numbering of the ramification filtration of the Galois group $\Gamma_{L/K}$.

In the situation of the present article, the two pro-algebraic groups $\bm{\Pic^0(X)}$ and $\bm{J(S)}$ are also naturally filtered: for each $n\geq 1$, we define in \S~\ref{fil} a pro-algebraic subgroup $\bm{\PP^{[n]}(S)}$ of $\bm{\Pic^{0}(X)}$ (respectively $\bm{J^{[n]}(S)}$ of $\bm{J(S)}$) whose group of $k$-rational points is given by
\[
\bm{\PP^{[n]}(S)}(k)\cong \ker\left(\Pic^0(X)\longrightarrow \Pic^0(X_n)\right), \quad \left(\text{resp}.\ \  \bm{J^{[n]}(S)}(k)\cong \ker\left(J(S)\longrightarrow J(S_n)\right)\right),
\]
where $X_n$ is the closed subscheme of $X$ defined by the ideal sheaf $\mathcal{I}^{n}$, and $S_n:=\mathrm{Spec}(\OO_K/\pi^n\OO_K)$. 
The family of pro-algebraic subgroups $\{\bm{\PP^{[n]}(S)}:n\geq 1\}$ (respectively $\{\bm{J^{[n]}(S)}:n\geq 1\}$) gives a decreasing filtration of $\bm{\Pic^0(X)}$ (respectively of $\bm{J(S)}$). In order to compare these two filtrations by means of the morphism $\bm{q}$ in \eqref{eq.ZPicJ}, we define at the beginning of \S~\ref{def de psi} a numerical function $\psi$, whose value at a positive integer $n$ is the smallest integer $m\geq 1$ such that the $\OO_K$-module $\HH^{1}(X_m,\OO_{X_m})$ is of length $n$. The function $\psi$ appears to be an analogue of Serre's Herbrand function mentioned above and was introduced in \cite{Raynaud4}. With this notation, the first main result of this paper can be stated as follows
\begin{theorem*}[See \ref{resultat final} and \ref{cor.finalresult}] The morphism $\bm{q}$ in $\eqref{eq.ZPicJ}$ maps $\bm{\PP^{[\psi(n)]}(S)}$ onto $\bm{J^{[n]}(S)}$.
\end{theorem*}
The proof of the above theorem requires a careful analysis of the length of the torsion part of the group $\HH^{1}(X,\OO_{X})$. This length is completely determined in \ref{lemma.omega} and \ref{prop.expression}.
\medskip

A second goal of this paper is to study the short exact sequence (\ref{eq.ZPicJ}) in the framework of the duality theorems for abelian varieties. By using the short exact sequence (\ref{eq.ZPicJ}), we get, for each torsor $X_K$ of order $d$, an element of the group $\mathrm{Ext}^{1}(\bm{J(S)},\mathbb{Z}/d\mathbb{Z})$ of extensions in the category of Serre pro-algebraic groups.
More generally, considering torsors of order dividing $d$ we can actually define a natural map of \emph{sets} (see \S~\ref{sec.comparison}, first paragraph, for details):
\begin{equation}\label{eq.phid}
\Phi_d\colon {}_d\mathrm{H}^{1}_{\mathrm{fl}}(K, J_K)\longrightarrow \mathrm{Ext}^{1}(\bm{J(S)},\mathbb{Z}/d\mathbb{Z}).
\end{equation}
This construction is analogous to the one used in \cite{SerreCFT} to relate the Galois group $\Gamma_K^\ab$ of the maximal abelian extension of
the field $K$ with the fundamental group of the pro-algebraic group $\bm{U_K}$: namely, let $L/K$ be a finite \emph{abelian} extension with Galois group $\Gamma_{L/K}$, and let $\bm{V_{L}}$ be the kernel of the norm map $\bm N_{L/K}$ in \eqref{eq.NormPro}.
One then has the following short exact sequence of pro-algebraic groups
\begin{eqnarray}\label{eq.Serre}
\xymatrix{0\ar[r] & \bm{V_{L}}\ar[r]& \bm{U_{L}} \ar[r]^{\bm N_{L/K}} & \bm{U_K} \ar[r] & 0}.
\end{eqnarray}
The above sequence provides a homomorphism $\pi_1(\bm{U_K})\to \pi_0( \bm{V_{L}})$ between the fundamental group of the pro-algebraic group $\bm{U_K}$ and the group of connected components of $ \bm{V_{L}}$.
Moreover, there is a canonical isomorphism $\tau\colon \pi_0( \bm{V_{L}})\rightarrow \Gamma_{L/K}$ (cf. \cite{SerreCFT}, 2.3). Now, the push-out of the sequence (\ref{eq.Serre}) via the composition of $\tau$ with the canonical homomorphism $ \bm{V_{L}}\to \pi_0( \bm{V_{L}})$ provides an element of $\mathrm{Ext}^{1}(\bm{U_K},\Gamma_{L/K})$, hence a homomorphism $\pi_1(\bm{U_K})\to\Gamma_{L/K}$, where $\Gamma_{L/K}$ coincides with its component group because it is a finite group.
 By passing to the limit on $L$, Serre obtained a homomorphism $\theta\colon \pi_1(\bm{U_K})\to\Gamma_{K}^\ab$.
There exists then a homomorphism
\begin{equation}\label{thetastar}
\theta^{\ast}\colon \mathrm{H}^{1}(K,\mathbb{Z}/d\mathbb{Z})
{\stackrel{\sim}{\longrightarrow} }
 \mathrm{Hom}(\Gamma_K^\ab,\mathbb{Z}/d\mathbb{Z})\longrightarrow{\mathrm{Hom}(\pi_1(\bm{U_K}),\mathbb{Z}/d\mathbb{Z})}
{\stackrel{\sim}{\longleftarrow} }
 \mathrm{Ext}^{1}(\bm{U_K},\mathbb{Z}/d\mathbb{Z}),
\end{equation}
which is in fact an isomorphism (cf. \cite{SerreCFT}, 4.1). From this fact Serre deduced the main result of \cite{SerreCFT}, namely that $\theta$ is an isomorphism, and thus provided a ``geometric'' characterization of $\Gamma_{K}^\ab$.

Now, our construction of the morphism $\Phi_d$ is an analogue of Serre's construction. Hence it makes sense to ask if the map $\Phi_d$ is an isomorphism too. To answer this question, we then come to the second main result of this paper, which gives a new construction of Shafarevich's pairing using the relative Picard functor. Since this discussion also holds for abelian varieties of higher dimension, let us, more generally, consider an abelian $K$-variety $A_K$, and a torsor $X_K$ under $A_K$ of order $d$. Let $A^{\prime 0}$ be the identity component of the N\'eron model of the dual abelian variety and $\GGr(A^{\prime 0})$ its perfect Greenberg realization (see \S~\ref{sec.proalg-gree}). 
Assume $K$ of mixed characteristic. It is known (cf. \cite{Beg}, 8.2.3) that it is still possible to associate with $X_K$ an extension of the pro-algebraic group $\GGr(A^{\prime 0})$ by $\Z/d\Z$  hence, by push-out, an extension of $\GGr(A^{\prime 0})$ by $\Q/\Z$. Now, using the canonical isomorphisms ${\rm Ext}^1(\GGr(A^{\prime 0}), \Q/\Z) \stackrel{\sim}{\leftarrow}{\rm Ext}^1(\GGr(A^{\prime }), \Q/\Z)
\stackrel{\sim}{\to} {\rm Hom}(\pi_1(\GGr(A^{\prime})), \Q/\Z)$ (cf. \cite{SerrePro}, 5.4),  the above association provides an isomorphism
\begin{equation}\label{eq.sha}
 {\rm H}^1_\fl(K,A_K)\stackrel{\sim}{\longrightarrow} {\rm Hom}(\pi_1(\GGr(A^{\prime})), \Q/\Z) \quad \text{\it (Shafarevich~duality)}
\end{equation}
(cf. \cite{Beg}, 8.3.6). In general, the isomorphism in \eqref{eq.sha} was proved by Shafarevich for the prime-to-$p$ parts (\cite{Sh}, p. 96), and by Bester and the first author, respectively in \cite{Bes}, 7.1, and \cite{Ber}, Theorem 3, for $K$ of equal positive characteristic; there are also related results of Vvedenski\v\i\ on elliptic curves (\cite{V}). In the fourth section of the paper, after recalling the construction of Shafarevich's duality \eqref{eq.sha} in the case of mixed characteristic (see \S \ref{sec.beg}), we slightly modify B\'egueri's construction using rigidificators to make it work in any characteristic. We get in this way a morphism $\Xi\colon {\rm H}^1_\fl(K,A_K)\to{\rm Ext}^1(\GGr(A^{\prime 0}),\Q/\Z)$ (see Proposition~\ref{pro.sha0}). Then we construct a morphism $\Xi'$ via the relative Picard functor (see Theorem~\ref{thm.main}) and we show that $\Xi'$ always coincides with the modified B\'egueri construction $\Xi$ and thus with Shafarevich's duality \eqref{eq.sha} for $K$ of characteristic $0$. In the characteristic $p$ case, the morphism $\Xi'$ coincides with \eqref{eq.sha} on the
prime-to-$p$ parts (see Proposition~\ref{pro.shap}). The analogous result for the $p$-parts, although expected, is still open.

In the fifth section of the paper, using the canonical isomorphism $A^{\prime 0}\stackrel{\sim}{\rightarrow} J$ induced by the one in Lemma~\ref{J_K et A_K} (i), as a direct corollary of the previous study of Shafarevich's pairing, we show that the map $\Phi_d$ in \eqref{eq.phid}, only defined for $A_K$ an elliptic curve, is an injective morphism of groups (see Corollary~\ref{cor.phin}). Furthermore if $d$ is prime to $p$, or with no restriction on $d$ in the mixed characteristic case, $\Phi_d$ can be identified with the restriction of \eqref{eq.sha} to the $d$-torsion subgroups via the isomorphism ${\rm Ext}^1(\GGr(A^{\prime 0}), \Z/d\Z)\stackrel{\sim}{\to} {\rm Hom}(\pi_1(\GGr(A^{\prime})), \Z/d\Z)$. In particular $\Phi_d$ is an isomorphism. This is done by showing that, in the case of elliptic curves, the homomorphism $\pi_1(\GGr(A^{\prime}))\to \Z/d\Z$ corresponding to the short exact sequence (\ref{eq.ZPicJ}) coincides with the homomorphism associated with $X_K$ via the morphism $\Xi'$.

This paper arises from the confluence and the comparison of the results contained in the preprints \cite{Ber2} and \cite{To}. Section \ref{Herbrand} presents detailed proofs of results contained in the unpublished manuscript \cite{Raynaud4}.

\section{The Picard functor and the Greenberg functor}
In this section we recall well-known results on Picard functors and Greenberg functors. Let $S:=\mathrm{Spec}(\OO_{K})$ be the spectrum of a discrete valuation ring $\OO_{K}$, $K$ the fraction field of $\OO_{K}$, and $k$ the residue field. 
Let $\Sch/S$ denote the category of $S$-schemes and $\mathfrak{Ab}$ the category of abelian groups. By a \emph{curve over $S$} we will mean an $S$-scheme $X\rightarrow S$ whose geometric fibres are pure of dimension $1$.

\subsection{The Picard functor}\label{Rappels Pic}
Let $f\colon X\rightarrow S$ be a proper morphism of schemes. Let
\[
\mathrm{Pic}_{X/S}\colon \Sch/S\longrightarrow \mathfrak{Ab}
\]
denote the \emph{relative Picard functor of $X$ over $S$}, \emph{i.e.}, the fppf sheaf associated with the presheaf $S'\mapsto \mathrm{Pic}(X\times_S S')$. It is also the sheaf for the \'etale topology associated with the presheaf $S'\mapsto \Pic(X\times_S S')$ (\cite{Raynaud}, 1.2).

Suppose furthermore that $f$ is proper and flat.
In general the functor $\mathrm{Pic}_{X/S}$ is not representable. It is representable by an algebraic $S$-space if and only if $X/S$ is \emph{cohomologically flat in dimension $0$}, \emph{i.e.}, if the formation of the direct image $f_{\ast}\mathcal{\mathcal{O}}_X$ commutes with base change.
Even when $\mathrm{Pic}_{X/S}$ is not representable, it has a nice presentation by algebraic $S$-spaces. To see this fact, we recall the notion of rigidificator. Given a morphism of schemes $S'\rightarrow S$ and a morphism of $S$-schemes $i\colon Y\rightarrow X$, $i'\colon Y'\rightarrow X'$ will denote the base change of $i$ along $S'\rightarrow S$.

\begin{definition}[\cite{BLR} 8.1/5] \label{rigidificateur} Let
$i\colon Y\hookrightarrow X$ be a closed immersion of $S$-schemes with $Y$ finite and flat over $S$. One says that $(Y,i)$ is a \emph{rigidificator} of
$\mathrm{Pic}_{X/S}$ if the following condition holds: for any $S$-scheme $S'$, the map
 \[
\Gamma(i')\colon\Gamma(X',\mathcal{O}_{X'})\longrightarrow
\Gamma(Y',\mathcal{O}_{Y'})
\]
is injective.
\end{definition}

 In the sequel let $f\colon X\to S$ be proper and flat, and let $(Y,i)$ be a rigidificator of $\mathrm{Pic}_{X/S}$; it exists by the hypothesis on $f$ (\cite{Raynaud}, proposition 2.2.3 (c)).
For any scheme $S'$ over $S$, an \emph{invertible sheaf on $X'$ rigidified along $Y'$}, is a pair $(\mathcal{L}, \alpha)$,
where $\mathcal{\mathcal{L}}$ is an invertible sheaf on $X'$ and $\alpha\colon \mathcal{O}_{Y'}\cong i'^{\ast}\mathcal{L}$ is an
isomorphism (\emph{i.e.}, $\alpha$ is a trivialization of $i'^{\ast}\mathcal{L}$). An isomorphism between two rigidified invertible sheaves $(\mathcal{L},\alpha)$, $(\mathcal{M},\beta)$, on $X'$ is an isomorphism of $\mathcal{O}_{X'}$-modules $u\colon \mathcal{L}\rightarrow \mathcal{M}$ such that the following diagram commutes:
\[
\xymatrix{i'^{\ast}\mathcal{L}\ar[rr]^{i'^{\ast}u} & &
i'^{\ast}\mathcal{M} \\ &
\mathcal{O}_{Y'}\ar[lu]^{\alpha}\ar[ru]_{\beta} & }.
\]
Let $(\mathrm{Pic}_{X/S},Y)(S')$ denote the set of isomorphism classes of invertible sheaves on $X'$ rigidified along $Y'$.
As $S'$ varies in the category of $S$-schemes $(\mathbf{Sch}/S)$, the association $S'\mapsto
(\mathrm{Pic}_{X/S},Y)(S')$ defines a functor of abelian groups $(\mathrm{Pic}_{X/S},Y)$, called \emph{the rigidified Picard functor of $X/S$ relative to the rigidificator $Y$}.
Concerning its representability we have:
\begin{theorem}[\cite{Raynaud}, 2.3.1 \& 2.3.2] Let $f\colon X\rightarrow S$ be a proper flat morphism. Then the functor $(\mathrm{Pic}_{X/S},Y)$ is representable by an algebraic space over $S$, locally of finite presentation. Furthermore, if $dim(X/S)\leq 1$, the algebraic space $(\Pic_{X/S},Y)$ is formally smooth over $S$.
\end{theorem}

One has a canonical morphism of sheaves of groups $r\colon (\mathrm{Pic}_{X/S},Y)\rightarrow \mathrm{Pic}_{X/S}$, that forgets the rigidification. \'Etale locally on $S'$, any element in $\mathrm{Pic}_{X/S}(S')$ is represented by an invertible sheaf on $X'$ such that its pull-back to $Y'$ is trivial (this is possible since $Y$ is finite).
Hence $r$ is an epimorphism for the \'etale topology.
To study the kernel of $r$, let $V_{X}^{\ast}$ (respectively $V_{Y}^{\ast}$) denote the fppf sheaf on $\mathbf{Sch}/S$, given by $S'\mapsto \Gamma(X', \OO_{X'})^{\ast}$ (respectively by $S'\mapsto \Gamma(Y',\OO_{Y'})^{\ast})$.
These sheaves are representable by $S$-schemes (cf. \cite{Raynaud}, 2.4.0).
By definition of rigidificator the natural map
$V_{X}^{\ast}\rightarrow V_{Y}^{\ast}$ is injective.
Let $u\colon V_{Y}^{\ast}\rightarrow (\Pic_{X/S},Y)$ be the map defined as follows on $S'$-sections, $S'$ an $S$-scheme:
\[
 a\in V_{Y}^{\ast}(S')=\Gamma(Y_{S'},\OO_{Y_{S'}}^{\ast})\longmapsto
(\OO_{X_{S'}},\alpha_a)\in (\Pic_{X/S},Y)(S')
\]
where $\alpha_{a}\colon \OO_{Y_{S'}}\rightarrow
\OO_{Y_{S'}}=\OO_{X_{S'}|_{Y_{S'}}}$ is the multiplication by
$a$. Clearly $\mathrm{im}(u)\subset \ker(r)$ and thus one obtains a complex of fppf sheaves over $S$:
\begin{eqnarray}\label{eq.picrig}
0\longrightarrow V_{X}^{\ast}\longrightarrow V_{Y}^{\ast}\stackrel{u}{\longrightarrow}
\left(\mathrm{Pic}_{X/S},Y\right)\stackrel{r}{\longrightarrow}
\mathrm{Pic}_{X/S}\longrightarrow 0,
\end{eqnarray}
which is exact for the \emph{\'etale} topology (\cite{Raynaud}, 2.1.2(b),
2.4.1).

We assume for the remainder of the section that the discrete valuation ring $\mathcal{O}_{K}$ is strictly henselian \emph{with algebraically closed residue field}. In particular, the Brauer group $\mathrm{Br}(K)$ is trivial.
 Let $f\colon X\rightarrow S$ be a proper and flat curve with geometrically connected fibres.
Denote by $\mathrm{P}$ (respectively by $(\mathrm{P},Y)$) the open subfunctor of $\mathrm{Pic}_{X/S}$ (respectively of $(\mathrm{Pic}_{X/S},Y)$) consisting of invertible sheaves of total degree $0$ (respectively of invertible sheaves of total degree $0$ rigidified along $Y$).
Then $(\mathrm{P},Y)$ is an open algebraic subspace of $(\mathrm{Pic}_{X/S},Y)$, and $\mathrm{P}$ is the schematic closure of $\left(\mathrm{Pic}_{X/S}\right)_{K}^{ 0}$ in $\mathrm{Pic}_{X/S}$, while $(\mathrm{P},Y)$ is the schematic closure of $(\mathrm{Pic}_{X/S},Y)_{K}^{0}$ in $ (\mathrm{Pic}_{X/S},Y)$.
Denote by $E$ the schematic closure of the unit section of $\mathrm{P}_{K}$ in $\mathrm{P}$, and define $\mathcal{Q}$ as the fppf quotient of $\mathrm{P}$ by $E$. It is the biggest separated quotient of $\mathrm{P}$.
It is representable by a separated group scheme over $S$ (\cite{Raynaud}, 3.3.1). Denote by $q$ the canonical map
\begin{equation}\label{eq.defdeq}
q\colon \PP\longrightarrow \mathcal{Q};
\end{equation}
it is surjective for the fppf topology, and it induces an isomorphism on generic fibres since $E_K$ is the unit section of $\PP_K = \Pic_{X_K/K}^{0}$.

\begin{theorem}[\cite{LLR}, 3.7]\label{Neron} Let $f\colon X\to S$ be a proper and flat curve with $X$ regular and $f_{\ast}\mathcal{O}_{X}=\mathcal{O}_S$. Then the group scheme $\mathcal{Q}/S$ is the N\'eron model of
$\mathrm{P}_{K}=\mathrm{Pic}^{ 0}_{X_{K}/K}$.
\end{theorem}

Suppose then $X$ regular. Denote by $J:=\mathcal{Q}^{0}$ the identity component of $\QQ$, and by $\Pic^{0}_{X/S}$ the subsheaf of $\PP\subset \Pic_{X/S}$ of invertible sheaves of degree $0$ on each irreducible component of $X$ (for the definition of $F^{0}$ when $F$ is a sheaf on $S$, see \cite{Raynaud}, 3.2 (d)).
 Since the functor $\Pic_{X/S}^{0}$ has connected fibres, the restriction to $\Pic_{X/S}^{0}$ of $q$ in \eqref{eq.defdeq} factors through the identity component $J$ of $\mathcal{Q}$.
By abuse of notation, the same letter $q$ will denote the induced map:
\begin{equation}\label{eq.SurPiczero}
q\colon \Pic_{X/S}^{0}\longrightarrow J .
\end{equation}
Finally, since $\mathrm{Br}(K)=0$, the following map
\begin{equation*}
\xymatrix{(\PP,Y)^{0}(S)=(\Pic_{X/S},Y)^{0}(S)\ar@{->>}[r] & J(S)},
\end{equation*}
induced by the morphism $r$ in \eqref{eq.picrig} is surjective (see the proof of 9.6/1 of \cite{BLR}); hence so too is the morphism (again denoted by $q$) induced by \eqref{eq.SurPiczero}:
\begin{equation}\label{eq.qP}
q=q(S)\colon \xymatrix {\PP^{0}(S)=\Pic^{0}_{X/S}(S)\ar@{->>}[r]& J(S)}.
\end{equation}

\subsection{Pro-algebraic groups and Greenberg functor}\label{sec.proalg-gree}
From now on, we suppose that the discrete valuation ring $\mathcal{O}_{K}$ is complete with algebraically closed residue field $k$.
In the following, \emph{a pro-algebraic group over $k$} is a pro-object in the category of $k$-group schemes of finite type (see \cite{Oo}, I.4). This notion does not coincide with the notion of pro-algebraic groups introduced by Serre in \cite{SerrePro} (\S~2.1 D\'efinition~1), where the author considers the category of $k$-group schemes, but up to purely inseparable morphisms. Since we use both categories, we call the objects of the latter abelian category (see \cite{SerrePro}, \S 2.4, Proposition~7) Serre pro-algebraic groups and denote them with bold letters.

Let $G$ be a smooth group scheme of finite type over $S$.
The Greenberg functors allow us to construct a pro-algebraic group over $k$, associated with $G$, whose group of $k$-points is $G(S)$.
Let us recall the construction.
Denote by $W$ the ring of Witt vectors of $k$ and by $\mathbb{W}$ the Witt functor on the category of $k$-algebras $\mathbf{Alg}/k$.
Let $n\in {\mathbb Z}_{\geq 1}$ and $\OO_{K,n}:=\OO_{K}/\pi^{n}\OO_K$. Then $\OO_{K,n}$ is canonically a $W$-module of finite length.
Let $\mathbb{R}_n$ be its associated Greenberg algebra (see \cite{Lipman}, Appendix~A), which is by definition the fpqc sheaf on $\mathbf{Alg}/k$ associated with the pre-sheaf: ${\mathsf A}\mapsto \mathcal{O}_{K,n}\otimes_W \mathbb{W}({\mathsf A})$.
One defines then $\mathrm{Gr}_n(G)$ as the sheaf on $\mathbf{Alg}/k$ given by ${\mathsf A}\mapsto G(\mathbb{R}_{n}({\mathsf A}))$.
It is representable by a smooth $k$-group scheme of finite type (\cite{Gree}, Proposition~7, and \cite{Gree2}, Corollary~1).
For any $n\geq 2$, the canonical map $\OO_{K,n}\rightarrow \OO_{K,n-1}$ induces a smooth morphism of $k$-group schemes $\alpha_{n}\colon \Gr_{n}(G)\rightarrow \Gr_{n-1}(G)$, whose kernel is a connected unipotent $k$-group scheme.
 Furthermore the canonical map $G(\OO_{K,n})\rightarrow \mathrm{Gr}_n(G)(k)$ is an isomorphism. Thanks to this identification the morphism $\alpha_n(k)\colon \Gr_{n}(G)(k)\rightarrow \Gr_{n-1}(G)(k)$ is identified with the canonical morphism $G(\OO_{K,n})\rightarrow G(\OO_{K,n-1})$.
The algebraic $k$-groups $\Gr_n(G)$ form then a projective system $\{(\Gr_n(G),\alpha_n)\}_{n\geq 1}$ of algebraic $k$-groups with smooth
transition maps having connected kernels.
We will denote the latter projective system by $\Gr(G)$ and call it the \emph{Greenberg realization} of $G$.

{As} explained in \cite{Mi}, III \S~4, the perfect group schemes $\GGr_n(G)$ associated with the $k$-group schemes $\Gr_n(G)$ are quasi-algebraic groups in the sense of \cite{SerrePro}, 1.2, and hence the projective system $\GGr(G):= \{(\GGr_n(G),\alpha_n)\}_{n\geq 1}$ determines a
pro-algebraic group in the sense of Serre; we will call it the \emph{perfect Greenberg realization} of $G$. The group of $k$-rational points of $\GGr(G)$ is $G(S)$. For this reason sometimes $\GGr(G)$ is denoted by $\bm{G(S)}$.

Observe that $\GGr(G)$ is defined for any smooth $S$-scheme, but it may not be a Serre pro-algebraic group, since its component group may not be profinite: for example, consider as $G$ the lft N\'eron model of $\G_{m,K}$ (\cite{BLR}, 10.1/5); then $\GGr(G)$ is the projective limit of perfect $k$-schemes having component group $\Z$. 
The functor $\GGr$ is exact on smooth $S$-group schemes (see \cite{Bes}, Lemma 1.1).

{Recall that the identity component $\bm G^0$ of a Serre pro-algebraic group $\bm G$ is defined as the smallest closed subgroup $\bm G'$ of $\bm G$ such that $\bm G/\bm G'$ has dimension zero (\cite{SerrePro}, 5.1/1)}. In the category of Serre pro-algebraic groups, the component group functor $\pi_0$, which maps $\bm G$ to $\bm G/\bm G^0$, admits a left derived functor $\pi_1$ which is left exact.
We list below some well-known facts used in this paper. By simply assuming them, the reader unfamiliar with the theory of Serre pro-algebraic groups will be able to follow the proofs.

\begin{itemize}
\item[(i)] If $Y$ is a smooth $S$-group scheme of finite type, then $\GGr(Y^0)\cong \GGr(Y)^0$, $\pi_0(\GGr(Y))$ is isomorphic to the component group of the special fibre $Y_s$, and $\pi_1(\GGr(Y))$ is a profinite group.
\item[(ii)] If $\bm P$ is a Serre pro-algebraic group and $\bm P^0$ is its identity component, then $\pi_1(\bm P)\cong \pi_1(\bm P^0)$.
\item [(iii)] A complex of smooth $S$-group schemes of finite type $0\to Y_1\to Y_2\to Y_3\to 0$ which is exact on $S$-sections provides a long exact
sequence of profinite groups (cf. \cite{SerrePro}, 10.2/1)
\begin{multline*}0\to \pi_1(\GGr( Y_1))\to \pi_1(\GGr(Y_2))\to \pi_1(\GGr( Y_3))\to \\
\to \pi_0(\GGr (Y_1))\to \pi_0(\GGr (Y_2))\to \pi_0(\GGr( Y_3))\to 0.
\end{multline*}
\end{itemize}

Let $f\colon Z\rightarrow S$ be a proper morphism of schemes such that $Z$ is of dimension $1$ with $f(Z)=\{s\}$. Note that the morphism $f$ then factors through some $S_n:=\spec(\OO_{K,n})\hookrightarrow S$. For instance, when $X/S$ is a proper flat curve, one can take $Z=X\times_S S_n$ viewed as an $S$-scheme. For $S'$ an $S$-scheme, let $\Pic^0(Z \times_{S}S')$ denote the subgroup of $\Pic(Z \times_{S}S')$ consisting of line bundles $\mathcal{L}$ on $Z\times_S S'$ such that, for any point $s'\in S'$ above $s\in S$, the restriction $\mathcal{L}|_{Z\times_{S}\{s'\}}$ is of degree $0$ on each irreducible component of $Z\times_{S} \{s'\}$. Then we define $\Pic_{Z/S}^{0}$ to be the associated fppf sheaf on $S$ of the following functor
\[
\Sch/S\longrightarrow \mathfrak{Ab},\qquad S'\mapsto \Pic^0(Z\times_S S').
\]
As the map $f$ is never flat, the Picard functors $\Pic_{Z/S}$ and $\Pic^{0}_{Z/S}$ are not representable. However, as shown by Lipman in \cite{Lipman}, the Greenberg realization of the sheaf $\Pic_{Z/S}$ (respectively of $\Pic^{0}_{Z/S}$) is represented by a smooth $k$-group scheme. More precisely, assuming that the morphism $f$ factors through $i_n\colon S_n=\mathrm{Spec}(\OO_{K,n})\hookrightarrow S$, we define $\Gr(\Pic_{Z/S})$ (respectively $\Gr(\Pic^{ 0}_{Z/S})$) to be the fppf sheaf associated with the presheaf (\cite{Lipman}, 1.8)
\[
\mathbf{Alg}/k\longrightarrow \mathfrak{Ab}, \qquad {\mathsf A} \mapsto
\Pic_{Z/S}(\mathbb{R}_n(\mathsf A)) \quad \left(\text{resp.} \quad {\mathsf A}\mapsto
\Pic^{0}_{Z/S}(\mathbb{R}_n(\mathsf A))\right)
\]
where, as above, $\mathbb{R}_n$ denotes the Greenberg algebra associated with $\mathcal{O}_{K,n}$.
 Note that the definition of $\mathrm{Gr}(\Pic_{Z/S})$ does not depend on the choice of the integer $n$ since $\mathrm{Gr}(\Pic_{Z/S})$ is isomorphic to the fpqc associated associated with the presheaf ${\mathsf A}\mapsto \Pic(Z\otimes_W \mathbb{W}(\mathsf A))$ (\cite{Lipman}~1.1 \& 1.8); the same proof shows that $\Gr(\Pic^{0}_{Z/S})$) is also independent of $n$.

\begin{theorem}[\cite{Lipman}]\label{Lipman} Let $f\colon Z\rightarrow S$ be as above. Then the functor $\Gr(\Pic_{Z/S})$
(respectively $\Gr(\Pic^{ 0}_{Z/S})$) is representable by a smooth $k$-group scheme (respectively by a smooth and connected $k$-group scheme),
whose dimension is equal to the length of the $W$-module $\HH^{1}(Z,\OO_{Z})$. Furthermore, the canonical morphisms
\[
\Pic(Z)\longrightarrow \Gr(\Pic_{Z/S})(k),\quad \quad
\Pic^{ 0}(Z)\longrightarrow \Gr(\Pic^{ 0}_{Z/S})(k)
\]
are isomorphisms.
\end{theorem}

\begin{proof} The statements for $\Gr(\Pic_{Z/S})$ are obtained by combining Theorem 1.2, (1.8), Corollary~8.6~(a) and Theorem~9.1 of \cite{Lipman}. To deduce the corresponding statements for $\Gr(\Pic^{0}_{Z/S})$, let $Z':=Z_{\mathrm{red}}$ be the maximal reduced closed subscheme of $Z$. The canonical map $Z'\hookrightarrow Z$ induces a morphism between the Picard functors $\Pic_{Z/S}\rightarrow \Pic_{Z'/S}$, hence also a morphism $u\colon \Gr(\Pic_{Z/S})\rightarrow \Gr(\Pic_{Z'/S})$ between the Greenberg realizations of Picard functors. By Proposition~\ref{pro.GrePicKer} $u$ is an epimorphism of smooth $k$-group schemes with $\ker(u)$ a smooth connected unipotent group. On the other hand, by definition, the canonical morphism of sheaves
\begin{equation*}\label{eq.pic3}
\Pic^{0}_{Z/S}\longrightarrow \Pic_{Z/S}\times_{\Pic_{Z'/S}}\Pic^{0}_{Z'/S}
\end{equation*}
is an isomorphism, from whence we deduce a similar isomorphism between the Greenberg realizations:
\begin{equation}\label{eq.cartforGr}
\Gr(\Pic^{0}_{Z/S}) \stackrel{\sim}{\longrightarrow}\Gr(\Pic_{Z/S})\times_{\Gr(\Pic_{Z'/S})}\Gr(\Pic^{0}_{Z'/S}).
\end{equation}
Consequently, the canonical morphism $u^0\colon \Gr(\Pic^{0}_{Z/S})\rightarrow \Gr(\Pic^{0}_{Z'/S})$ is an epimorphism of fppf sheaves on $\mathrm{Spec}(k)$, with kernel isomorphic to $\ker(u)$. 
Now, since $Z'$ is reduced, the canonically map $Z'\to S$ factors through $i_1\colon \mathrm{Spec}(k)\hookrightarrow S$, and $\Gr(\Pic_{Z'/S}^{0})\cong \Pic^{0}_{Z'/k}$ (cf. \cite{Lipman}, p.~29, last paragraph) is representable by a smooth connected $k$-group scheme. Therefore, the sheaf $\Gr(\Pic^{0}_{Z/S})$ is representable and, furthermore, it is isomorphic to $\mathrm{Gr}(\Pic_{Z/S})^{0}$ because $u^0$ is an epimorphism, with $\ker(u^0)\cong \ker(u)$ a connected affine $k$-group scheme. Finally, consider the following commutative diagram
\[
\xymatrix{\Pic^0(Z)\ar[r]\ar[d] & \Gr(\Pic_{Z/S}^{0})(k)\ar[d] \\ \Pic(Z)\times_{\Pic(Z')}\Pic^0(Z')\ar[r] & \Gr(\Pic_{Z/S})(k)\times_{\Gr(\Pic_{Z'/S})(k)}\Gr(\Pic^{0}_{Z'/S})(k) }.
\]
By the definition of $\Pic^0$ and \eqref{eq.cartforGr}, both vertical morphisms are bijective. Moreover, by what we have shown at the beginning of the proof and the fact that $\Gr(\Pic_{Z'/S}^{0})\cong \Pic_{Z'/k}^{0}$, the lower horizontal morphism is also bijective. Therefore, the upper horizontal morphism is bijective. This finishes the proof.
\end{proof}

\begin{proposition}\label{pro.GrePicKer} Let $\iota\colon Z'\hookrightarrow Z$ be a closed subscheme defined by a nilpotent ideal $\mathcal{J}\subset \OO_{Z}$. Then the canonical morphism of smooth $k$-group schemes
\[
u\colon \Gr(\Pic_{Z/S})\longrightarrow \Gr(\Pic_{Z'/S})
\]
is an epimorphism. Moreover, $\ker(u)$ is a smooth connected unipotent group.
\end{proposition}
\begin{proof} By Theorem~\ref{Lipman}, $u$ is represented by a morphism of smooth $k$-group schemes. Then $u$ is an epimorphism as soon as it is surjective on $k$-rational points. By the previous theorem, we are reduced to showing that the canonical map $\Pic(Z)\rightarrow \Pic(Z')$ is surjective. 
Consider then the following short exact sequence of abelian sheaves
\[
0\longrightarrow 1+\mathcal{J}\longrightarrow \OO_{Z}^{\ast}\longrightarrow \iota_{\ast}\OO_{Z'}^{\ast}\longrightarrow 0.
\]
It provides an exact sequence of cohomology groups
\[
\HH^{1}(Z,1+\mathcal{J})\longrightarrow \Pic(Z)  \longrightarrow \Pic(Z') \longrightarrow \HH^{2}(Z,1+\mathcal{J}).
\]
Since $\mathrm{dim}(Z)=1$, the $\HH^2$ group on the right vanishes and hence the map $\Pic(Z)\rightarrow \Pic(Z')$ is surjective.

By the proof of Proposition 2.5 in \cite{Lipman} (in particular the last paragraph of page 31), $\ker(u)$ is representable by a connected unipotent group scheme over $k$. It remains to show that $\ker(u)$ is smooth. 
Let $\mathcal{N}\subset \OO_{Z}$ be the nilpotent radical of $\OO_Z$, and $Z_n\hookrightarrow Z$ the closed subscheme defined by $\mathcal{J}\mathcal{N}^{n}\subset \OO_{Z}$ ($n\geq 0$).
Then we have $Z_0=Z'$ and $Z_m=Z$ for $m$ an integer sufficiently large. In this way, we obtain an ascending sequence of closed subschemes $\{Z_n\}_{n\geq 0}$ of $Z$, and hence a sequence of surjective morphisms of smooth $k$-group schemes $u_n\colon \Gr(\Pic_{Z_{n+1}/S})\rightarrow \Gr(\Pic_{Z_{n}/S})$ ($n\geq 0$). 
Therefore, in order to prove that $\ker(u)$ is smooth, we only need to show that $\ker(u_n)$ is smooth for all $n$. Hence, up to replacing $Z'\hookrightarrow Z$ with $Z_{n}\hookrightarrow Z_{n+1}$, we may assume furthermore that $\mathcal{J} \mathcal{N}=0$.
 Now, the proof of 6.4 in \cite{Lipman} shows that $\ker(u)$ is isomorphic to the kernel of the canonical map $v\colon \mathbb{H}\rightarrow \mathbb{H}'$, where $\mathbb{H}$ (respectively $\mathbb{H}'$) denotes the associated fppf sheaf of the functor
\[
\mathbf{Alg}/k\longrightarrow \mathfrak{Ab}, \qquad {\mathsf A} \mapsto
\HH^1(Z_{\mathsf A}, \OO_{Z_{\mathsf{A}}}), \quad (\text{respectively} \quad{\mathsf A} \mapsto
\HH^1(Z_{\mathsf A}', \OO_{Z_{\mathsf{A}}'}) )
\]
with $Z_{\mathsf A}:=Z\otimes_{\OO_K}\mathbb{R}_n(\mathsf{A})$ and $Z_{\mathsf{A}}':=Z'\otimes_{\OO_K}\mathbb{R}_n(\mathsf{A})$. Both $\mathbb{H}$ and $\mathbb{H}'$ are representable by smooth $k$-group schemes (\cite{Lipman}, Theorem~1.4 and Corollary~8.4), and $v$ is a surjective morphism of $k$-group schemes (\cite{Lipman}, Corollary~4.4). To finish the proof, it remains to show that the morphism $v$ is smooth, or equivalently, that the induced map between the Lie algebras $\mathrm{Lie}(v)\colon \mathrm{Lie}(\mathbb{H})\to \mathrm{Lie}(\mathbb{H}')$ is surjective (\cite{LLR} Prop.~1.1 (e)).
 By \cite{Lipman}, Theorem~8.1, $\mathrm{Lie}(\mathbb{H})$ has a natural grading by $k$-vector subspaces:
\[
\mathrm{Lie}(\mathbb{H})=\bigoplus_{i\geq 0}\mathrm{Lie}^{p^i}(\mathbb{H}), \qquad \text{with} \quad \mathrm{Lie}^{p^i}(\mathbb{H})\stackrel{\sim}{\longrightarrow}\mathrm{im}(\mu_i)^{(-i)},
\]
where $\mu_i=\mu_{i,Z}$ is the canonical map $\HH^1(Z,p^i\OO_Z/p^{i+1}\OO_Z) \rightarrow \HH^1(Z,\OO_Z/p^{i+1}\OO_Z)$, and for any $k$-vector space $V$, $V^{(-i)}$ denotes the $k$-vector space with the same underlying abelian group as $V$ and the scalar multiplication $(a,x)\mapsto a^{p^{-i}}x$ for any $a\in k, x\in V$. 
Since the grading is functorial (\cite{Lipman}, p.~77), in order to showing that $\mathrm{Lie}(v)$ is surjective, we have to prove that the map $\mathrm{im}(\mu_{i,Z})\rightarrow \mathrm{im}(\mu_{i,Z'})$ is surjective. Now, since $\mathrm{dim(Z)}=1$, the surjective morphism of coherent sheaves $p^{i}\OO_Z/p^{i+1}\OO_{Z} \rightarrow \iota_{\ast} \left(p^i\OO_{Z'}/p^{i+1}\OO_{Z'}\right)$ induces a surjective map
\[
\HH^1(Z,p^i\OO_{Z}/p^{i+1}\OO_{Z})\longrightarrow \HH^1(Z',p^i\OO_{Z'}/p^{i+1}\OO_{Z'}), 
\]
and this implies the surjectivity of the map $\mathrm{im}(\mu_{i,Z})\rightarrow \mathrm{im}(\mu_{i,Z})$.
\end{proof}

\section{Herbrand functions}\label{Herbrand}
Unless explicitly mentioned, the results of this section are directly taken from the unpublished work \cite{Raynaud4} of Raynaud. Some of them have already appeared in \cite{KatUen}.
In this section and the next, let $K$ be a complete discrete valued field with algebraically closed residue field $k$ and ring of integers $\mathcal{O}_{K}$.
 Let $A_K$ be an elliptic curve, $X_K$ a torsor under $A_K$ over $K$, and $f\colon X\rightarrow S$ the proper regular minimal model of $X_K/K$ over $S$. Let $X_s=\sum_{i=1}^{r}n_iC_i$ be the decomposition of the special fibre $X_s$ into
the sum of its reduced irreducible components, and denote by $d$ the gcd of the integers $n_i$.
Moreover, let $D$ be the divisor of $X$ given by
\begin{equation}\label{eq.d}
D:=\frac{1}{d}X_s=\sum_{i=1}^{r}\frac{n_i}{d} C_i,
\end{equation}
and let $\mathcal{I}$ be the (locally principal) ideal sheaf of $D$. The special fibre $X_s$ of $X$ is then defined by the ideal sheaf $\mathcal{I}^{d}=\pi\OO_{X}\subset\OO_X$. For each $n\in {\mathbb Z}_{\geq 1}$, let ${i_n\colon}X_n\hookrightarrow X$ be the closed immersion defined by the ideal $\mathcal{I}^{n}\subset \OO_X$.
For $\mathcal{L}$ an invertible sheaf (respectively for $\Sigma$ a divisor) on $X$ and $Z$ a divisor on $X$, let $\mathcal{L}\cdot Z$ (respectively $\Sigma\cdot Z$) denote the \emph{intersection number} of $\mathcal{L}$ and $Z$ (respectively of $\Sigma$ and $Z$). Finally, let $\phi(n)$ be the length of the $\OO_K$-module $\HH^1(X_n,\OO_{X_n})$.
In this section, we will first study the function $n\mapsto \phi(n)$ by means of some numerical invariants attached to $X/S$ (Lemma~\ref{ki}). Then we will define from $\phi$ two piecewise linear functions $\varphi,\psi$ ({see \S~\ref{def de psi} and the figure therein}), which are the analogue of the Herbrand functions in Serre's description of local class field theory (\cite{SerreCFT}, \S~3) and will be used in the next section.

\setcounter{lemma}{0}
We recall the following useful result.
\begin{lemma}[\cite{Mumford}, p.~332]\label{lemme cohomologique} Let $Y\rightarrow S$ be a proper flat curve such that $Y$ is regular and such that $Y_K/K$ is geometrically irreducible. Let $Y_s=\sum_{i=1}^{r}n_iE_i$ be the decomposition of the special fibre $Y_s$ into the sum of its reduced irreducible components, and $n$ the gcd of the integers $n_i$. Assume moreover that $\omega_{Y/S}\cdot E_i=0$ for all $i$, where $\omega_{Y/S}$ denotes the dualizing sheaf of $Y/S$. Let finally $Y_1$ be the closed subscheme of $Y$ defined by the ideal sheaf $\OO_Y(-\sum_{i=1}^{r}\frac{n_i}{n}E_i)$ and $L$ an invertible sheaf over $Y_1$,
 which is of degree $0$ on $E_i$ for all $i$. Then, if $\HH^{0}(Y_1,L)\neq 0$, we have $L\cong \OO_{Y_1}$ and $\HH^{0}(Y_1,\OO_{Y_1})\cong k$.
\end{lemma}

The curve $Y_1$ on $Y$ is indecomposable of canonical type in the sense of \cite{Mumford}, p.~330. The lemma in \cite{Mumford} (p.~332) is stated for an indecomposable curve of canonical type on a smooth proper algebraic surface defined over an algebraically closed field. However, one can check that the proof there works in our setting too.

\subsection{Study of the dualizing sheaf}\label{etude de dualisant}
We begin with two classical results. Recall that, given a morphism of $K$-group schemes $\alpha\colon G\rightarrow H$ and a morphism of $K$-schemes $f\colon Y\rightarrow Z$ , with $G$ (respectively $H$) acting on $Y$ (respectively $Z$), the morphism $f$ is said to be equivariant with respect to $\alpha$ if $f(g\cdot y)=\alpha(g)\cdot f(y)$ for any $g\in G$ and any $y\in Y$.

\begin{lemma}\label{J_K et A_K} Let $A_K$ be an elliptic curve over $K$, and $X_K/K$ a torsor under $A_K$. Let $J_K:=\Pic_{X_K/K}^{0}$ be the jacobian of $X_K/K$.
\begin{itemize}
\item[(i)] There is a canonical isomorphism of elliptic curves $\iota\colon A_{K}'\xrightarrow{\sim} J_K$, where $A_{K}'$ denotes the dual of $A_K$.
\item[(ii)] Let $\sigma \colon A_K\rightarrow J_K$ be the isomorphism obtained by composing the isomorphism $\alpha\colon A_K\xrightarrow{\sim} A_{K}'$,
sending $a\in A_K(\overline{K})$ to $\mathcal{O}_{A_{\overline{K}}}(a-o)\in A_{K}'(\overline{K})$ (here $o\in A_K{(\bar K)}$ is the neutral element), with the isomorphism $\iota$ in (i). Then the
canonical map $\psi\colon X_K\rightarrow \mathrm{Pic}^{1}_{X_K/K}$ mapping $x$ to $\OO_{X_K}(x)$ is equivariant with respect to the
isomorphism $\sigma$, where $\Pic_{X_K/K}^{1}$ is the Picard scheme which classifies the line bundles of degree $1$ on $X_K$.
In particular, under the identification given by $\sigma$, the (class of the) torsor $X_K$ in $\HH^{1}_{\fl}(K,A_K)$ corresponds to the (class of the) torsor $\Pic_{X_K/K}^{1}$ in $\HH^{1}_{\fl}(K,J_K)$.
\end{itemize}
\end{lemma}
\begin{proof}
The first fact is proved in \cite{Raynaud3}, XIII, 1.1.
For the second fact, in order to verify that the morphism $\psi$ is equivariant with respect to the morphism $\sigma$, by descent, we need only prove the corresponding statement over a separable closure $\bar{K}$ of $K$.
Over $\bar{K}$, the isomorphism $\sigma\colon A_{\bar{K}}\rightarrow J_{\bar{K}}$ can be explicitly described by mapping $a\in A_{\bar{K}}$ to the class of $\OO_{X_{\bar{K}}}(a\cdot x_0-x_0)$ with $x_0\in X_K(\bar{K})$. Hence to show the desired property of (ii), we are reduced to proving the following isomorphism between line bundles on $X_{\bar{K}}$: for any $x\in X_{\bar{K}}(\bar{K})$, and any
$a\in A_{\bar{K}}(\bar{K})$ we have $\OO_{X_{\bar{K}}}(a\cdot x)\cong \OO_{X_{\bar{K}}}(a\cdot x_0-x_0)\otimes \OO_{X_{\bar{K}}}(x)$,
or equivalently, $\OO_{X_{\bar{K}}}(a\cdot x_0-x_0)\cong \OO_{X_{\bar{K}}}(a\cdot x-x)$. Indeed, since the
isomorphism $\iota_{\bar{K}}$ is independent of the choice of $x_0$, we have
$\sigma(a)=\iota_{\bar{K},x_0}(\alpha(a))=[\OO_{X_{\bar{K}}}(a\cdot x_0-x_0)]=\iota_{\bar{K},x}(\alpha(a))=[\OO_{X_{\bar{K}}}(a\cdot x-x)]$.
\end{proof}

\begin{remark}\label{rem.J_K et A_K toute dim} In \cite{Raynaud3}, XIII, 1.1, the conclusion in \ref{J_K et A_K} (i) is proved for any abelian variety $A_K$.
\end{remark}

\begin{lemma}\label{d1=d2=d} Let $d_1$ be the order of (the class of) the torsor $X_{K}$ in the group $\HH^{1}_\fl(K,J_{K})$, $d_2$ the minimal degree of extensions $K'$ of $K$ such that $X_K(K')\neq \emptyset$, and $d_3$ the minimum of the multiplicities of the irreducible components of $X_s$. Then $d_1=d_2=d_3=d$, with $d$ as in \eqref{eq.d}.
\end{lemma}

\begin{proof} We will prove this Lemma by showing that $d\leq d_1\leq d_2\leq d_3\leq d$.
Let $n\in {\mathbb Z}_{>0}$. By Lemma~\ref{J_K et A_K} (ii), the class $n[X_K]\in \HH^1(K,A_K) \stackrel{\sim}{\to} \HH^1(K,J_K)$ can be represented by the
irreducible component $\Pic^{n}_{X/S}$ of $\Pic_{X_K/K}$ which classifies the invertible sheaves of degree $n$.
Hence, the class $n[X_K]\in \HH^1(K,A_K)$ is trivial if and only if $\Pic^{n}_{X_K/K}(K)\neq \emptyset$. On the other hand, since $\OO_K$ is
strictly henselian with algebraically closed residue field, we have $\mathrm{Br}(K)=0$.
Hence, $\Pic_{X_K/K}^n(K)=\Pic^n(X_K)$. As a result, $d_1$ is also the minimum of the degrees of the divisors with positive degree on $X_K$.
 Now, let $\Sigma_K\subset X_K$ be any divisor with positive degree, and let $\Sigma$ be its schematic closure in $X$.
Then we have $\mathrm{deg}(\Sigma_K)=\Sigma\cdot X_s=\Sigma\cdot(dD)=d(\Sigma\cdot D)$. Hence $d|\mathrm{deg}(\Sigma_K)$,
in particular $d\leq \mathrm{deg}(\Sigma_K)$.
 As a result, we get $d\leq d_1$.
Next, by definition of $d_2$, there exists a closed point of degree $d_2$ on $X_K$, hence a divisor of degree $d_2$ on $X_K$,
so we have $d_1\leq d_2$.
Since $\OO_K$ is strictly henselian, for each $i$, we can find a \emph{positive} divisor $\Delta_i$ of $X/S$ of degree $n_i$ (\cite{BLR} 9.1/10).
In particular, we have $d_2\leq n_i$ for each $i$, hence we get $d_2\leq d_3$.
Finally, to see that $d_3\leq d$, note that a
suitable combination of $\Delta_i$ gives us a divisor $\Delta'$ of degree $d$ of $X_K$.
In general, the divisor $\Delta'$ might not be positive, but since $X_K$ is of arithmetic genus $1$, and $d\geq 1$,
we have $h^{0}(X_K,\OO_{X_{K}}(\Delta_K'))>0$.
Hence there exists a \emph{positive} divisor $\Delta_K$ of degree $d$ of $X_K$ which is linearly equivalent to $\Delta_K'$.
 Let $\Delta:=\overline{\Delta_K}$ be the schematic closure of $\Delta_K$ in $X$. Then we have
\[
d=\mathrm{deg}(\Delta_K)=\Delta\cdot X_s=d(\Delta\cdot D).
\]
In particular, $\Delta\cdot D$=1, and $\Delta\cap D=\{y\}$ consists of only one point, and $D$ is \emph{regular} at $y$.
Let $C_i$ be the irreducible component of $D$ passing through $y$, then $C_i$ is of multiplicity $d$ in $X_s$, whence $d_3\leq d$.
This completes the proof.
\end{proof}

Since $S$ is an affine Dedekind scheme and $X$ is regular, by a result of Lichtenbaum (\cite{Liu}~8.3.16), $X$ is \emph{projective} over $S$. 
Let $\omega_{X/S}$ denote the canonical (invertible) sheaf of $f\colon X\to S$ (see \cite{Liu}, 6.4.7, for the definition and \cite{Liu}, 6.4.32, for the fact that it is isomorphic to the dualizing sheaf). 
Since $X_K/K$ is smooth projective curve of genus $1$ and the formation of the canonical sheaf is compatible with flat base change (\cite{Liu}, 6.4.9(b)), one finds that $(\omega_{X/S})_{|X_K}\cong \omega_{X_K/K}\cong \OO_{X_K}$ (\cite{Liu}, Example 7.3.35).
For any $n\geq 0$, let 
\[
\omega_n:=i_n^*(\omega_{X/S}(X_n))=i_n^*(\omega_{X/S}\otimes_{\OO_X} \OO_X(X_n)).
\]
Let $f_n$ denote the composite of ${i_n\colon}X_n\hookrightarrow X$ and $f{\colon X\to S}$. Then we have
\begin{equation}\label{omega}\omega_n= i_n^*(\omega_{X/S}\otimes_{\OO_X} \OO_X(X_n))=i_n^*(\omega_{X/S}\otimes_{\OO_X} \mathcal I^{n\vee})\cong i_n^*\omega_{X/S}\otimes_{\OO_{X_n}}(\mathcal I^n/ \mathcal I^{2n}) ^\vee, \end{equation} 
where $\mathcal{F}^{\vee}:=\mathcal{H}om_{X}(\mathcal{F},\OO_{X})$ for any coherent sheaf $\mathcal F$ on $X$. 
The following lemma is well-known.

\begin{lemma}\label{lem.rh} 
\begin{enumerate}
\item[(i)]
We have $f^{!}\OO_S\cong \omega_{X/S}[1]$ and $f_{n}^{!}\OO_{S}\cong \omega_{n}$ respectively in the derived category of quasi-coherent sheaves on $X$ and on $X_n$.
\item[(ii)] 
Let $\mathcal{F}$ be a coherent sheaf on $X_n$. Then we have a canonical isomorphism of $\OO_K$-modules
\[
\mathrm{Hom}_{\mathcal{O}_K}(\mathrm{H}^{1}\left(X_n,\mathcal{F}),\mathcal{O}_{K}/\pi^m\mathcal{O}_{K}\right) \stackrel{\sim}{\longrightarrow} \mathrm{H}^{0}(X_n,\mathcal{F}^{\vee}\otimes \omega_n) 
\]
for any integer $m\geq n$. 
\end{enumerate}
\end{lemma}
\begin{proof} 
(i). Since the structure morphism $X\to S$ is projective, it factors into a closed immersion $\iota\colon X\hookrightarrow Y:=\mathbb P_S^N$ for some $N$, followed by the projection $g\colon Y\rightarrow S$. Observe that by the smoothness and projectivity of $g$ we have $ \omega_{Y/S} =\mathrm{det}\ \Omega_{Y/S}^{1}:= \Omega_{Y/S}^{N}$, and $g^{!}\OO_S=\omega_{Y/S}[N]$ in the derived category of quasi-coherent sheaves on $Y$ (\cite{RD}, Corollary~3.4 p.~383 and proof of Theorem~4.1 p.~389).
Let $\mathcal J\subset \OO_{Y}$ denote the ideal sheaf defining $\iota$. Then $\iota^*({\mathcal J}/{\mathcal J}^{2})$ is a locally free $\OO_X$-module of rank $N-1$ (\cite{Liu}, Corollary~6.3.8). 
By \cite{Liu}~6.4.7, we have 
\begin{equation}\label{eq.omega-X/S}
\omega_{X/S}\cong \iota^{*}(\mathrm{det}\ \Omega_{Y/S}^{1})\otimes_{\OO_{X}}
 \mathrm{det} \left( \iota^{*}({\mathcal J}/{\mathcal J}^{2})\right)^{\vee} .
\end{equation}
Since $f=g\circ \iota$, from the functoriality of $(-)^{!}$ with respect to compositions of morphisms of finite type, we have the following isomorphisms in the derived category of quasi-coherent sheaves on $X$
\[
f^{!}\OO_{S}\cong \iota^{!}(g^{!}\OO_S)\cong \iota^{!} (\omega_{Y/S}[N])\cong (\iota^{!}\omega_{Y/S})[N].
\]
By the computation of $\iota^{!}$ in \cite{RD}, III, Corollary 7.3, recalling that $\omega_{X/Y}:= \mathrm{det}\left(\iota^*(\mathcal J/\mathcal J^{2})\right)^\vee $ (\cite{RD} p.~140--141, Definition (b)), we get 
\[
\iota^{!}{\omega}_{Y/S}\cong \iota^{\ast}{\omega}_{Y/S}  \otimes_{\OO_{X}}  \mathrm{det} \left( \iota^{*}({\mathcal J}/{\mathcal J}^{2})\right)^{\vee}[1-N]
\cong \omega_{X/S}[1-N],
\]
where we use \eqref{eq.omega-X/S} for the latter isomorphism. We conclude, then, that $f^{!}\OO_S \cong \iota^{!}\omega_{Y/S}[N]\cong \omega_{X/S}[1]$, as asserted. The isomorphism $f_{n}^{!}\OO_{S}\cong \omega_{n}$ can be derived in exactly the same way, after recalling that $f_n=f\circ i_n$ with $\iota_n\colon X_n\to X$ a regular immersion of codimension $1$.

(ii). The Grothendieck-Serre duality theorem (\cite{RD}, VII, Corollary 3.4 (c)) implies a canonical isomorphism
\begin{equation}\label{eq.GroSerDuality}
\mathrm{RHom}_{S}(\mathrm{R}f_{n\ast}\mathcal{F},\OO_{S})\stackrel{\sim}{\longrightarrow}\mathrm{RHom}_{X_n}(\mathcal{F},f_{n}^{!}\OO_{S}).
\end{equation}
Now,  $\mathrm{R}^{i}f_{n\ast}\mathcal{F}$ is the Zariski sheaf associated with the presheaf $U\mapsto \HH^{i}(X_n\times_S U, \mathcal{F}|_{X_n\times_S U})$, for any $i\in \Z$. Consequently, since $X_n$ is a scheme of dimension $1$, $\mathrm{R}^{i}f_{n\ast}\mathcal{F}=0$ for $i\neq 0,1$. 
As a result, $\mathrm{R}f_{n\ast}\mathcal{F}$ has bounded cohomology. 
Furthermore, $\mathrm{R}^{i}f_{n\ast}\mathcal{F}$ is the coherent $\OO_S$-module whose group of global sections is $\HH^{i}(X_n, \mathcal{F})$; for $i=0$ the fact is evident since $\mathrm{R}^{0}f_{n\ast}\mathcal{F}=f_{n\ast}\mathcal{F}$, while for $i=1$ one uses Grothendieck spectral sequence $ \HH^{i}(S,{\mathrm R}^jf_{n\ast}\mathcal{F})\Rightarrow \HH^{i+j}(X_n, \mathcal{F})$ (\cite{Weibel}, 5.8.2) and the vanishing of $\HH^{i}(S, f_{n\ast} \mathcal{F})$ for $i>0$, due to the affiness of $S$.

We next compute the $\HH^0$ of the left-hand side complex in \eqref{eq.GroSerDuality}. Consider the following spectral sequence (see \cite{Weibel} chap.~5, prop.~5.7.6 for the homological version of this spectral sequence, and the extra sign $-$ is due to the fact that $\mathrm{Hom}_{S}(-,\OO_S)$ is a contravariant functor)
\begin{equation}\label{eq.spectralseq}
E_{2}^{i,j}=\mathrm{Ext}^{i}_{S}(\mathrm{R}^{-j}f_{n\ast}\mathcal{F},\OO_S)\Longrightarrow \HH^{i+j}(\mathrm{RHom}_{S}(\mathrm{R}f_{n\ast}\mathcal{F},\OO_S)).
\end{equation}
Since $S$ is the spectrum of a discrete valuation ring, any coherent $\OO_{S}$-module $\mathrm{R}^{-j}f_{n\ast}\mathcal{F}$ admits a free resolution of length $1$. In particular, $E_{2}^{i,j}=0$ for $i\neq 0,1$. Therefore, the differential maps $d_{r}\colon E_{r}^{i,j}\rightarrow E_{r}^{i+r,j-r+1}$ are all zero for all $r\geq 2$. Consequently, $E_{\infty}^{i,j}=E_{2}^{i,j}$, and it is zero whenever $j\notin \{0,-1\}$ or $i\notin\{0,1\}$. The usual six terms short exact sequence associated to the Grothendieck spectral sequence provides then an isomorphism
\begin{equation}\label{H01}
\HH^{0}\left(\mathrm{RHom}_{S}(\mathrm{R}f_{n \ast}\mathcal{F},\OO_{S})\right)\stackrel{\sim}{\longrightarrow} \mathrm{Ext}^{1}_{S}(\mathrm{R}^1f_{n \ast}\mathcal{F},\OO_{S})
\end{equation}
because $\mathrm{Hom}_{S}(f_{n\ast}\mathcal{F},\OO_S)$ is zero, since $S$ is integral and the stalk of $f_{n\ast}\mathcal{F}$ at the generic point is trivial.

On the other hand, by (i), we have $f_{n}^{!}\OO_S\cong \omega_n$ in the derived category of quasi-coherent sheaves on $X_n$. As a result, we find
\begin{equation}\label{H02}
\HH^{0}\left(\mathrm{RHom}_{X_n}(\mathcal{F},f_{n}^{!}\OO_{S})\right)\cong\HH^0
(\mathrm{R}\mathrm{Hom}_{X_n}(\mathcal{F},\omega_n))\cong \mathrm{Hom}_{X_n}(\mathcal{F},\omega_n)\cong \HH^{0}(X_n,\mathcal{F}^{\vee}\otimes\omega_{n}).
\end{equation}
Now, on applying $\HH^{0}$ to both sides of \eqref{eq.GroSerDuality}, using \eqref{H01}, \eqref{H02}, and recalling that the functor $\Gamma(S,-)$ provides an equivalence between the category of quasi-coherent $\OO_S$-modules and the category of $\OO_K$-modules, we get the following isomorphism
\[
\mathrm{Ext}^{1}_{\OO_K}(\HH^{1}(X_n,\mathcal{F}), \OO_K)\stackrel{\sim}{\longrightarrow} \HH^{0}(X_n,\mathcal{F}^{\vee}\otimes\omega_{n}).
\]
Finally, since the $\mathcal{O}_{K}$-module $\mathrm{H}^{1}(X_n,\mathcal{F})$ is killed by $\pi^n$, we also have
\[ \mathrm{Hom}_{\mathcal{O}_K}(\mathrm{H}^{1}\left(X_n,\mathcal{F}),\mathcal{O}_{K}/\pi^m\mathcal{O}_{K}\right)
{\stackrel{\sim}{\longrightarrow}}
\mathrm{Ext}^{1}_{\mathcal{O}_K}(\mathrm{H}^{1}(X_n,\mathcal{F}),\mathcal{O}_K)
\]
for any integer $m\geq n$.
\end{proof}
\begin{lemma}[\cite{Raynaud4}, 3.7.1]\label{interNum} For any $i=1,\ldots,r$, we have $\omega_{X/S}\cdot
C_i=0$.
\end{lemma}
\begin{proof} Since $\omega_{X/S}|_{X_K}\cong
\OO_{X_K}$, we have $\omega_{X/S}\cdot X_{s}=0$, \emph{i.e.},
$\sum_{i=1}^{r}n_i\left(\omega_{X/S}\cdot C_i\right)=0$.
In particular, we get the Lemma if $r=1$. Suppose now $r\geq 2$;
since $C_i\cdot X_s=0$, we obtain $C_i\cdot C_i<0$. If $\omega_{X/S}\cdot C_i<0$, by \cite{Liu} Proposition 9.3.10 (a), the divisor $C_i$ is exceptional in the sense of \cite{Liu} Definition 9.3.1, which gives us a contradiction with the fact that $X/S$ is a minimal regular surface. So $(\omega_{X/S}\cdot
C_i)\geq 0$. As a result, $\omega_{X/S}\cdot C_i=0$ for any $i$.
\end{proof}

\begin{corollary}[{\cite{Raynaud4}, 3.7.2}]\label{dualisant} There is a unique integer $n$, $0\leq n< d$, such that $\omega_{X/S}\cong \mathcal{I}^{n}$.
\end{corollary}

\begin{proof} Since $\omega_{X/S}|_{X_K}\cong \OO_{X_K}$,
$\omega_{X/S}\cong \OO_{X}(Y)$, with $Y$ a divisor of $X$ with support contained in $X_s$.
Hence, $Y$ is a combination of the components $C_i$.
On the other hand, according to the previous Lemma, $Y\cdot C_i=0$, we have $Y\cdot Y=0$, hence, $Y$ is a rational multiple
of $X_s$ (\cite{BLR}, 9.5/10), \emph{i.e.}, $Y$ is linearly equivalent to $nD$ with $0\leq n<d$, whence the Corollary follows.
\end{proof}

The next result is needed in the proof of Proposition~\ref{prop.expression}.

\begin{corollary}[{\cite{Raynaud4}, 3.7.3}]\label{Classique} Suppose that $f\colon X\to S$ has a section $\sigma$. Let $\mathcal{J}$ be the ideal sheaf defining the closed immersion $\sigma\colon S\hookrightarrow X$ and set $\omega:=\sigma^{\ast}(\mathcal{J}/\mathcal{J}^{2})$.
 Then we have a canonical isomorphism $ f^{\ast}\omega \xrightarrow{\sim} \omega_{X/S}$.
\end{corollary}

\begin{proof} By assumption, the torsor $X_{K}$ has a $K$-rational point; it is then trivial as a torsor under $J_K$. According to Lemma \ref{d1=d2=d}, we have then $d=1$. As a result, $\omega_{X/S}\cong \OO_{X}$ (Corollary \ref{dualisant}), and the canonical morphism $f^{\ast}f_{\ast}\omega_{X/S}\rightarrow \omega_{X/S}$ given by adjunction is an isomorphism. On the other hand, we have:
\[
\OO_S\cong (f\sigma)^{!}\OO_{S}\cong \sigma^{!}\left(f^{!}\OO_S\right)\cong
\sigma^{!}(\omega_{X/S})[1]\cong \sigma^{\ast}\omega_{X/S}\otimes \omega^{\vee},
\]
where the first isomorphism follows from the fact that $f\circ \sigma=\mathrm{id}_{S}$, the second one uses the functoriality of $(-)^{!}$ with respect to compositions of morphisms, the third one is induced from Lemma~\ref{lem.rh}~(i), and the last one comes from the calculation of $\sigma^{!}$ (\cite{RD}, III, Corollary 7.3).
Hence $\omega\cong \sigma^{\ast}\omega_{X/S}\cong \sigma^{\ast}f^{\ast}f_{\ast}\omega_{X/S}\cong f_{\ast}\omega_{X/S}$ and we get
the following canonical isomorphisms $f^{\ast}\omega\cong f^{\ast}f_{\ast} \omega_{X/S} \cong \omega_{X/S}$, as desired.
\end{proof}

Let $n\geq 2$ be an integer, and ${\mathcal L}$ an invertible sheaf on $X$, which is of degree $0$ on each component of $X_1$.
Consider the following short exact sequence of sheaves over $X$
\begin{equation*}\label{suite de depart}
0\longrightarrow i_{n-1*} (\mathcal I|_{X_{n-1}}) \longrightarrow i_{n*} \OO_{X_n}\longrightarrow i_{1*}\OO_{X_{1}}\longrightarrow 0,
\end{equation*}
where $ |_{X_m}$ stands for the inverse image $i_m^*$.
By tensoring by the invertible sheaf $ {\mathcal L}^{\vee}\otimes \omega_{X/S}(X_n):= {\mathcal L}^{\vee}\otimes\omega_{X/S}\otimes \mathcal I^{n\vee}$ and recalling that $\omega_n:=i_n^*(\omega_{X/S}(X_n))$, we get the following
exact sequence of sheaves over $X$
\begin{equation}\label{eq.Lomega}
0\longrightarrow {\mathcal L}^{\vee}\otimes i_{n-1*} \omega_{n-1}\longrightarrow {\mathcal L}^{\vee}\otimes i_{n*} \omega_{n}\longrightarrow
{\mathcal L}^{\vee}\otimes i_{1*}(\omega_{n}|_{X_1})\longrightarrow 0 .
\end{equation}
Since $i_{n*}(\mathcal{L}^{\vee}|_{X_n}\otimes \omega_n)=i_{n*}(i_{n}^{\ast}\mathcal{L}^{\vee}\otimes \omega_n)\cong \mathcal{L}^{\vee}\otimes i_{n*}\omega_n$ and ${\mathcal L}^{\vee}\otimes i_{1*}(\omega_{n}|_{X_1})\cong i_{1*}(\mathcal L^\vee|_{X_1}\otimes \omega_n|_{X_1})$ we have the following exact sequence
\begin{equation}\label{exact}
0\longrightarrow
\HH^{0}(X_{n-1},{\mathcal L}^{\vee}|_{X_{n-1}}\otimes\omega_{n-1})\stackrel{h_1}{\longrightarrow}
\HH^{0}(X_n,{\mathcal L}^{\vee}|_{X_{n}}\otimes\omega_n)\stackrel{h_2}{\longrightarrow}
\HH^{0}(X_1,{\mathcal L}^{\vee}|_{X_{1}}\otimes\omega_n|_{X_{1}}).
\end{equation}
As a result, considering the lengths of the finite $\OO_K$-modules in \eqref{exact} and applying Lemma~\ref{lemme cohomologique}, we get
\begin{equation}\label{inequalities}
h^{0}(X_{n-1},{\mathcal L}^{\vee}|_{X_{n-1}}\otimes\omega_{n-1})\leq
h^{0}(X_n,{\mathcal L}^{\vee}|_{X_{n}}\otimes\omega_n)\leq
h^{0}(X_{n-1},{\mathcal L}^{\vee}|_{X_{n-1}}\otimes\omega_{n-1})+1.
\end{equation}

\begin{lemma}[\cite{Raynaud4}, 3.7.6] \label{suite exacte de dualisant} For $n$ a positive integer, we have either $\omega_{n}\cong {\mathcal L}|_{X_{n}}$, and in this case $\HH^{0}(X_1,{\mathcal L}^{\vee}|_{X_1}\otimes\omega_n|_{X_1})\cong k$ and the sequence (\ref{exact}) is right exact, or $\omega_n\not\cong
{\mathcal L}|_{X_{n}}$, in which case the canonical morphism $h_1\colon \HH^{0}(X_{n-1},{\mathcal L}^{\vee}{|_{X_{n-1}}}\otimes \omega_{n-1})\to \HH^{0}(X_n,{\mathcal L}^{\vee}{|_{X_{n}}}\otimes\omega_n)$ 
is bijective.
\end{lemma}
\begin{proof} Suppose first that $\omega_n\cong {\mathcal L}|_{X_n}$, so that ${\mathcal L}^{\vee}|_{X_{n}}\otimes \omega_n\cong \mathcal{O}_{X_n}$.
Then the map $h_2$ in \eqref{exact} can be identified (in a non canonical way) with the canonical map
\begin{equation}\label{suite1}
\HH^{0}(X_n,\OO_{X_n})\longrightarrow \HH^{0}(X_1,\OO_{X_1}).
\end{equation}
By Lemma~\ref{lemme cohomologique}, we have $ \HH^{0}(X_1,\OO_{X_1})\cong k $. Hence, every element of $\HH^{0}(X_1,\OO_{X_1})$ can be lifted to an element of $\HH^{0}(X_n,\OO_{X_n})$. As a result, the map (\ref{suite1}) is surjective and hence the complex (\ref{exact}) is right exact. Thus we obtain the first assertion.
Next, suppose $\omega_n\ncong {\mathcal L}|_{X_n}$. If $({\mathcal L}^{\vee}|_{X_{n}} \otimes\omega_n )|_{X_{1}}\not\cong \OO_{X_1}$ we conclude by (Lemma~\ref{lemme cohomologique}) that $h_2=0$ and hence $h_1$ is an isomorphism. 
Suppose then $({\mathcal L}^{\vee}|_{X_{n}} \otimes\omega_n )|_{X_{1}}\cong \OO_{X_1}$ and show that still $h_2=0$. 
Let $a\in \HH^{0}(X_n,\mathcal{L}^{\vee}|_{X_n}\otimes\omega_n)$. Suppose that the image of $a$ in $ \HH^0(X_1,\mathcal{L}^{\vee}|_{X_1}\otimes \omega_n|_{X_1})\cong k$ (Lemma~\ref{lemme cohomologique}) is non-zero. Then $a$ is nowhere vanishing because its restriction at any point of $X_1$ (\emph{i.e.}, at any point of $X_n$) is non-zero. Hence the invertible sheaf $\mathcal{L}^{\vee}|_{X_n}\otimes\omega_n$ is trivial because it admits a nowhere vanishing section. But this contradicts the hypothesis.
\end{proof}

Next we investigate the Picard group of $X_n$. Let $n\geq 2$ be an integer. The kernel of the surjective morphism
$\Pic(X_n)\rightarrow \Pic(X_{n-1})$ is an $\OO_K$-module of finite length killed by $p$. More precisely, consider the closed
immersion $ i_{n-1}' \colon X_{n-1}\hookrightarrow X_n$, given by the ideal sheaf
$\mathfrak{N}:=\mathcal{I}^{n-1}/\mathcal{I}^{n}\subset \OO_{X_{n}}$. The sheaf $\mathfrak{N}$ is nilpotent (in fact,
$\mathfrak{N}^{2}=0$), so we get a short exact sequence of abelian sheaves over $X_n$:
\[
0\longrightarrow 1+\mathfrak{N}\longrightarrow \OO_{X_n}^{\ast}\longrightarrow
i_{n-1*}'\OO_{X_{n-1}}^{\ast}\longrightarrow 0.
\]
Since $\mathfrak{N}^{2}=0$, the morphism $x\mapsto 1+x$ defines an isomorphism of abelian sheaves
\begin{equation}\label{seq.beta}
\beta\colon \mathfrak{N}\longrightarrow 1+\mathfrak{N}.
\end{equation}
Since $X_n$ has dimension $1$, the cohomology groups $\HH^{2}(X_n,1+\mathfrak{N})\cong \HH^{2}(X_n,\mathfrak{N})$ are zero.
 In this way we get the following long exact sequence
\begin{equation}\label{seq.PicR}
\HH^{0}(X_{n-1},\OO_{X_{n-1}}^{\ast})\stackrel{\partial^\ast}{\longrightarrow} \HH^{1}(X_{n},1+\mathfrak{N})\longrightarrow \Pic(X_n) \stackrel{\alpha}{\longrightarrow}
\Pic(X_{n-1})\longrightarrow 0.
\end{equation}
On the other hand, from the following exact sequence of sheaves over $X_n$:
\begin{equation*}
0\longrightarrow \mathfrak{N}\longrightarrow \OO_{X_n}\longrightarrow
i_{n-1*}'\OO_{X_{n-1}}\longrightarrow 0,
\end{equation*}
we obtain a long exact sequence (recall that
$\HH^{2}(X_n,\mathfrak{N})=0$):
\begin{equation}\label{seq.seqforH1}
\HH^{0}(X_{n-1},\OO_{X_{n-1}})\stackrel{\partial}{\longrightarrow}
\HH^{1}(X_{n},\mathfrak{N})\longrightarrow \HH^{1}(X_{n},\OO_{X_n})
\stackrel{\alpha'}{\longrightarrow} \HH^{1}(X_{n-1},\OO_{X_{n-1}})\longrightarrow 0
.
\end{equation}
We have then the following result:
\begin{lemma}[D\'evissage d'Oort, \cite{Oort}, Proposition in \S~6]\label{Oort}
Consider the morphisms $\partial$ and $\partial^{\ast}$ given respectively in \eqref{seq.seqforH1} and \eqref{seq.PicR} above, then one has $\beta(\mathrm{im}(\partial))=\mathrm{im}(\partial^{\ast})$.
\end{lemma}

As a consequence, $\ker(\alpha)\cong \mathrm{coker}(\partial)$ (as abelian sheaves). Since
$\mathfrak{N}=\mathcal{I}^{n-1}/\mathcal{I}^{n}$ is an $\OO_K$-module killed by $p$, it follows that $\mathrm{ker}(\alpha)\cong \mathrm{coker}(\partial)$ is an $\OO_K$-module of finite length killed by $p$. Hence the order of $\mathcal{I}|_{X_n}$ in $ \Pic(X_n)$ is equal either to the order of $\mathcal{I}|_{X_{n-1}}$, or $p$ times the order of $\mathcal{I}|_{X_{n-1}}$. 
Let $d'$ denote the order of the invertible sheaf $\mathcal{I}|_{X_1}$. Then for any $n\geq 2$, the order of $\mathcal{I}|_{X_n}$ is of the form $d'p^{\ell}$. 
On the other hand, since $\mathcal{I}|_{X_K}\cong \OO_{X_K}$, and since the invertible sheaf $\mathcal{I}$ is of order $d$, by Lemma~6.4.4
in \cite{Raynaud}, for $m\in {\mathbb Z}$ sufficiently large, the invertible sheaf $\mathcal{I}|_{X_m}$ is of order divisible by $d$.
Moreover, since $\mathcal{I}$ is of order $d$,
the order of $\mathcal{I}|_{X_m}$ divides $d$. Hence for $m\gg 0$, the order of $\mathcal{I}|_{X_m}$ is equal to $d$.
Let $e\geq 0$ be the integer such that $d=d'p^{e}$. For $i=0,\ldots, e$, let $m_i$ denote the smallest integer $n$ such that $\mathcal{I}|_{X_{n}}$ is of order $d'p^{i}$.
Let also \[\phi(n):=h^{1}(X_n,\OO_{X_n})\] be the length of the $\OO_K$-module $\HH^{1}(X_n,\OO_{X_n})$. 
By Lemma~\ref{lem.rh}~(ii) (with $\mathcal{F}=\OO_{X_n}$), $\phi(n)=h^0(X_n,\omega_{n})$. Therefore, according to the first inequality of \eqref{inequalities} we have $\phi(n)\geq \phi(n-1)$. Similarly, by Lemma~\ref{lem.rh} again, $\phi(n)>\phi(n-1)$ if and only if $h^0(X_n,\omega_{n})>h^0(X_{n-1},\omega_{n-1})$, hence if and only if $\omega_n\cong \OO_{X_n}$, in which case $\phi(n)=\phi(n-1)+1$ (Lemma~\ref{suite exacte de dualisant}). One then gets from Lemma \ref{Oort} the following Corollary:

\begin{corollary}\label{Pic_S} Let $n\geq 2$ be an integer. We have either $\phi(n)=\phi(n-1)$, in which case the morphism
$\alpha\colon\Pic(X_{n})\rightarrow \Pic(X_{n-1})$ is an isomorphism, or $\phi(n)=\phi(n-1)+1$, in which case $\ker(\alpha)$ is an
$\OO_{K}$-module of length $1$, and hence a $k$-vector space of dimension $1$.
\end{corollary}

\begin{proof} If $\phi(n)=\phi(n-1)$, the morphism $\alpha'$ in \eqref{seq.seqforH1} is an isomorphism, and this implies that $\partial$ in \eqref{seq.seqforH1} is surjective. Consequently, by Lemma~\ref{Oort}, the morphism $\partial^{\ast}$ in \eqref{seq.PicR} is surjective. Hence $\alpha\colon \Pic(X_{n})\rightarrow \Pic(X_{n-1})$ is an isomorphism. When $\phi(n)=\phi(n-1)$, $\ker(\alpha')\cong \mathrm{coker}(\partial)$ is an $\OO_{K}$-module of length $1$. From Lemma~\ref{Oort}, under the identification $\HH^{1}(X_n,\mathfrak{N})\cong \HH^{1}(X_n,1+\mathfrak{N})$ induced by the isomorphism $\beta$ in \eqref{seq.beta}, we get that $\ker(\alpha)\cong\mathrm{coker}(\partial^{\ast})$ is a $\OO_K$-module of dimension $1$. Hence $\ker(\alpha)$ is a $k$-vector space of dimension $1$, as desired.
\end{proof}

\begin{lemma}[\cite{Raynaud4}, 3.7.9]\label{ki}
\begin{itemize}
\item[(i)] For $i=0,\ldots,e$, the sheaf $\omega_{m_i}$ is isomorphic to $\mathcal{O}_{X_{m_i}}$.
\item[(ii)] For $i=0,\ldots, e-1$, there exists an integer $k_i>0$ such that $m_{i+1}=m_i+k_id'p^{i}$.
\item[(iii)] The integers $m\in (m_i,m_{i+1}]$ such that $\phi(m)=\phi(m-1)+1$ are exactly those which can be
written as $m=m_i+hd'p^{i}$ for some integer $h$.
\end{itemize}
\end{lemma}

\begin{proof} (i) Let $n>1$ be an integer such that the order of $\mathcal{I}|_{X_n}$ is different from that of $\mathcal{I}|_{X_{n-1}}$.
We then have $\phi(n)=\phi(n-1)+1$, and the canonical map $\Pic(X_n)\rightarrow \Pic(X_{n-1})$ has a kernel of length $1$ (Corollary \ref{Pic_S}). Now use the exact sequence (\ref{exact}) once again. 
By Lemma~\ref{lem.rh}~(ii), the injective morphism $\HH^{0}(X_{n-1},\omega_{n-1})\rightarrow \HH^{0}(X_n,\omega_n)$ has a non-trivial cokernel, and this implies that $\omega_n\cong \OO_{X_n}$ (Lemma~\ref{suite exacte de dualisant}).

For the assertions (ii) and (iii), recall that
$\omega_{m_i}=\omega_{X/S}(m_iD)|_{X_{m_i}}$, and by Corollary~\ref{dualisant}, there exists an integer $n$ ($1\leq n\leq d-1$)
such that $\omega_{X/S}\cong \mathcal{I}^{n}$.
Let $\mathcal{L}_m:=\omega_{X/S}(mD)$, then $\mathcal{L}_m\cong \mathcal{I}^{n-m}$.
 But $\omega_{m_{i+1}}=\mathcal{L}_{m_{i+1}}|_{X_{m_{i+1}}}=\mathcal{I}^{n-m_{i+1}}|_{X_{m_{i+1}}}\cong \OO_{X_{m_{i+1}}}$.
Hence, $\omega_{m_{i+1}}|_{X_{m_i}}=\mathcal{I}^{n-m_{i+1}}|_{X_{m_i}}\cong \OO_{X_{m_i}}$ by (i).
Since $\omega_{m_i}=\mathcal{I}^{n-m_i}|_{X_{m_i}}\cong \OO_{X_{m_i}}$, we have $\mathcal{I}^{m_{i+1}-m_i}|_{X_{m_i}}\cong \OO_{X_{m_i}}$,
and thus there exists an integer $k_i>0$ such that $m_{i+1}=m_i+k_id'p^{i}$.
The same argument also gives us that, for an integer $m$ so that $m_{i+1}\geq m> m_i$ and $\phi(m)>\phi(m-1)$, we have $\omega_m\cong \OO_{X_m}$ and there exists an integer $0<h\leq k_i$ verifying $m=m_i+hd'p^{i}$.
Conversely, let $m$ be an integer of the form $m=m_i+hd'p^{i}$ (for some $0<h\leq k_i$); we show that $\phi(m)>\phi(m-1)$.
We may assume that $m<m_{i+1}$, hence $\mathcal{I}|_{X_m}$ is of order $d'p^{i}$. By Lemma~\ref{suite exacte de dualisant},
we only need to show that $\omega_{m}\cong \OO_{X_m}$.
But
\begin{eqnarray*}
\omega_{m}&= &\mathcal{L}_m|_{X_m} \cong 
\mathcal{I}^{n-m}|_{X_{m}}=\mathcal{I}^{n-m_i-hd'p^{i}}|_{X_{m}} = \mathcal{I}^{n-m_{i+1}+m_{i+1}-m_i-hd'p^{i}}|_{X_{m} }
\\
& \cong & \omega_{m_{i+1}}|_{X_{m}}\otimes
\mathcal{I}^{(k_i-h)d'p^{i}}|_{X_m} \cong \OO_{X_m} ,
\end{eqnarray*}
since $\omega_{m_{i+1}}\cong \OO_{X_{m_{i+1}}}$ and $\mathcal{I}|_{X_{m}}$
is of order $d'p^{i}$. This completes the proof.
\end{proof}

\subsection{The function $\psi$} \label{def de psi}
We now come to the key construction of this section.
We define a function $\varphi\colon {\mathbb R}_{\geq 0}\rightarrow {\mathbb R}_{\geq 0}$ such that the graph of $\varphi$ is just the upper concave envelope of the set $\{(n,\phi(n))~|~n\in {\mathbb Z}_{\geq 0}\}\subset {\mathbb R}^{2}$.
Then $\varphi$ is a continuous piecewise linear function with $\varphi(0)=0$, and $\varphi(1)=1$ by Lemma \ref{lemme cohomologique}. Since $X/S$ is proper flat of relative dimension $1$, by \cite{EGA3}~7.7.5 (the equivalence of conditions (a') and (b') for $p=-1$) and 7.8.4 (the equivalence of conditions (b) and (e) for $p=-2$), we have
\[
\HH^{1}(X,\OO_X)\otimes_{\OO_{K}}\OO_{K}/{\pi^n\OO_K}\cong \HH^{1}(X_{nd},\OO_{X_{nd}}).
\]
Therefore $\lim_{n\rightarrow +\infty}\phi(nd)=+\infty$ since $\HH^{1}(X,\OO_{X})\otimes_{\OO_K}K\cong \HH^{1}(X_K,\OO_{X_K})$ is of dimension $1$ over $K$. Consequently, the function $\varphi\colon \mathbb{R}_{\geq 0}\rightarrow \mathbb{R}_{\geq 0}$ is strictly increasing with unbounded image, hence bijective.
Let $\psi\colon{\mathbb R}_{\geq 0}\rightarrow {\mathbb R}_{\geq 0}$ be its inverse. Then $\psi$ is also continuous and piecewise linear. For any integer $n\geq 1$, $\psi(n)$ is just the smallest integer $m\geq 1$ such that $\phi(m)=n$.
If we denote by ${\delta_n}$ the order of the invertible sheaf $\mathcal{I}|_{X_{\psi(n)}}$, then by Lemma~\ref{ki}~(iii) we have
\begin{equation}\label{eq.psind}
\psi(n+1)=\psi(n)+ \delta_n,
\end{equation}
and for all $m\in {\mathbb Z}$ such that $\psi(n)\leq m<\psi(n+1)$, the morphism of groups $ \Pic(X_{m})\rightarrow \Pic(X_{\psi(n)})$
is an isomorphism (Corollary \ref{Pic_S}). Observe that $\delta_1=d'$ with $d'$ the order of $\mathcal I|_{X_1}$. 
We call the functions $\varphi,\psi\colon{\mathbb Z}_{\geq 1}\rightarrow{\mathbb Z}_{\geq 1}$ the Herbrand functions; they are exact analogues
of the Herbrand functions used by Serre in \cite{SerreCFT}, \S 3.

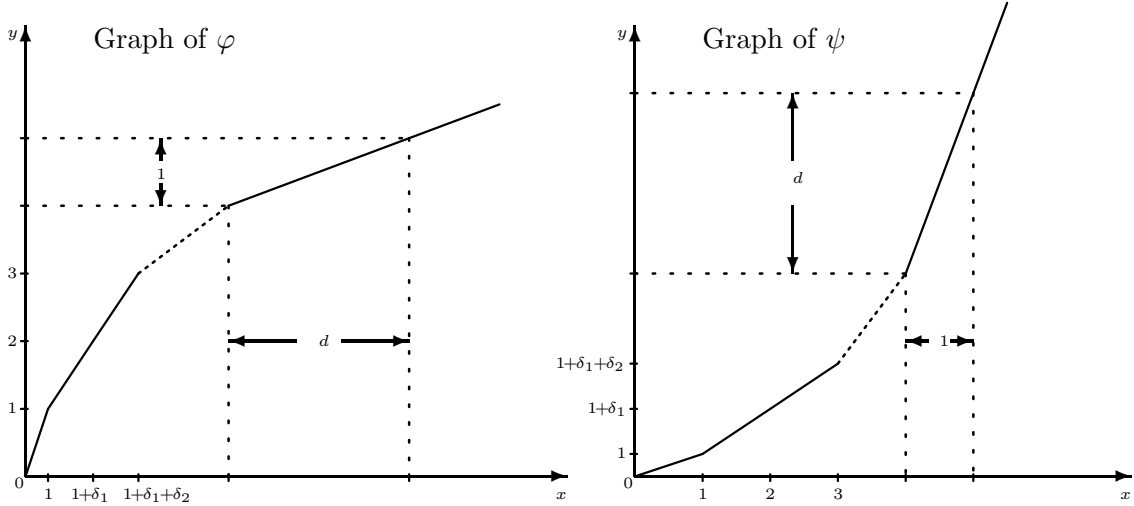
\setlength{\unitlength}{0.3cm}
\begin{picture}(70,25)
\thicklines \put(1,1){\vector(1,0){24}}
\put(1,1){\vector(0,1){20}} \drawline(1,1)(2,4)(4,7)(6,10)
\dottedline{0.4}(6,10)(10,13)\drawline(10,13)(18,16)(22,17.5)

\drawline(2,1.1)(2,0.8)\drawline(4,1.1)(4,0.8)\drawline(6,1.1)(6,0.8)
\drawline(10,1.1)(10,0.8)\drawline(18,1.1)(18,0.8)
\drawline(0.8,4)(1.1,4)\drawline(0.8,7)(1.1,7)\drawline(0.8,10)(1.1,10)
\drawline(0.8,13)(1.1,13)\drawline(0.8,16)(1.1,16)

\put(0.5,0.5){\tiny{$0$}} \put(1.8,0){\tiny{$1$}}
\put(3.0,0){\tiny{$1\!\!+\!\!\delta_1$}}
\put(5.3,0){\tiny{$1\!\!+\!\!\delta_1\!\!+\!\!\delta_2$}}\put(24.5,0){\tiny{$x$}}

\dottedline{0.8}(10,13)(10,1) \dottedline{0.8}(18,16)(18,1)
\put(13,7){\vector(-1,0){3}}\put(15,7){\vector(1,0){3}}\put(14,6.8){\tiny{$d$}}

\put(0.2,3.8){\tiny{$1$}}\put(0.2,6.8){\tiny{$2$}}
\put(0.2,9.8){\tiny{$3$}}\put(0.2,20.5){\tiny{$y$}}

\dottedline{0.8}(1,13)(10,13)\dottedline{0.8}(1,16)(18,16)
\put(7,15){\vector(0,1){1}}\put(7,14){\vector(0,-1){1}}\put(6.8,14.2){\tiny{1}}

\put(4,20){Graph of $\varphi$}

\put(28,1){\vector(1,0){22}} \put(28,1){\vector(0,1){20}}
\drawline(28,1)(31,2)(34,4)(37,6)
\dottedline{0.4}(37,6)(40,10)\drawline(40,10)(43,18)(44.5,22)

\drawline(31,1.1)(31,0.8)\drawline(34,1.1)(34,0.8)\drawline(37,1.1)(37,0.8)
\drawline(40,1.1)(40,0.8)\drawline(43,1.1)(43,0.8)

\drawline(27.8,2)(28.1,2)\drawline(27.8,4)(28.1,4)\drawline(27.8,6)(28.1,6)
\drawline(27.8,10)(28.1,10)\drawline(27.8,18)(28.1,18)

\put(27.5,0.5){\tiny{$0$}} \put(30.8,0){\tiny{$1$}}
\put(33.8,0){\tiny{$2$}} \put(36.8,0){\tiny{$3
$}}\put(49.5,0){\tiny{$x$}}

\dottedline{0.8}(40,10)(40,1) \dottedline{0.8}(43,18)(43,1)
\put(41,7){\vector(-1,0){1}}\put(42,7){\vector(1,0){1}}\put(41.5,6.8){\tiny{1}}

\put(27.2,1.8){\tiny{$1$}}\put(25.9,3.8){\tiny{$1\!\!+\!\!\delta_1$}}
\put(24.5,5.8){\tiny{$1\!\!+\!\!\delta_1\!\!+\!\!\delta_2$}}\put(27.2,20.5){\tiny{$y$}}

\dottedline{0.8}(28,10)(40,10)\dottedline{0.8}(28,18)(43,18)
\put(35,13){\vector(0,-1){3}}\put(35,15){\vector(0,1){3}}\put(35,14){\tiny{$d$}}

\put(31,20){Graph of $\psi$}

\end{picture}

To finish this section, we present some corollaries of the previous discussion, which will be useful in the next section.
The first follows directly from Lemma~\ref{suite exacte de dualisant}, \eqref{inequalities}, and Lemma~\ref{lem.rh}~(ii).

\begin{corollary}\label{lemme cohomologie pour L} Let ${\mathcal L}$ be an invertible sheaf on $X$, of degree $0$ on each component of $X_1$, and let $n$ be an integer $\geq 2$.
Then we have $h^{1}(X_{n-1},{\mathcal L}|_{X_{n-1}})\leq h^{1}(X_n,{\mathcal L}|_{X_{n}})\leq h^{1}(X_{n-1},{\mathcal L}|_{X_{n-1}})+1$.
Moreover, $ h^{1}(X_n,{\mathcal L}|_{X_{n}})= h^{1}(X_{n-1},{\mathcal L}|_{X_{n-1}})+1$ if and only if ${\mathcal L}|_{X_n}\cong \omega_n$.
\end{corollary}

\begin{corollary}[\cite{Raynaud4}, 3.9.1]\label{condition pour L} Let ${\mathcal L}$ be an invertible sheaf on $X$, of degree $0$ on each component of $X_1$,
and let $n$ be an integer $\geq 1$.
Then, if the morphism $ i_n^*\colon \HH^{1}(X,{\mathcal L})\rightarrow \HH^{1}(X_n,{\mathcal L}|_{X_{n}})$ is not bijective,
there exists an integer $m>n$ such that ${\mathcal L}^{\vee}|_{X_m}\otimes \omega_m$ is trivial on $X_m$.
\end{corollary}

\begin{proof} We first remark that $i_n^*$ is surjective,
and we have $\HH^{1}(X,{\mathcal L})\cong\varprojlim_{m\geq 0} \HH^{1}(X_m, {\mathcal L}|_{X_m})$.
 Moreover, $i_n^*$ is not bijective if and only if there exists $m> n$ such that $\HH^{1}(X_m, {\mathcal L}|_{X_m})\rightarrow \HH^{1}(X_{m-1},{\mathcal L}|_{X_{m-1}})$ is not injective.
By Lemma~\ref{lem.rh}, this is equivalent to saying that the injective morphism
$\HH^{0}(X_{m-1},{\mathcal L}^{\vee}|_{X_{m-1}}\otimes\omega_{m-1})\rightarrow \HH^{0}(X_m, {\mathcal L}^{\vee}|_{X_m}\otimes\omega_{m})$
is not surjective.
Hence, we have ${\mathcal L}^{\vee}|_{X_m}\otimes \omega_{m}\cong \OO_{X_m}$ (Lemma~\ref{suite exacte de dualisant}) and one concludes.
\end{proof}

For $\mathcal{L}$ an invertible sheaf on $X$ of degree $0$ on each component of $X_1$, the $\OO_K$-module $\HH^{1}(X,\mathcal{L})$ is of \emph{infinite} length if and only if $\mathcal{L}|_{X_K}\cong \OO_{X_K}$, that is, if and only if $\mathcal{L}$ is isomorphic to a power of $\mathcal{I}=\OO_{X}(-D)$ (see the proof of Corollary~\ref{dualisant}). Hence when $\mathcal{L}|_{X_K}$ is not isomorphic to $\OO_{X_K}$, the $\OO_K$-module $\HH^{1}(X,\mathcal{L})$ is of \emph{finite} length. Moreover, we have

\begin{corollary}[\cite{Raynaud4}, 3.9.2]\label{cle} Let ${\mathcal L}$ be an invertible sheaf on $X$ of degree $0$ on each component of $X_1$, and let $n\geq 1$ be an integer. Suppose that the $\OO_K$-module $\HH^{1}(X,{\mathcal L})$ is of length $\geq n$.
Then ${\mathcal L}|_{X_{\psi(n)}}\cong \mathcal{I}^{i}|_{X_{\psi(n)}}$ with $i$ a suitable integer.
\end{corollary}

\begin{proof} We will prove by induction that, under our assumptions and for any $n'$, $1\leq n'\leq n-1$,
the $\OO_K$-module $\HH^{1}(X_{\psi(n'+1)-1},{\mathcal L}|_{X_{\psi(n'+1)-1}})$ is of length $n'$.
As a result, the canonical map $\HH^{1}(X,{\mathcal L})\rightarrow \HH^{1}(X_{\psi(n)-1},{\mathcal L}|_{X_{\psi(n)-1}})$ is not bijective
(here we define by convention $\HH^{1}(X_{0},{\mathcal L}|_{X_{0}})=0$).
Hence Corollary~\ref{condition pour L} provides an integer $m>\psi(n)-1$ such that
${\mathcal L}|_{X_{m}}\cong \omega_{m}\cong \mathcal{I}^{\bar {n}-m}|_{X_{m}}$, where $\bar n$ is the integer appearing in Corollary~\ref{dualisant}. Since $\psi(n)\leq m$, the Corollary follows.

We begin with the case where $n'=1$ (hence $n\geq 2$). By Lemma~\ref{lemme cohomologique}, either the $\OO_K$-module $\HH^{1}(X_1,{\mathcal L}|_{X_1})$ is of length $1$, which is equivalent to saying that ${\mathcal L}|_{X_1}\cong \OO_{X_1}$, or $\HH^{1}(X_1,{\mathcal L}|_{X_1})$ is trivial, and, in this case, the natural morphism $\HH^{1}(X,{\mathcal L})\to \HH^{1}(X_1,{\mathcal L}|_{X_1})$ is not bijective.
By Corollary ~\ref{condition pour L} there is then an integer $m>1$, such that ${\mathcal L}|_{X_m}\cong \omega_{m}=\omega_{X/S}(X_m)|_{X_m}$.
Hence ${\mathcal L}|_{X_1}\cong \omega_{m}|_{X_1}$ is, isomorphic to $ \mathcal{I}^{\bar {n}-m}|_{X_{1}}$, again applying Corollary~\ref{dualisant}.
Thus, in both cases, we have ${\mathcal L}|_{X_1}\cong \mathcal{I}^{i}|_{X_1}$ for a suitable integer $i$.
Moreover, for $m'$ an integer such that $1\leq m'\leq \psi(2)-1 $, the canonical morphism $\Pic^{0}(X_{m'})\rightarrow \Pic^{0}(X_1)$ is bijective (see Corollary~\ref{Pic_S}); hence ${\mathcal L}|_{X_{m'}}\cong \mathcal{I}^{i}|_{X_{m'}}$. 
But $\psi(2)=\psi(1)+\delta_1=1+\delta_1$, and $\omega_{X/S}$ is isomorphic to a power of $\mathcal{I}$ (Corollary~\ref{dualisant}), there exists a unique integer $m'$ such that $1\leq m' \leq\psi(2)-1$, and $ {\mathcal L}|_{X_{m'}}\cong \omega_{m'}(=(\omega_{X/S}\otimes \mathcal{I}^{-m'})|{X_{m'}})$.
Hence, by Lemma~\ref{lemme cohomologie pour L}, we find that the $\OO_K$-module $\HH^{1}(X_{\psi(2)-1},{\mathcal L}|_{X_{\psi(2)-1}})$ is of length $1$.
Suppose now that the above assertion has been verified for $1\leq n'-1< n$ (with $n'<n$). Under the assumptions of the Lemma, the map
\[
\HH^{1}(X,{\mathcal L})\longrightarrow \HH^{1}(X_{\psi(n')-1},{\mathcal L}|_{X_{\psi(n')-1}})
\]
is not surjective. Hence, there exists an integer $m\geq \psi(n')$, such that ${\mathcal L}|_{X_m}\cong \omega_{m}$, and so ${\mathcal L}|_{X_{\psi(n')}}\cong \mathcal{I}^{i}|_{X_{\psi(n')}}$ for $0\leq i<\delta_{n'}$. 
On the other hand, since $\psi(n'+1)=\psi(n')+\delta_{n'}$ (see \eqref{eq.psind}), there exists a unique integer $m$ such that $\psi(n')\leq m\leq \psi(n'+1)-1$, and ${\mathcal L}|_{X_{m}}\cong \omega_m$, in particular, the $\OO_K$-module $\HH^{1} (X_{\psi(n'+1)-1}, {\mathcal L}|_{X_{\psi(n'+1)-1}})$ is of length $n'$. 
This finishes the induction, and hence also the proof of the Corollary.
\end{proof}

\subsection{Numerical studies}\label{Etude numerique}

We maintain the notation introduced in the beginning of this section. So $X_K$ denotes a $K$-torsor under an elliptic curve $A_K$, and $f\colon X\rightarrow S$ denotes its $S$-proper regular minimal model. We consider the $S$-proper minimal regular model $f'\colon X'\rightarrow S$ of the elliptic curve $\Pic^{0}_{X_K/K}$. As a result, $f'$ has a canonical section given by $e$, the schematic closure of the identity element of $X'_{K}$ in $X'$. Its dualizing sheaf is $f{'}^{\ast}\omega$ (with $\omega$ defined by the section $e$, see Corollary~\ref{Classique}). The main result of this part is that, by using some numerical invariants of $X/S$, one can then recover the sheaf $\omega_{X/S}$ from the sheaf $\omega$ ({Lemma~\ref{lemma.omega}}).

According to \cite{LLR}, Theorem 3.8, there exists a morphism of $\OO_K$-modules
\[
\tau_X\colon \HH^{1}(X,\OO_{X})\longrightarrow \HH^{1}(X',\OO_{X'})
\]
which extends the natural isomorphism over the generic fibre. Moreover, its kernel is the torsion part of
$\HH^{1}(X,\OO_X)$ (\cite{LLR}, 3.1 a)), and the $\OO_K$-modules $\ker(\tau_X)$ and $\mathrm{coker}(\tau_X)$ have the same length.
In this section we give an estimate for this length.

 By duality, we obtain the following map
\[
\tau_{X}^{\vee}\colon \HH^{0}(X',\omega_{X'/S})\cong
\left(\HH^{1}(X',\OO_{X'})\right)^{\vee}\longrightarrow
\left(\HH^{1}(X,\OO_X)\right)^{\vee}\cong
\HH^{0}(X,\omega_{X/S}).
\]
On the other hand, we have $f'_{\ast}\omega_{X'/S}\cong f_{\ast}'f^{'\ast}\omega\cong \omega$
(see Corollary~\ref{Classique} with $\omega$ defined via the $S$-section of $f'$ associated with the identity element of $X_K'$).
Hence we get following canonical map:
\[
\tau_{X}^{\vee}\colon\omega(S)\cong \HH^{0}(X',\omega_{X'/S})\longrightarrow
\HH^{0}(X,\omega_{X/S})=f_{\ast}\omega_{X/S}(S)
\]
which is injective, but not an isomorphism in general. Since $S$ is affine, $\tau_{X}^{\vee}$
also corresponds to an injective morphism of sheaves, again denoted by $\tau_{X}^{\vee}:$
\begin{equation}\label{tau}
\tau_{X}^{\vee}\colon \omega\longrightarrow f_{\ast}\omega_{X/S}.
\end{equation}
which gives by adjunction the following non-zero canonical morphism of sheaves on $X$:
\[
\tau_{X}^{\vee '}\colon f^{\ast}\omega\longrightarrow \omega_{X/S}.
\]
Since $\tau_{X}^{\vee}$ is an isomorphism on the generic fibre $X_K$ of $X$, the same holds for $\tau_{X}^{\vee'}$.
Under the identification given by the restriction of $\tau_{X}^{\vee'}$ to ${X_K}$, $f^{\ast}(\omega)$
and $\omega_{X/S}$ are naturally identified with two $\mathcal{O}_X$-submodules of $\omega_{X_K/K}\cong\left(\omega_{X/S}\right)|_{X_K}$.

\begin{lemma}[\cite{Raynaud4}, pp.~31--32]\label{lemma.omega} We have $\omega_{X/S}=f^{\ast}\omega \otimes \mathcal{I}^{-\chi}$ for a suitable non-negative integer $\chi$, as submodule of $\omega_{X_K/K}$.
Moreover, $[\chi/d]$, \emph{i.e.}, the biggest integer $\leq \chi/d$, is the length of the torsion part of $\HH^{1}(X,\OO_X)$.
\end{lemma}
\begin{proof} Since $f^{\ast}\omega$ is invertible and the scheme $X$ is regular, the generically injective morphism $\tau_{X}^{\vee '}$
is automatically injective.
By tensoring both sides with $\omega^{-1}_{X/S}$, we get an invertible ideal sheaf
\[
\mathcal{J}:=f^{\ast}\omega\otimes \omega_{X/S}^{-1}\hookrightarrow \mathcal{O}_{X}.
\]
The closed subscheme $V(\mathcal{J})$ of $X$ defined by $\mathcal{J}$ has support contained in $X_s$. Moreover, the intersection
numbers $V(\mathcal{J})\cdot C_i=0$ for any irreducible component $C_i$ of $X_s$ (Lemma~\ref{interNum}).
Hence the effective divisor $V(\mathcal{J})\hookrightarrow X$ is a multiple of $D=X_1=V(\mathcal{I})\hookrightarrow X$.
So one can find some non-negative integer
$\chi\in \mathbb{N}$ such that $\mathcal{J}=f^{\ast}\omega \otimes \omega_{X/S}^{-1}=\mathcal{I}^{\chi}$.
In other words, we find the following identification $\omega_{X/S}=f^{\ast}\omega \otimes \mathcal{I}^{-\chi}$
as submodules of $\omega_{X_K/K}$.
 This proves the first assertion. Under the latter identification, the morphism $\tau_{X}^{\vee '}$ is then obtained from the
canonical map: $\mathcal{O}_{X}\hookrightarrow \mathcal{I}^{-\chi}$ after tensoring both sides by $f^{\ast}\omega$.
Hence the morphism $\tau_{X}^{\vee}$ in \eqref{tau} can now be described as the following composition:
\[
\tau_{X}^{\vee}: \omega\stackrel{\sim}{\longrightarrow} f_{\ast}f^{\ast}\omega \cong f_{\ast}(f^{\ast}\omega\otimes \mathcal{O}_{X})\longrightarrow f_{\ast} (f^{\ast}\omega\otimes \mathcal{I}^{-\chi})=f_{\ast}\omega_{X/S},
\]
where the first isomorphism is just the adjunction map, and the third map is induced by the canonical injection $\mathcal{O}_{X}\hookrightarrow \mathcal{I}^{-\chi}$.
Hence, by using the projection formula
\[
f_{\ast}(f^{\ast}\omega\otimes \mathcal{I}^{-\chi})\cong \omega\otimes f_{\ast}\left(\mathcal{I}^{-\chi}\right),
\]
the morphism $\tau_{X}^{\vee}$ is obtained from the canonical map $\mathcal{O}_{S}=f_{\ast}\mathcal{O}_{X}\rightarrow f_{\ast}\left(\mathcal{I}^{-\chi}\right)$ after tensoring by the invertible sheaf $\omega$. Now, if we identify these two sheaves as $\mathcal{O}_S$-submodule of
$f_{K,\ast}\mathcal{O}_{X_K}=\mathcal{O}_{\mathrm{Spec}(K)}$ (here $f_K$ is the generic fibre of $f$), we have
$f_{\ast}\left(\mathcal{I}^{-\chi}\right)=\pi^{-[\chi/d]}\mathcal{O}_S\subset \mathcal{O}_{\mathrm{Spec}(K)}$.
Indeed, if we write $f_{\ast}(\mathcal{I}^{-\chi})=\pi^{-r}\OO_S$ for some non-negative integer $r$, then $r$ is the
largest integer $r'$ such that $\pi^{-r'}\OO_S\subset f_{\ast}(\mathcal{I}^{-\chi})$. But this last inclusion is equivalent to the inclusion
\[
f^{\ast}(\pi^{-r'}\OO_S)=\mathcal{I}^{-dr'}\subset \mathcal{I}^{-\chi},
\]
hence is also equivalent to the condition $-dr'\geq -\chi$, namely $r'\leq \chi/d$. The maximum of the
possible $r'$ is then given by $r=[\chi/d]$, and thus we obtain $f_{\ast}(\mathcal{I}^{-\chi})=\pi^{-[\chi/d]}\OO_S$.
Hence, if we identify $\omega$ and $f_{\ast}\omega_{X/S}$ as $\mathcal{O}_S$-submodule of
$\omega_K=\omega\otimes_{\mathcal{O}_S}\mathcal{O}_{\mathrm{Spec}(K)}$, the injection \eqref{tau} gives us the following equality inside
$\omega_K$: $f_{\ast}(\omega_{X/S}) =\pi^{-[\chi/d]}\omega$. As a result, we find that
$\mathrm{coker}(\tau_{X}^{\vee})$ and the torsion part of $\HH^{1}(X,\mathcal{O}_X)$ are both of length $[\chi/d]$.
\end{proof}

\begin{proposition}[\cite{Raynaud4}, pp.~31--32]\label{prop.expression} Let $X_K/K$ be a torsor under an elliptic curve $A_K$. Let $\chi\geq 0$ be the integer introduced in Lemma~\ref{lemma.omega}. Then one has
\begin{eqnarray}\label{expression}
\chi=(d-1) +k_0(d-d')+k_1(d-d'p)\ldots+k_{e-1}(d-d'p^{e-1}),
\end{eqnarray}
where the positive integers $k_i$ were introduced in Lemma~\ref{ki} and $d'$ is the order of the invertible sheaf  $\mathcal{I}_{|X_1}$, with $d= d'p^e$.
\end{proposition}

\begin{proof} By Lemma~\ref{ki}, we have $\phi(m_e)=1+k_0+\ldots k_{e-1}$, and, for $n>m_e$, we have $\phi(n)=\phi(n-1)+1$ if and only if $n-m_e$ is a multiple of $d=d'p^{e}$.
 In particular, if we write $m_e=hd-a$ for non-negative integers $h$, $0\leq a<d$, we have $\phi(m_e)=\phi(hd)$.
Let ${\mathcal T}$ be the torsion subsheaf of ${\mathcal M}:= \mathrm{R}^{1}f_{\ast}\OO_X$.
Consider ${\mathcal L} :={\mathcal M}/{\mathcal T}$, which is free of rank $1$ over $S$.
We have $\mathrm{R}^{1}f_{\ast}(i_{nd*}\OO_{X_{nd}})= \mathrm{R}^{1}f_{\ast}\left(\OO_{X}/\pi^{n}\OO_X\right)\cong {\mathcal M}/\pi^{n}{\mathcal M}$.
Hence, for $n\geq h$, the length of ${\mathcal M}/\pi^{n}{\mathcal M}$, \emph{i.e.}, $\phi(nd)$, increases by $1$ with $n$.
 This means that ${\mathcal T}$ is killed by $\pi^{h}$, and that
$\ell({\mathcal M}/\pi^{n}{\mathcal M})=\ell({\mathcal T})+\ell({\mathcal L}/\pi^{n}{\mathcal L})$.
In particular, if we take $n=h$, we get
\begin{equation}\label{phi(hd)}
\phi(m_e)=\phi(hd)=\ell\left(\mathrm{R}^{1}f_{\ast}\mathcal{O}_{X_{hd}}\right)=
\ell\left({\mathcal M}/\pi^{h}{\mathcal M}\right)=\ell({\mathcal T})+h.
\end{equation}
On the other hand, $\omega_{m_e}=\mathcal{I}^{-(\chi+m_e)}|_{X_{m_e}}$ is the trivial invertible sheaf (Lemmas~\ref{ki}, (i) and \ref{lemma.omega}), and since $\mathcal{I}|_{X_{m_e}}$
is of order $d$, there exists an integer $\alpha$ such that $\chi+m_e=\alpha d $. Hence $\chi=(\alpha-h)d+a$, and we have
$\ell(T)=[\chi/d]=\alpha-h$. Thus, (by using the equality (\ref{phi(hd)}) and Lemma~\ref{lemma.omega}), we find that
$ \ell(T)=[\chi/d]=\alpha-h=\phi(m_e)-h$. Hence $\alpha=\phi(m_e)$, and
\begin{eqnarray*}
\chi& = & \phi(m_e)d-m_e \\ & =&
(1+k_0+\ldots+k_{e-1})d-(1+k_0d'+\ldots+k_{e-1}d'p^{e-1})
\\ &=&
(d-1) +k_0(d-d')+k_1(d-d'p)\ldots+k_{e-1}(d-d'p^{e-1}).
\end{eqnarray*}
\end{proof}

\begin{corollary}[\cite{Raynaud4}, 3.8.2]\label{coro.tametorsor} The following conditions are equivalent:
\begin{itemize}
\item[(i)] $X/S$ is cohomologically flat in dimension $0$.
\item[(ii)] $\chi<d$.
\item[(iii)] $e=0$.
\item[(iv)] $\mathcal{I}|_{X_1}$ is of order $d$.
\end{itemize}
\noindent Moreover, if these conditions are satisfied, we have $\chi=d-1$.
\end{corollary}

\begin{proof} The $S$-scheme $X$ is cohomologically flat in dimension $0$ if and only if ${\mathcal T}$, the torsion subsheaf of
${\mathrm R}^1f_*\OO_X$, is trivial, \emph{i.e.}, if and only if $\ell({\mathcal T})=[\chi/d]=0$; the latter assertion is equivalent to saying that $\chi<d$, hence (i)$\Leftrightarrow$(ii). The equivalence between (iii) and (iv) comes from the definition of $e$. If $e>0$, then by Lemma~\ref{ki}~(ii) all integers $k_i$ are positive and $\chi\geq d$ by \eqref{expression}. Therefore (ii)$\Leftrightarrow$(iii).
\end{proof}

\begin{remark}\label{rem.tameandwild} Further conditions equivalent to the conditions in Corollary~\ref{coro.tametorsor} can be found in \ref{coro.tametorsor2}. By the discussion after Lemma~\ref{Oort}, these conditions are satisfied if $(d,p)=1$. But note that, as shown in \cite{Raynaud}~Remark~9.4.3~(d), it is possible for the equivalent conditions of Corollary~\ref{coro.tametorsor} to be satisfied when $p$ divides $d$.
\end{remark}

\section{Filtrations and comparison of the pro-algebraic structures}\label{filtrations}
{Recall that $\OO_K$ is} a complete discrete valuation ring with field of fractions $K$ and algebraically closed residue field $k$ of positive characteristic $p>0$, and $X$ denotes the proper regular minimal $S$-model of a $K$-torsor $X_K$ under an elliptic curve.
For each positive integer $n$ we have a canonical morphism of groups $\Pic^{0}(X)\rightarrow \Pic^{0}(X_n)$, from where we obtain a filtration on the group $\Pic^{0}_{X/S}(S)\cong\Pic^{0}(X)$ by the subgroups $\ker(\Pic^0(X)\rightarrow \Pic^0(X_n))$, $n\geq 1$.
On the other hand, the group $J(S)$ of the $S$-points of the identity component $J$ of the N\'eron model of the elliptic curve $\Pic^{0}_{X_K/K}$ is naturally filtered by the subgroups $\ker(J(S)\rightarrow J(S_n))$, $n\geq 1$, where $S_n:=\spec(\OO_K/\pi^n \OO_K)$. The reader should bear in mind that $X_n$ denotes the closed subscheme of $X$ defined by the ideal sheaf $\mathcal I^n$ and that $\mathcal{I}\subset \OO_X$ is the ideal sheaf such that $\mathcal I^d=\pi \OO_X\subset \OO_X$.
The aim of this section is to study the relation between these two filtrations with respect to the natural morphism of sheaves $q\colon \Pic^{0}_{X/S}\rightarrow J$ in \eqref{eq.SurPiczero}. The result can be stated in a satisfying form by using the Greenberg realization functors (see Theorem~\ref{resultat final}).

For this construction, we shall make intensive use of the notion of \emph{dilatation} of an $S$-scheme. We recall this construction briefly (see \cite{BLR}, \S~3.2, for more details). Let $H$ be a smooth $S$-scheme of finite type, $W\hookrightarrow
{H_s}$ a closed subscheme over $k$.
Denote by $\mathcal{J}$ the ideal sheaf of $W\hookrightarrow H$.
Let $\mathrm{Bl}_{W}(H)$ denote the blowing-up of $H$ along the center $W\hookrightarrow H$.
Then, by definition, \emph{the dilatation of $H$ along the center $W\hookrightarrow H$} is the largest open scheme $H'\subset \mathrm{Bl}_W(H)$ such that the ideal $\mathcal{J}\OO_{H'}\subset \OO_{H'}$ is generated by $\pi$.
According to \cite{BLR}, 3.2/1, $H'$ is a flat $S$-scheme, satisfying the following universal property:
let $Z$ be a flat $S$-scheme and $v\colon Z\rightarrow H$ a morphism of $S$-schemes such that its restriction to special fibres,
${v_s}\colon Z_s\rightarrow H_s$, factors through $W\hookrightarrow H_s$, then there exists a unique $S$-morphism $v'\colon Z\rightarrow H'$
rendering the obvious diagram commutative.
Furthermore, if $W$ is a smooth over $k$, then $H'$ is smooth over $S$ (\cite{BLR}, 3.2/3), and if $H$ is an $S$-group scheme, then the same is $H'$ (combine \cite{BLR}, 3.2/1 and 3.2/2).

In order to simplify the presentation, for the remainder of the section we will use the following notation: let $n\in \mathbb{Z}_{\geq 0}$ and $ m\in \mathbb Z_{\geq 0}\cup \{\infty\}$ with $n\leq m$ (by convention $n\leq \infty$ for any $n\in \mathbb Z_{\geq 0}$). Denote by $\mathrm{P}^{[n,m]}$ the kernel of the canonical morphism of functors $\Pic^{0}_{X_m/S}\rightarrow \Pic^{0}_{X_n/S}$.
Here, we set $X_{\infty}=X$ and $\Pic^{0}_{X_0/S}=0$, the final object in the category of abelian fppf-sheaves on $S$. Furthermore let
\[
\PP^{[n]}:=\PP^{[n,\infty]}=\ker(\Pic^{0}_{X/S}\longrightarrow
\Pic^{0}_{X_n/S}), \qquad \quad
\PP_{[n]}:=\PP^{[0,n]}=\Pic^{0}_{X_n/S}.
\]
In particular, $\PP^{[0]}=\Pic_{X/S}^{0}$.
For any integer $n\geq 1$, we define by induction a smooth $S$-group scheme $J^{[n]}$ as the the
dilatation of $J^{[n-1]}$ along the unit element of the special fibre of $J^{[n-1]}_s$ (here, $J^{[0]}:=J$).

\setcounter{lemma}{0}

\begin{lemma}\label{JJJ} For any $n\in {\mathbb Z}_{\geq 1}$, we have the following exact sequence:
\begin{equation}\label{eq.J}
0\longrightarrow J^{[n]}(S)\longrightarrow J(S)\longrightarrow
J(S_n)\longrightarrow 0.
\end{equation}
\end{lemma}

\begin{proof} The map $J(S)\to J(S_n)$ is surjective by the smoothness of $J$. By the universal property of dilatations we get inclusions $J^{[n]}(S)\subseteq J(S)$ and the exactness of \eqref{eq.J} for $n=1$. In order to prove the exactness for $n>1$ we need to work with the local description of dilatations as in \cite{BLR}, \S~3.2.
Let then $U\subset J$ be an open neighbourhood of the zero section $0:=0_J$ of $J$, and $f\colon U\to Z:= \mathbb{A}_S^m=\spec(\OO_K[x_1,\dots, x_m])$ an \'etale morphism of $S$-schemes sending $0\in J(S)$ to the zero section $0':=0_{Z}$ of $Z$; see \cite{BLR}, 2.2/11. 
Up to shrinking $U$ we may assume that $0_s\in U_s(k)$ is the only point above $0_{s}'\in Z_s(k)=k^m$. Let $U^{[n]}$ denote the pre-image of $U\subset J$ via the canonical map $J^{[n]}\rightarrow J$. 
Then, for $n\geq 1$, $U^{[n]}$ is the dilatation of $U^{[n-1]}$ along the closed point $0_s\in U^{[n-1]}_{s}(k)$ and $U^{[n]}(S)=J^{[n]}(S)$.
Define now inductively $Z^{[0]}:=Z$, and for $n>0$, $Z^{[n]}\cong \spec(\OO_K[\pi^{-n}x_1,\dots, \pi^{-n}x_m])$ the dilatation of $Z^{[n-1]}$ along the zero section of $Z^{[n-1]}_s$; on algebras the canonical map $Z^{[n]}\to Z^{[n-1]}$ sends $\pi^{-n+1}x_i$ to $\pi^{-n+1}x_i:=\pi(\pi^{-n}x_i)$. By direct computations, considering $Z$ as an $S$-group scheme via the isomorphism $Z\cong \G_a^m$ we get an exact sequence
\begin{equation}\label{eq.Z}
0\longrightarrow Z^{[n]}(S)\longrightarrow Z(S)\longrightarrow
Z(S_n)\longrightarrow 0
\end{equation}
for any $n$. Moreover, using the \'etaleness of $f$, one shows inductively that $f$ induces a morphism $f^{[n]}\colon U^{[n]}\rightarrow Z^{[n]}$ which identifies $U^{[n]}\rightarrow U$ with the base change of $Z^{[n]}\rightarrow Z$ along $f$. In particular, $f^{[n]}$ is \'etale, and $0_s\in U^{[n]}_s$ is the only point above the zero section of $Z_s^{[n]}$.

Let $0_n$ (respectively $0'_n$) denote the composition of $S_n\to S$ with the $0$ (respectively $0'$) section.
The morphism $f$ induces a bijection between $U^{[n]}(S)$ and $Z^{[n]}(S)$ for $n\geq 1$. Indeed, for $n=1$, the \'etaleness of the map $f$ assures that for any $\sigma \in Z(S)$ that becomes $0'$ modulo $\pi$ there is a unique $\sigma'\in U(S)$ that becomes $0$ modulo $\pi$. For higher $n$ one proceeds by induction recalling that $0_1\in U^{[n-1]}(S_1) $ is the only point above $0_1'\in Z^{[n-1]}(S_1)$, and that the induced morphism $f^{[n-1]}\colon U^{[n-1]}\rightarrow Z^{[n-1]}$ is \'etale.

We are now ready to show that \eqref{eq.J} is exact. Take $\tau\in J^{[n]}(S)= U^{[n]}(S)=Z^{[n]}(S)$ and let $\sigma:
=f(\tau)\in Z(S)$. Then the reduction of $\sigma$ modulo $\pi^n$ is $0_n'$ by the exactness of \eqref{eq.Z}. Hence the reduction modulo $\pi^n$ of $\tau$ in $U(S_n)$ must be $0_n\in U(S_n)$ since $U $ is \'etale over $Z$, and $0_s$ is the only point in $U_s(k)$ above $0'_s$. In particular the reduction of $\tau$ is $0_n\in J(S_n)$ and hence \eqref{eq.J} is a complex.
Finally, take $\tau\in J(S)$ whose reduction modulo $\pi^n$ is $0_n\in J(S_n)$.
 In particular $\tau\in U(S)$, and $f(\tau)\in Z(S)$ is contained in the kernel of the natural map $Z(S)\rightarrow Z(S_n)$. Thus, $f(\tau)\in Z^{[n]}(S)= U^{[n]}(S)=J^{[n]}(S)$, as desired.
\end{proof}
By \ref{JJJ}, one then has a diagram with exact rows and columns
\begin{eqnarray}\label{diagram for J}
\xymatrix{&   J^{[n]}(S)\ar@{^(->}[d]\ar@{=}[r]      & J^{[n]}(S) \ar@{^(->}[d]&&\\
0\ar[r]& J^{[n-1]}(S)\ar[r]\ar@{->>}[d]& J(S)\ar[r]\ar@{->>}[d]& J(S_{n-1})\ar[r]\ar@{=}[d]& 0 \\ 0\ar[r]& J^{[n-1]}(S_1)\ar[r]^{\varrho_n}& J(S_n)\ar[r]& J(S_{n-1})\ar[r]& 0.}
\end{eqnarray}

\subsection{Pro-algebraic structures}\label{fil}
Recall that, in this paper, a \emph{pro-algebraic group over $k$} is a pro-object in the category of $k$-group
schemes of finite type (see \S~\ref{sec.proalg-gree}).
The aim of this subsection is to show that, with the help of Greenberg realization functors, the morphism
$q\colon \mathrm{Pic}^0(X)\rightarrow J(S)$ is pro-algebraic in nature.

Let $n\geq 1$ be an integer. Consider $\Gr(\PP_{[n]})$, the Greenberg realization of the Picard functor $\PP_{[n]}=\Pic_{X_{n}/S}^{0}$ ({Theorem~\ref{Lipman}}). The natural morphism of functors $\PP_{[n+1]}\rightarrow \PP_{[n]}$ induces a morphism of smooth $k$-group schemes $\alpha_n\colon \mathrm{Gr}(\PP_{[n+1]})\rightarrow \mathrm{Gr}(\PP_{[n]})$. Thus we obtain a pro-algebraic group over $k$ (in the sense of
\S~\ref{sec.proalg-gree})
\[
\Gr(\Pic_{X/S}^{0}):=\{(\Gr(\PP_{[n]}), \alpha_n)\}_{n\geq 1}.
\]

\begin{lemma}\label{lemme en greenberg}
The morphism $\alpha_n\colon \mathrm{Gr}(\PP_{[n+1]})\rightarrow \mathrm{Gr}(\PP_{[n]})$ is a smooth and surjective morphism of smooth connected $k$-group schemes. Moreover, either $\alpha_n$ is an isomorphism, in which case we have $\phi(n+1)=\phi(n)$, or $\ker(\alpha_n)$ is a
$k$-vector group of dimension $1$, in which case we have $\phi(n+1)=\phi(n)+1$.
\end{lemma}
\begin{proof}By Theorem \ref{Lipman} (and its proof) $\mathrm{Gr}(\PP_{[n]})$ is (represented by) a smooth connected $k$-group scheme which is the identity component of $ \mathrm{Gr}(\Pic_{X_{n}/S})$. Now, by Proposition~\ref{pro.GrePicKer} the canonical map $\mathrm{Gr}(\Pic_{X_{n+1}/S})\to \mathrm{Gr}(\Pic_{X_{n}/S})$ is smooth and surjective with connected unipotent kernel. Hence, so is the restriction of $\alpha_n$ to the identity components. One concludes then by Corollary~\ref{Pic_S}.
\end{proof}

On passing to the projective limit of the perfect group schemes $\GGr(\PP_{[n]})$ and using the fact that
\[
\Pic^{0}(X)\cong\varprojlim \Pic^{0}(X_n)\cong\varprojlim \Gr(\PP_{[n]})(k),
\]
we get a pro-algebraic structure in the sense of Serre on the group $\Pic_{X/S}^{0}(S)=\Pic^{0}(X)$. Denote the Serre pro-algebraic group
so obtained by
\[
\bm{\Pic^{0}(X)}:=\varprojlim \bm{\Gr}(\PP_{[n]}).
\]
Similarly, since $\PP^{[n]}(S)=\ker(\Pic^{0}_{X/S}(S)\rightarrow \PP_{[n]}(S))$, we find that the group $\PP^{[n]}(S)$ can also
be endowed with a pro-algebraic structure in the sense of Serre, denoted by $\bm{\PP^{[n]}(S)}$.
Thus we obtain a decreasing filtration of $\bm{\Pic^{0}(X)}$ by pro-algebraic subgroups:
\begin{equation}\label{eq.FilPic}
\ldots \subset \bm{\PP^{[n+1]}(S)}\subset \bm{\PP^{[n]}(S)}\subset \ldots \subset \bm{\PP^{[1]}(S)}\subset \bm{\PP^{[0]}(S)}=\bm{\Pic^{0}(X)}.
\end{equation}

Secondly, from the $S$-group scheme $J$, we can construct a pro-algebraic group $\{(\Gr_{n}(J),\beta_n)\}_{n\geq 1}$, where each $\Gr_{n}(J)$ is smooth, and hence a Serre pro-algebraic algebraic group
\[
\bm{J(S)}:=\bm{\Gr}(J):=\varprojlim \bm{\Gr_n}(J)
\]
whose group of $k$-points is $J(S)$.
Moreover, the canonical map $J(S)=\Gr(J)(k)\rightarrow \Gr_n(J)(k)=J(S_n)$ is also pro-algebraic in nature, hence its kernel
can also be endowed with a pro-algebraic structure.
This last pro-algebraic group, according to the short exact sequence \eqref{eq.J}, is just the sub-pro-algebraic group
$\bm{J^{[n]}(S)}\subset \bm{J(S)}$ induced by the canonical map of $S$-group schemes $J^{[n]}\rightarrow J$.
In this way, we also obtain a decreasing filtration of $\bm{J(S)}$ by sub-pro-algebraic groups:
\begin{equation}\label{eq.FilJ}
\ldots \subset \bm{J^{[n+1]}(S)}\subset \bm{J^{[n]}(S)}\subset \ldots \subset \bm{J^{[1]}(S)}\subset \bm{J^{[0]}(S)}=\bm{J(S)}.
\end{equation}

On the other hand, for each integer $n\geq 1$, the morphism $q\colon \Pic_{X/S}^{0}\rightarrow J$ induces a morphism of
functors $\Pic_{X/S}^{0}\times_S S_n=\PP_{[nd]}\times_S S_n\rightarrow J\times_S S_n$, hence a morphism of algebraic $k$-groups:
\[
\Gr(\PP_{[nd]})\longrightarrow \Gr_n(J).
\]
In particular, we obtain a morphism of pro-algebraic groups:
\begin{eqnarray}\label{eq.morproalg}
\Gr(\Pic_{X/S}^{0})=\{(\Gr(\PP_{[n]}), \alpha_n)\}_{n\geq 1}\longrightarrow
\{(\Gr_{n}(J),\beta_n)\}_{n\geq 1}=\Gr(J).
\end{eqnarray}
In this way, we find that the canonical morphism $q\colon \mathrm{Pic}^0(X)\rightarrow J(S)$ is the morphism on $k$-rational
points induced by a morphism of Serre pro-algebraic groups:
\begin{eqnarray}\label{q pro-alg}
\bm q\colon \bm{\mathrm{Pic}^{0}(X)}\longrightarrow \bm{J(S)}.
\end{eqnarray}
In fact, we can be more precise in comparing the two filtrations \eqref{eq.FilPic} and \eqref{eq.FilJ}. The main result of this section
(see Theorem \ref{resultat final}) says that the above filtrations are compatible via $\bm q$ and this fact suggests that the morphism $\bm{q}$ should be thought as an analogue of the norm map studied by Serre in \cite{SerreCFT}~\S~3.3-3.4. In order to explore the compatibility between the two filtrations, we start by proving a useful result on the length of torsion sheaves.

\subsection{A result on intersection theory}\label{resultat intersection}
The results of this section hold for $\OO_K$ any discrete valuation ring; as usual $S:=\spec(\OO_K)$ and $s$ is the closed point of $S$. In the following, a coherent sheaf on an integral scheme is called a \emph{torsion coherent sheaf} if its stalk at the generic point is trivial. Moreover, for a torsion coherent sheaf $\mathcal{T}$ defined over the spectrum of a discrete valuation ring, let $\ell(\mathcal T)$ denote the length of $\mathcal T$. For $\alpha\in Z(S)$ an $S$-point of a separated $S$-scheme $Z$, let $\alpha(S)$ be the image of $S$ by $\alpha$, together with the reduced subscheme structure.

\begin{proposition}\label{intersection}
Let $Z$ be a smooth $S$-scheme of finite type and $\xi$ a generic point of its special fibre $Z_s$. Let $\alpha\colon S\rightarrow Z$ be a section of $Z/S$ such that $\alpha(s)\in \overline{\{\xi\}}\subset Z_s$. Let $\mathcal{M}$ be a coherent sheaf on $Z$, whose support $\mathrm{Supp}(\mathcal{M})$ is pure of codimension $1$ in $Z$.
Suppose that the stalk $\mathcal{M}_{\xi}$ of $\mathcal{M}$ at $\xi$ is of length $\ell$ as a torsion $\OO_{Z,\xi}$-module, and that $\alpha(S)\nsubseteq \mathrm{Supp}(\mathcal{M})$.

(1) We have $\ell(\alpha^{\ast}\mathcal{M}) \geq \ell$, with equality if and only if the following conditions are satisfied:
the support of $\mathcal{M}$ at $\alpha(s)$ is contained in $Z_s$, \emph{i.e.}, $\overline{\{\xi\}}\subset Z_s$ is the only possible component of $\mathrm{Supp}(\mathcal{M})$ containing $\alpha(s)$, and $\mathcal{M}$ is Cohen-Macaulay at $\alpha(s)$.

(2) Suppose furthermore that the support of $\mathcal{M}_K$ on $Z_K$ is not empty, and let $H_{K}$ be the scheme having $\mathrm{Supp}(\mathcal{M}_{K})$ as underlying space with its reduced structure and $H:=\overline{H_K}\subset Z$ its schematic closure in $Z$ (which is a relative effective divisor). Suppose moreover that $\alpha(s)\in H_s$. Let $\zeta$ be the generic point of an irreducible component of $H$ passing through $x$. Then $\ell(\alpha^{\ast}\mathcal{M})\geq \ell+1$. Moreover, if the equality
$\ell(\alpha^{\ast}\mathcal{M})=\ell+1$ holds, then
 (a)~$\mathcal{M}$ is Cohen-Macaulay at $\alpha(s)$; (b)~$H$ is regular at $\alpha(s)$, and $\mathcal{M}$ is of length $1$ at $\zeta$; (c)~$H$ cuts the closed subscheme $\alpha(S)\hookrightarrow Z$ transversally at $\alpha(s)$, \emph{i.e.}, the intersection number at $\alpha(s)$ of the closed subscheme $\alpha(S)$ (of dimension $1$) with $H$ (of codimension $1$ and regular around the point $\alpha(s)$) is equal to $1$.
\end{proposition}

Before proving this result consider the following technical Lemma.

\begin{lemma}\label{lemme technique} Let $Z=\mathrm{Spec}({\mathsf A})$ be a local noetherian regular scheme of dimension $2$,
and $\mathcal{M}$ a torsion coherent $\OO_Z$-module such that $\mathrm{Supp}(\mathcal{M})$ is of dimension $1$.
Let $H_1,\ldots, H_n$ be the irreducible components of $\mathrm{Supp}(\mathcal{M})$, endowed with the reduced subscheme structure.
Denote by $\xi_i$ the generic point of $H_i$, and by $\ell_i$ the length of $\mathcal{M}_{\xi_{i}}$ as an $\OO_{Z,\xi_i}$-module. Let finally $f\in {\mathsf A}$ be an element, which is part of a system of regular parameters of ${\mathsf A}$, such that $Z_1:=V(f)\subset Z$ is not contained in $\mathrm{Supp}(\mathcal{M})$. 
Then $\ell(\mathcal{M}/f\mathcal{M})\geq \sum_{i=1}^{n}\ell_i$, with equality if and only if the following conditions are satisfied: (i) for each $i$, the scheme $H_i$ is regular
and cuts $Z_1$ transversally in $Z$; (ii) the $\OO_Z$-module $\mathcal{M}$ is Cohen-Macaulay.
\end{lemma}

\begin{proof} Remark first that a coherent $\mathcal{O}_{Z}$-module $\mathcal{N}$ with one dimensional support is Cohen-Macaulay if and
only if $\mathcal{N}$ has no embedded associated points.
Indeed, suppose first that $\mathcal{N}$ has no embedded associated points. Let $\mathfrak{P}_1,\ldots,\mathfrak{P}_r\subset \mathsf A$
be the minimal ideals of the support of $\mathcal{N}$, and let $f\in \mathfrak{m}_{\mathsf A}\setminus \cup_{i}\mathfrak{P}_i$
 (where $\mathfrak{m}_{\mathsf{A}}\subset \mathsf A$ is the maximal ideal). Then multiplication by $f$ provides an injective map
(\cite{SerreLocalAlg}, Chapter\ I, \S\ B\, Corollary 1 of Proposition 7)
 \[
\mathcal{N}\longrightarrow \mathcal{N},\quad n\mapsto f\cdot n.
\]
Hence, the maximal $M$-sequence of $\mathcal{N}$ has at least $1=\mathrm{dim}(\mathcal{N})$ element, which implies that
$\mathcal{N}$ is Cohen-Macaulay (\cite{SerreLocalAlg}, \S~B.1, Definition 1). The converse statement follows from Proposition~13 of \S\ B.2 in \cite{SerreLocalAlg}.

 In order to prove the Lemma, we use induction on $n$.
 Let us begin with the case where $n=1$.
Denote by $\xi=\xi_1$ the generic point of $\mathrm{Supp}(M)_{\mathrm{red}}=H_1=H$, and by $\ell=\ell_1$ the length of $\mathcal{M}$
at $\xi$. Hence, $\mathcal{M}_{\xi}$ has a filtration by $\OO_{Z,\xi}$-submodules:
\[
0=\mathcal{M}_{\xi,0}\subset
\mathcal{M}_{\xi,1}\subset\ldots\subset
\mathcal{M}_{\xi,\ell}=\mathcal{M}_{\xi},
\]
where the successive quotients are isomorphic to $k(\xi)$.
We then define $\mathcal{M}_{i}$ as the inverse image of $\mathcal{M}_{\xi,i}$ via the canonical map $\mathcal{M}\rightarrow \mathcal{M}_{\xi}$,
thus obtaining a filtration on $\mathcal{M}$:
\[
\mathcal{M}_{-1}:=0\subseteq\mathcal{M}_{0}\subset \mathcal{M}_{1}\subset\ldots\subset
\mathcal{M}_{\ell}=\mathcal{M}.
\]
In general, ${\mathcal M}_{0}\neq 0$, and it is trivial if and only if $\mathcal{M}$ has no embedded associated points.
 For each $i\geq 0$, let $\mathcal{C}_i:=\mathcal{M}_{i}/\mathcal{M}_{i-1}$, which has no embedded associated point by definition whenever $i\geq 1$.
Moreover, $\mathcal{C}_{i,\xi}\cong \mathcal{M}_{\xi,i} /\mathcal{M}_{\xi,i-1}$, hence $\mathcal{C}_{0,\xi}=0$ and for $i\geq 1$, we have $\mathcal{C}_{i,\xi}\cong k(\xi)$.
 In particular, if $i\geq 1$, we have $\mathcal{C}_{i}\neq 0$ with schematic support equal to $H=\mathrm{Supp}(\mathcal{M})_{\mathrm{red}}$.
Indeed, we only need to show that the schematic support $\mathrm{Supp}(\mathcal{C}_i)=V(\mathrm{Ann}(\mathcal{C}_i))$ is \emph{reduced}.
Since $\mathcal{C}_{i,\xi}\cong k(\xi)$, $\mathrm{Supp}(\mathcal{C}_i)$ is reduced at $\xi$, hence it is generically reduced.
Furthermore, since $\mathcal{C}_i$ has no embedded associated points, so too is the scheme $\mathrm{Supp}(\mathcal{C}_i)$.
So $\mathrm{Supp}(\mathcal{C}_i)$ is reduced, and hence, equal to $H$ as subscheme of $Z$.
On the other hand, for $i\geq 1$, since the $\mathcal{O}_{Z}$-module $\mathcal{C}_i$ has no embedded associated points, and $Z_1=V(f)$ is
not contained in $H$, the map ``multiplication by $f$'':
\[
\mathcal{C}_i\longrightarrow \mathcal{C}_i, ~~~~ x\mapsto f\cdot x
\]
is injective for $i\geq 1$. From this, we get a filtration of $\mathcal{M}/f\mathcal{M}$:
\[
0\subseteq \mathcal{M}_{0}/f\mathcal{M}_0\subset
\mathcal{M}_{1}/f\mathcal{M}_1\subset\ldots\subset
\mathcal{M}_{\ell}/f\mathcal{M}_{\ell}=\mathcal{M}/f\mathcal{M},
\]
where for each $i\geq 1$, the quotient of $\mathcal{M}_{i}/f\mathcal{M}_{i}$ by $\mathcal{M}_{i-1}/f\mathcal{M}_{i-1}$ is isomorphic
to $\mathcal{C}_i/f\mathcal{C}_i$ which is non-zero since $\mathcal{C}_i\neq 0$. As a result, $\ell(\mathcal{M}/f\mathcal{M})\geq \ell$.
Moreover, $\ell(\mathcal{M}/f\mathcal{M})=\ell$, if and only if the following two conditions are realized:
(a) $\mathcal{M}_{0}/f\mathcal{M}_0=0$ which means $\mathcal{M}_0=0$ by Nakayama's lemma;
(b) for each $i$ ($1\leq i\leq \ell$), the $\OO_{Z}$-module $\mathcal{C}_i/f\mathcal{C}_i$ is of length $1$ over $\mathcal{O}_Z/f\OO_Z$.

Now, suppose that $\ell(\mathcal{M}/f\mathcal{M})=\ell$, or equivalently that the conditions (a) and (b) above are verified.
We will prove that $\mathcal{M}$ is Cohen-Macaulay, and the schematic support $H$ of $C_i$ is regular and cuts
the subscheme $V(f)\hookrightarrow Z$ transversally.
In fact, condition (a) implies that the $\mathcal{O}_{Z}$-module $\mathcal{M}$ has no embedded associated points, in particular, is Cohen-Macaulay.
On the other hand, suppose that $\mathrm{Ann}(\mathcal{C}_i)=(g)\subset A$ (hence $H$ is defined by the equation $g=0$ in $Z$),
and let $c\in \mathcal{C}_i$ be such that $c\notin f\mathcal{C}_i$. Condition (b) together with Nakayama's Lemma imply that the
$\OO_{Z}$-module $\mathcal{C}_i$ is generated by $c$. The morphism $\OO_Z\rightarrow \mathcal{C}_i=\OO_Z c$ defined
by $\lambda\mapsto \lambda c$ is then surjective, with kernel the ideal $(g)=\mathrm{Ann}(\mathcal{C}_i)=\mathrm{Ann}(c)$. Therefore,
$\OO_{Z}/(g,f)\cong \mathcal{C}_i/f\mathcal{C}_i$ is of length $1$ over $\OO_{Z}/f\OO_Z$. Hence
$\mathrm{Supp}(\mathcal{C}_i)=\mathrm{Supp}(\mathcal{M})_{\mathrm{red}}=H=V(g)$ is regular and cuts $V(f)\hookrightarrow Z$ transversally.
Conversely, suppose that $\mathcal{M}$ is Cohen-Macaulay and that the scheme $H$ is regular and cuts $V(f)\hookrightarrow Z$ transversally.
In particular, $\mathcal{M}$ has no embedded associated point, which implies that $\mathcal{M}_0=0$, whence condition (a) holds.
Moreover, since $H=\mathrm{Spec}({\mathsf A}/g{\mathsf A})$ is regular of dimension $1$, ${\mathsf A}/g{\mathsf A}$ is a principal ideal
domain. Therefore, the $\mathcal{O}_H=\mathcal{O}_{Z}/g\mathcal{O}_{Z}$-module $\mathcal{C}_i$ is free of rank $1$. Hence,
\[
\ell(\mathcal{C}_i/f\mathcal{C}_i)=\ell({\mathsf A}/(f,g))=1
\]
since $H=V(g)\hookrightarrow Z$ cuts $V(f)\hookrightarrow Z$ transversally.
In this way we get condition (b), which completes the proof of the Lemma in the case $n=1$.

Suppose now that the assertion of the lemma has been verified for $n-1\geq 1$. Let $\mathcal{M}'\subset \mathcal{M}$ be the submodule
defined as the kernel of the following map
\[
\mathcal{M}\longrightarrow \bigoplus_{i=2}^{n}\iota_{i,\ast}\iota_{i}^{\ast}\mathcal{M}
\]
with $\iota_i\colon \mathrm{Spec}(k(\xi_i))\rightarrow Z$ the canonical map, and define
$\mathcal{M}''$ by the following exact sequence:
\[
0\longrightarrow \mathcal{M}'\longrightarrow \mathcal{M}\longrightarrow
\mathcal{M}''\longrightarrow 0.
\]
Then $\mathcal{M}''$ has no embedded associated points (and so is Cohen-Macaulay) and has support $\cup_{i=2}^{n}H_i$, while $\mathcal{M}'$
has support $H_1$. One then has the following exact sequence (since $\mathcal{M}''$ Cohen-Macaulay
and $V(f)\nsubseteq \mathrm{Supp}(\mathcal{M}'')$):
\[
0\longrightarrow \mathcal{M}'/f\mathcal{M}'\longrightarrow
\mathcal{M}/f\mathcal{M}\longrightarrow
\mathcal{M}''/f\mathcal{M}''\longrightarrow 0.
\]
Hence, we have $\ell(\mathcal{M}/f\mathcal{M})=\ell(\mathcal{M}'/f\mathcal{M}')+
\ell(\mathcal{M}''/f\mathcal{M}'')$.
Moreover, by the definitions of $\mathcal{M}'$ and $\mathcal{M}''$, we have
\[
\mathcal{M}_{\xi_1}' \cong \mathcal{M}_{\xi_1}, \quad \textrm{and } \mathcal{M}_{\xi_i}''\cong \mathcal{M}_{\xi_i} \ \textrm{for }i=2,\ldots, n.
\]
By applying the induction hypothesis, we get
\begin{eqnarray*}
\ell(\mathcal{M}/f\mathcal{M}) = \ell(\mathcal{M}'/f\mathcal{M}')+
\ell(\mathcal{M}''/f\mathcal{M}'') \geq \ell(\mathcal{M}_{\xi_1}')+\sum_{i=2}^{n}\ell(\mathcal{M}_{\xi_i}'') = \sum_{i=1}^{n}\ell_i,
\end{eqnarray*}
with equality if and only if $\ell(\mathcal{M}'/f\mathcal{M}')=\ell_1$, and $\ell(\mathcal{M}''/f\mathcal{M}'')=\sum_{i=2}^{n}\ell_i$.
In other words, equality holds if and only if (a) $\mathcal{M}',\mathcal{M}''$ are Cohen-Macaulay, and
(b) the subschemes $H_i$ are regular cutting $Z_1$ transversally in $Z$. Since $\mathcal{M}''$ is already Cohen-Macaulay,
 condition (a) is equivalent to saying that $\mathcal{M}$ is Cohen-Macaulay. This completes the proof of the Lemma.
\end{proof}

\begin{proof}[Proof of \ref{intersection}]
Since $\alpha\colon S\rightarrow Z$ is a section of $Z/S$, there exist elements $f_1,\ldots,f_d$ of $\OO_{Z,x}$ which generate,
together with $\pi$, the maximal ideal of $\OO_{Z,x}$ and $\alpha(S)=V(f_1,\ldots,f_d)\hookrightarrow Z$. Up to replacing $Z$ by
its localization at $x$, we may assume that $Z$ is the spectrum of a regular local ring of dimension $d+1$.
In particular $Z_s=\overline{\{\xi\}}$ is regular and irreducible.
We will prove the Proposition by induction on $d$. The case $d=0$, \emph{i.e.}, $S=Z$, is trivial.
We start illustrating the case $d=1$. In this case, $Z$ is a $2$-dimensional local regular scheme, with $\mathcal{M}$ a torsion
coherent module on $Z$. When $\ell=0$, the conclusion of (1) is clear since in this situation, we always have
$\ell(\alpha^{\ast}\mathcal{M})\geq \ell=0$, and an equality means that $x\notin \mathrm{Supp}(\mathcal{M})$, or equivalently,
$\mathcal{M}_x=0$. In fact, here we have $\mathcal{M}=0$ since $x\in Z$ is the only closed point of the local scheme $Z$.
 To finish the proof of (1) for $d=1$, we may assume that $\ell\geq 1$.
Since $\xi\in \mathrm{Supp}(\mathcal{M})$, the closed subscheme $Z_s=\overline{\{\xi\}}\subseteq \mathrm{Supp}(\mathcal{M})$ is
one of the irreducible components of $\mathrm{Supp}(\mathcal{M})$ in $Z$.
 Now by applying Lemma~\ref{lemme technique}, we get $\ell(\alpha^{\ast}\mathcal{M})\geq \ell$.
The equality holds if and only if $Z_s$ is the only component of $Z$, $Z_s$ cuts $\alpha(S)$ transversally in $Z$
and $\mathcal M$ is Cohen-Macaulay at $x$. We now consider assertion (2).
By assumption, $\mathrm{Supp}(\mathcal{M})$ is the union of the one dimensional subscheme $H$ with, possibly, $Z_s$
if $\ell>0$; hence $\ell(\alpha^{\ast}\mathcal{M})\geq \ell+1$. If $\ell(\alpha^{\ast}\mathcal{M})=\ell+1$, on applying
Lemma~\ref{lemme technique} once again, we see that $\mathcal M$ is Cohen-Macaulay at $x$, the subscheme $H\subset Z$ is irreducible,
regular and cuts $\alpha(S)$ transversally at $x$. Moreover, $\mathcal{M}$ must have length $1$ at the generic point of $H$. This
proves (2), and hence the Proposition, for $d=1$.

For the general case, denote by $Z_1\hookrightarrow Z$ the closed subscheme defined by the ideal $(f_1)$, by $\xi_1$ the generic point of $Z_{1,s}$
and by $\mathcal{M}_1$ the pull-back of $\mathcal{M}$ to $Z_1$.
Then $Z_1$ is a regular local scheme of dimension $d$, which is not contained in the support of $\mathcal{M}$.
In particular, $\mathcal{M}_1$ is again a torsion coherent sheaf on $Z_1$. The morphism $\alpha\colon S\rightarrow Z$ factors
through $Z_1\hookrightarrow Z$, and we denote by $\alpha_1\colon S\rightarrow Z_1$ the morphism obtained in this way.
In particular, $\alpha^{\ast}{\mathcal{M}}\cong \alpha_{1}^{\ast}\mathcal{M}_1$ and $\alpha(S)\not\subseteq {\rm Supp}({\mathcal M}_1)$.
Hence, in order to prove the first assertion of (1), we only need to verify that the module $\mathcal{M}_1$ is of
length $\geq \ell$ at $\xi_1$ and then apply the induction hypothesis. To see the inequality $\ell(\mathcal{M}_{1,\xi_1})\geq \ell$, since
\[
\mathcal{M}_{1,\xi_1}\cong \mathcal{M}_{\xi_1}/f_1 \mathcal{M}_{\xi_1},
\]
we are reduced to showing that the restriction of the torsion sheaf $\widetilde{\mathcal{M}}:=\mathcal{M}|_{\mathrm{Spec}(\OO_{Z,\xi_1})}$
to the subscheme $\mathrm{Spec}(\OO_{Z,\xi_1}/f_{1}\OO_{Z,\xi_1})\hookrightarrow \mathrm{Spec}(\OO_{Z,\xi_1})=:\widetilde{Z}$
is of length $\geq \ell$. By definition, $\xi\in Z_s$ is contained in the special fibre $\widetilde{Z}_{s}$ of $\widetilde{Z}$,
and $\widetilde{\mathcal{M}}$ has length $\ell$ at $\xi\in \widetilde{Z}_{s}$.
Hence, we need only apply Lemma~\ref{lemme technique} to the two dimensional regular local scheme $\widetilde{Z}$ to get the conclusion.
We can also summarize the previous arguments by the following relations:
\begin{eqnarray}\label{eq.relations between length}
\ell(\alpha^{\ast}\mathcal{M})=\ell(\alpha_{1}^{\ast}\mathcal{M}_1)\geq \ell(\mathcal{M}_{1,\xi_1})=\ell(\widetilde{\mathcal{M}}/f_{1}\widetilde{\mathcal{M}})\geq \ell (\widetilde{\mathcal{M}}_{\xi})=\ell(\mathcal{M}_{\xi})=\ell.
\end{eqnarray}

Next, we examine the condition $\ell(\alpha^{\ast}\mathcal{M})=\ell$.
By \eqref{eq.relations between length}, we have $\ell(\alpha^{\ast}\mathcal{M})=\ell$ if and only if
(a) $\ell(\widetilde{\mathcal{M}}/f_{1}\widetilde{\mathcal{M}})=\ell(\widetilde{\mathcal{M}}_{\xi})=\ell$,
and (b) $\ell(\mathcal{M}_{1,\xi_1})=\ell(\alpha_{1}^\ast\mathcal{M}_1)$.
Consider the conditions
 (${\rm a}'$): $\mathcal{M}$ is Cohen-Macaulay having support contained in $Z_s$ at $\xi_1$ and
(${\rm b}'$): $\mathcal{M}_1$ is Cohen-Macaulay, with support contained in $Z_{1,s}$ at $x$.
On applying the induction hypothesis to the torsion module $\mathcal{M}_1$ on the $d$-dimensional scheme $Z_1$, one checks immediately
 that condition (b) is equivalent to condition (${\rm b}'$).
Furthermore on applying the induction hypothesis to the torsion module $\widetilde{\mathcal{M}}=\mathcal{M}|_{\mathrm{Spec}(\OO_{Z,\xi_1})}$
on the $2$-dimensional local scheme $\widetilde{Z}=\mathrm{Spec}(\OO_{Z,\xi_1})$, and since $\widetilde{Z}$ is the localization of $Z$
at the point $\xi_1\in Z$, we find that conditions (a) and (${\rm a}'$) are equivalent.
Hence $\ell(\alpha^{\ast}\mathcal{M})=\ell$ if and only if (${\rm a}'$) and (${\rm b}'$) hold.

Now, we proceed with the proof of the second part of (1). Suppose first $\ell(\alpha^{\ast}\mathcal{M})=\ell$, or equivalently,
that the previous conditions (${\rm a}'$) and (${\rm b}'$) hold, and prove that $\mathcal{M}$ is Cohen-Macaulay with support
contained in $Z_{s}$ at $x$.
We first claim that the multiplication by $f_1$ on $\mathcal{M}$ provides an injective map.
To see this fact, let $\mathcal{M}'$ be the submodule of $\mathcal{M}$ formed by the elements killed by a power of $f_1$,
and let $\mathcal{M}''=\mathcal{M}/\mathcal{M}'$. By definition,
$\mathrm{Supp}(\mathcal{M}')\subset Z_1\cap \mathrm{Supp}(\mathcal{M})$, which is hence of codimension at least $2$.
By definition of $\mathcal{M}''$, the multiplication by $f_1$ on $\mathcal{M}''$ is an injective map,
hence the canonical map
\begin{equation}\label{eq.MM}
\mathcal{M}'/f_{1}\mathcal{M}'\longrightarrow \mathcal{M}/f_{1}\mathcal{M}=\mathcal{M}_1
\end{equation}
is injective.
On one hand, condition (${\rm a}'$) above implies that $\xi_1\notin \mathrm{Supp}(\mathcal{M'})$
(since $\mathrm{Supp}(\mathcal{M}')\subset V(f_1)\cap Z_s=Z_{1,s}\subset Z_s$ and $\mathcal{M}$
is Cohen-Macaulay at $\xi_1$ by (${\rm a}'$)), hence $\xi_1\not\in\mathrm{Supp}(\mathcal{M}'/f_{1}\mathcal{M}')$.
 In particular, $\mathrm{Supp}(\mathcal{M}'/f_1\mathcal{M}')\cap Z_{1,s}\subsetneq Z_{1,s}$.
On the other hand, condition (${\rm b}'$) and the injection \eqref{eq.MM} imply that the support
of $\mathcal{M}'/f_{1}\mathcal{M}'$ is contained in $Z_{1,s}$.
Hence $\mathcal{M}'/f_{1}\mathcal{M}'=0$, and then $\mathcal{M}' =0$ by Nakayama's Lemma.
As a result, the multiplication by $f_1$ on $\mathcal{M}$ is injective.
Moreover the quotient sheaf $\mathcal{M}_1=\mathcal{M}/f_1\mathcal{M}$ is Cohen-Macaulay of dimension $d-1=\mathrm{dim}(Z_1)-1$
by (${\rm b}'$), hence also $\mathcal{M}$ is Cohen-Macaulay.
 To see that the support of $\mathcal{M}$ is contained in $Z_s$, suppose $\mathrm{Supp}(\mathcal{M})$ contains a component
$\Gamma$ different from $Z_s$ at the point $x$.
Since we have shown that $\mathcal{M}$ is Cohen-Macaulay at $x$, it follows that $\Gamma$ is also of codimension $1$.
By condition (b) above, $\mathcal{M}_{1}$ has support contained in $Z_{1,s}$ at $x$. So $\Gamma\cap Z_1\subset Z_{1,s}$ which is in
fact an equality of sets for reasons of dimension.
As a result, one finds that $\xi_1\in \Gamma$, which means that $\mathrm{Supp}(\mathcal{M})$ has at least two components
(of codimension $1$) at $\xi_1$, but this is impossible because of the condition (${\rm a}'$).
This proves that $\mathcal{M}$ is Cohen-Macaulay with support contained in $Z_s$ at $x$. Conversely, suppose $\mathcal{M}$ is
Cohen-Macaulay with support contained in $Z_s$ at $x$. We must prove that $\ell(\alpha^{\ast}\mathcal{M})=\ell$.
First of all, this condition implies in particular that $\mathcal{M}$ is Cohen-Macaulay with support contained
in $Z_s$ at $\xi_1$, namely condition (${\rm a}'$) holds. To complete the proof of (1), we only need to show that
condition (${\rm b}'$) also holds.
It is clear that the support of $\mathcal{M}_{1}$ is contained in $Z_{1,s}$, so we need only verify that $\mathcal{M}_1$ is Cohen-Macaulay.
Since $Z_{1,s}=\mathrm{Supp}(\mathcal{M}_1)$ has dimension equal to $\mathrm{dim}(\mathcal{M})-1$, $\mathcal{M}_1$ is also Cohen-Macaulay (\cite{SerreLocalAlg}, Chapter~IV $\S$ B.2, Proposition~14). This finishes the proof of (1).

To finish the proof of (2), since this is a local question for the \'etale topology on $S$, we may assume that $S$ is strictly local,
in particular the residue field is an infinite set. This implies that the residue field $k(x)$ of $Z$ at $x$ is also infinite.
Since $k(x)$ is an infinite field, up to replacing $f_1$ by $f_1+\lambda f_2$, for a suitable $\lambda\in \OO_{Z,x}^{\ast}$,
we may assume that $Z_{1,s}\nsubseteq H_s$, so that $H_{1,s}=H_s\cap Z_{1,s}\hookrightarrow Z_{1,s}$ is of codimension $1$ in $Z_{1,s}$
(where $H_1:=H\cap Z_1$).

As we have seen in \eqref{eq.relations between length}, $\mathcal{M}_{1}$ has length $\geq\ell$ at $\xi_1$.
Since $x\in H_{1,s}$, on applying the induction hypothesis to $Z_1$, we find that
\begin{equation}\label{eq.relationlength2}
\ell(\alpha^{\ast}\mathcal{M})=\ell(\alpha_{1}^{\ast}\mathcal{M}_{1})\geq \ell(\mathcal{M}_{1,\xi_1})+1\geq \ell +1.
\end{equation}
This proves the first assertion in (2). From now on, we suppose that $\ell(\alpha^{\ast}\mathcal{M})=\ell+1$.
According to \eqref{eq.relationlength2}, we get
$\ell(\alpha^{\ast}\mathcal{M})=\ell(\alpha_{1}^{\ast}\mathcal{M}_1)=\ell+1$, and $\mathcal{M}_{1}$ is of length $\ell$ at $\xi_1$.
By the induction hypothesis, we have (i) $H_{1,\mathrm{red}}$ is irreducible and regular, and moreover $\alpha_1(S)$ cuts $H_1$ transversally in $Z_1$ at $x$; (ii) $\mathcal{M}_1$ is Cohen-Macaulay at $x$ in $Z_1$, and if we denote by $\zeta_1\in H_1$ the generic point of $H_1$, then $\mathcal{M}_1$ is of length $1$ at $\zeta_1$.
Denote by $Z'$ the localization of $Z$ at $\zeta_1$, and by $\mathcal{M}'$ the inverse image of $\mathcal{M}$ by the canonical morphism $Z'\rightarrow Z$.
Then, $\mathcal{M}'/f_1\mathcal{M}'$ is of length $1$ over $\OO_{Z'}/f_1\OO_{Z'}$.
Hence, on applying Lemma~\ref{lemme technique} to the torsion module $\mathcal{M}|_{\mathrm{Spec}(\mathcal{O}_{Z,\zeta_1})}$ on the two dimension regular local scheme $\mathrm{Spec}(\mathcal{O}_{Z,\zeta_1})$ we get that $H$ is regular at $\zeta_1$, and it cuts $Z_1$ transversally at $\zeta_1$.
Moreover, $\mathcal{M}$ is Cohen-Macaulay with support contained in $H$ at $\zeta_1$, and $\mathcal{M}$ is of length $1$ at the generic point $\zeta$ of $H$. Using now the fact that $H_{1}=H\cap Z_1$ is irreducible, we find that $H$ itself must be irreducible, since otherwise, $H$ would have at least two components at $\zeta_1$. Therefore, $H_1$ is generically reduced.
But $H_1$ is a divisor inside a regular scheme $Z_1$, hence $H_1$ is Cohen-Macaulay and thus $H_1$ is reduced. By assertion (i),
we find that $H$ is irreducible and regular, cutting $Z_1$ transversally at $x$ inside $Z$. Now we need only verify that $\mathcal{M}$
is Cohen-Macaulay on $Z$. By (ii), we need only show that $\mathcal{M}$ has no embedded associated points.
Denote by $\mathcal{N}'$ the biggest quotient without embedded associated points of $\mathcal{M}$, and denote by $\mathcal{N}$
the $\OO_Z$-submodule defined by the following exact sequence
\[
0\longrightarrow \mathcal{N}\longrightarrow \mathcal{M}\longrightarrow
\mathcal{N}'\longrightarrow 0;
\]
we have the short exact sequence
\[
0\longrightarrow \mathcal{N}/f_1\mathcal{N}\longrightarrow \mathcal{M}/f_1\mathcal{M}\longrightarrow \mathcal{N}'/f_1\mathcal{N}' \longrightarrow 0.
\]
According to Nakayama's Lemma, to complete the proof, we need only show that $\mathcal{N}/f_1\mathcal{N}=0$.
By definition, $\mathrm{Supp}(\mathcal{N}')=\mathrm{Supp}(\mathcal{M})$, and $\ell(\mathcal{N}_{\xi}' )=\ell(\mathcal{M}_{\xi})=\ell$.
Hence, according to the first part of (2) (which has already been proved), we have $\ell(\alpha^ {\ast}\mathcal{N}')\geq \ell+1$.
On the other hand, there is a surjection $\alpha^{\ast}\mathcal{M}\rightarrow \alpha^{\ast}\mathcal{N}'$, whence we have
 $\ell+1=\ell(\alpha^{\ast}\mathcal{M})\geq \ell(\alpha^{\ast}\mathcal{N}' )$.
As a result, we have $\ell(\alpha^{\ast}\mathcal{N}' )=\ell+1$. Hence the $\OO_{Z}$-module $\mathcal{N}'$ again satisfies
the assumptions of the assertion (2) of this Proposition: On applying what was proved a few lines above to the torsion sheaf
$\mathcal{N}'$ in place of $\mathcal M$, we get $\ell\left(\left(\mathcal{N}'/f_{1}\mathcal{N}'\right)_{\zeta_1}\right)=1$,
and $\ell\left(\left(\mathcal{N}'/f_{1}\mathcal{N}'\right)_{\xi_1}\right)=\ell$.
Hence $(\mathcal{N}/f_{1}\mathcal{N})_{\zeta_1}=(\mathcal{N}/f_{1}\mathcal{N})_{\xi_1}=0$. As a result, the support of
$\mathcal{N}/f_{1}\mathcal{N}$ is of dimension $<d-2$.
But we have seen that $\mathcal{M}_{1}=\mathcal{M}/f_1\mathcal{M}$ is Cohen-Macaulay with support $H_1\cup Z_{1,s}$
of dimension $d-2$, hence we must have $\mathcal{N}/f_{1}\mathcal{N}=0$.
This completes the proof.
\end{proof}

\subsection{Preliminaries on the comparison between the pro-algebraic structures}\label{sec.compfil}
The aim of this section is to show that the canonical map of sheaves $q\colon \Pic^{0}_{X/S}\rightarrow J$ in \eqref{eq.SurPiczero}
induces, for each $n\geq 1$, a morphism of smooth algebraic $k$-groups
\begin{equation}\label{eq.qn}
q_n\colon \Gr(\PP_{[\psi(n)]})=\mathrm{Gr}(\Pic^{0}_{X_{\psi(n)}/S})\longrightarrow
\mathrm{Gr}_n(J),
\end{equation}
and that the maps $q_n$ are compatible in the evident way.

Let $n\geq 1 $ be an integer. We have seen in \S~\ref{fil} that there exists a morphism of sheaves
$\Pic^{0}_{X_{nd}/S} \times_S S_n\rightarrow J\times_S S_n$.
In this way we get a morphism of smooth algebraic $k$-groups
\begin{equation}\label{mor}
q_n'\colon \mathrm{Gr}(\PP_{[nd]})\longrightarrow \mathrm{Gr}_{n}(J) .
\end{equation}
Since $\psi(n)\!\!\leq\!\! nd$ (cf. \S~\ref{def de psi}), there is a canonical morphism of smooth algebraic $k$-groups
\[
\Gr(\PP_{[nd]})\longrightarrow
\Gr(\PP_{[\psi(n)]}).
\]
So, in order to prove the existence of $q_n$ as above, it suffices to verify that the morphism $q_n'$ in \eqref{mor} factors as follows:
\[
\xymatrix{ {\Gr}(\PP_{[nd]})\ar[r]^{~q_n'}\ar[d]&
{\Gr}_{n}(J)\\
{\Gr}(\PP_{[\psi(n)]})\ar@{.>}[ru]_{q_n}& }
\]
On the other hand, since the morphism of algebraic $k$-groups $\mathrm{Gr}(\PP_{[nd]})\rightarrow
\mathrm{Gr}(\PP_{[\psi(n)]})$ has smooth kernel (see Lemma~\ref{lemme en greenberg}) and $k$ is algebraically closed, \emph{we need only
check the factorization on the level of $k$-points.} Since the maps $\Pic^{0}(X)\to \Pic^{0}(X_{[\psi(n)]})={\Gr}(\PP_{[\psi(n)]})(k)$ are
surjective, the verification is reduced to proving the existence of the following factorization:
\begin{equation*}\label{dia.PicJq}
\xymatrix{\Pic^{0}(X)\ar@{->>}[r]\ar@{->>}[d]_{q}&
\Pic^{0}(X_{\psi(n)})\ar@{.>}[d]^{q_n}\\ J(S)\ar@{->>}[r]& J(S_n)}
\end{equation*}
where we again denote by $q_n$ the map induced by \eqref{eq.qn} on $k$-points.
 With the help of the rigidified Picard functor we will establish this factorization via induction on $n$.

\subsection{Comparison of the pro-algebraic structures}\label{Not}
 In the following discussion we fix a rigidificator $Y\hookrightarrow X$ of the relative Picard functor $\Pic_{X/S}$, and for simplicity, we
denote by $G:=(\Pic_{X/S},Y)^{0}$ the identity component of the rigidified Picard functor of $X/S$ along $Y/S$.
As usual $J$ denotes the identity component of the N\'eron model of $\Pic^{0}_{X_K/K}$ over $S$.
According to Proposition 3.2 in \cite{LLR}, $G$ is representable by a smooth separated $S$-group scheme. Consider the
canonical morphism of $S$-group schemes $r\colon G=(\Pic_{X/S},Y)^{0}\rightarrow \Pic^{0}_{X/S}$ (recalled in \S~\ref{Rappels Pic}),
which is surjective for the \'etale topology. Let $H$ be the schematic closure of $\ker(r_{K})\subset G_{K}$ in $G$. It is a flat $S$-group
scheme of finite type with smooth generic fibre, which is also the kernel of the canonical morphism $\theta\colon G\rightarrow J$
(composition of $r\colon G\rightarrow \Pic^{0}_{X/S}$ and the epimorphism $q\colon \Pic^{0}_{X/S}\rightarrow J$; see
for example \cite{Raynaud}, 4.1, for the fact that $\ker(\theta)=H$). Since $S$ is strictly local, the morphism $r$ induces
a surjective map (still denoted by $r$) between the $S$-sections:
\[
r\colon G(S)\longrightarrow \Pic^{0}(X).
\]
Let $\mathfrak{L}$ be the (rigidified) Poincar\'e sheaf on $X\times G$.
For each $n\in {\mathbb Z}_{\geq 1}$, we denote by
\[
\xymatrix{r_n\colon G(S)\ar[r]^{r} & \Pic^{0}(X)\ar[r] & \Pic^{0}(X_{\psi(n)})}
\]
the composition of maps which sends $\varepsilon \in G(S)$ to $\mathcal{L}_{\varepsilon}|_{X_{\psi(n)}}\in
\Pic^{0}(X_{\psi(n)})$, where $\mathcal{L}_{\varepsilon}$ is the sheaf $(\mathrm{id}_{X}\times \varepsilon)^{\ast}\mathfrak{L}$.
These maps are all surjective since $\OO_K$ is strictly henselian.

Let $p_G\colon X\times_S G\rightarrow G$ be the projection onto the second factor, and consider the object
$\mathrm{R}p_{G,\ast}\mathfrak L$ in the derived category of $\mathcal{O}_S$-modules.
It is well known that this complex is quasi-isomorphic to a perfect complex of perfect amplitude contained in $[0,1]$,
\emph{i.e.}, locally for the Zariski topology on $G$, $\R p_{G,\ast}\mathfrak{L}$ can be represented by a complex
\[
\xymatrix{\ldots\ar[r]& \mathcal{F}^{0}\ar[r]^{u}&
\mathcal{F}^{1}\ar[r]& \ldots}.
\]
with $\mathcal{F}^{i}$ ($i=0,1$) locally free $\OO_G$-modules of the same rank (\cite{EGA3}, 6.10.5).
The cokernel $\mathcal{M}$ of $u$ gives the $\mathcal{O}_{G}$-module $\mathrm{R}^{1}p_{G,\ast}\mathfrak{L}$, and for
any section $\varepsilon\colon S\rightarrow G$ of $G/S$, the pull-back $\varepsilon^{\ast}\mathcal{M}$ is given by the
cohomology group $\mathrm{H}^{1}(X,\mathcal{L}_{\varepsilon})$.
 On the other hand, for $L$ an invertible sheaf of degree $0$ on $X_{K}$, $\HH^{1}(X_{K},L)\neq 0$ if and only if $L\cong \OO_{X_K}$.
Therefore, the morphism $u$ above is injective and $\mathrm{det}(u)\neq 0$.
Hence $\mathcal{M}$ is a torsion $\OO_G$-module which admits a resolution of length $1$ by locally free $\OO_G$-modules.
In particular, $\mathcal{M}$ is Cohen-Macaulay, with support $\mathrm{Supp}(\mathcal{M})\subset G$ purely of codimension $1$ satisfying
the inclusion relations (of sets) $H\subset \mathrm{Supp}(\mathcal{M}) \subset H\cup G_s$.

\begin{lemma}\label{longueur 0} Let notation be as above. Then $\mathrm{Supp}(\mathcal{M})=H$ as sets.
\end{lemma}
\begin{proof}
Let $\xi$ be the generic point of $ G_s$, and $\ell$ the length of the $\OO_{G,\xi}$-module $\mathcal{M}_{\xi}$.
We will first prove by contradiction that $\ell=0$.
Suppose then $\ell\geq 1$.
Let $\varepsilon\in G(S)$ be a section of $G$, and let $\mathcal{L}_{\varepsilon}:=(\mathrm{id}_{X}\times \varepsilon)^{\ast}\mathfrak{L}$.
According to Proposition~\ref{intersection}, the $\OO_K$-module $\HH^{1}(X,\mathcal{L}_{\varepsilon})\cong \varepsilon^{\ast}\mathcal{M}$
is of length at least $ \ell\geq 1$.
By Corollary~\ref{cle}, this is equivalent to saying that $\mathcal{L}_{\varepsilon}|_{X_1}\cong \mathcal{I}^{i}|_{X_1}$ with $i$ a suitable integer.
This last fact implies that the surjective homomorphism
\[
r_1\colon G(S)\longrightarrow \Pic^{0}(X_1), ~~~~\varepsilon\mapsto
\mathcal{L}_{\varepsilon}|_{X_1}
\]
has finite image. However, this produces a contradiction since $k$ is algebraically closed and so $\Pic^{0}(X_1)\cong \Pic^{0}_{X_1/k}(k)$
is an infinite group. Therefore ${\mathcal M}_\xi=0$. In particular, $\xi\notin \mathrm{Supp}({\mathcal M})$, which completes the proof
of the Lemma since $\mathrm{Supp}(\mathcal{M})\subset G$ is purely of codimension $1$.
\end{proof}

Let us begin the comparison of the two filtrations defined in \S~\ref{fil} at the level $n=1$. Since $X_1/S$ can be defined
over the closed point $s$ of $S$, its Picard functor $\PP_{[1]}=\Pic^{0}_{X_1/S}$ can also be defined over $s$.
Hence, by adjunction, the morphism of functors $r_1\colon G\rightarrow \PP_{[1]}$ corresponds to a morphism of algebraic groups
over the closed points $s$ of $S$:
\[
r_{1,s}\colon G_s \longrightarrow
\PP_{[1],s}=\Pic^{0}_{X_1/k},
\]
which renders the following diagram commutative
\[
\xymatrix{G\ar[r]\ar[d] & \Pic_{X/S}^{0}\ar[r] & \PP_{[1]}=i_{\ast}\PP_{[1],s}\\
i_{\ast}G_s\ar@{.>}[rru]_{i_{\ast}r_{1,s}}
& & }
\]
Let $x\in G_s(k)$ be a closed point, $\varepsilon\in G(S)$ a section lifting $x$ and put
$\mathcal{L}_{\varepsilon}=(\mathrm{id}_{X}\times \varepsilon)^{\ast}\mathfrak{L}$.
This is a rigidified invertible sheaf on $X$.
Since $\mathrm{Supp}(\mathcal{M})=H$ (Lemma \ref{longueur 0}), one finds that $x=\varepsilon(s)\in H_s(k)$
if and only if the $\OO_K$-module $\varepsilon^{\ast}\mathcal{M}=\HH^{1}(X,\mathcal{L}_{\varepsilon})$ is of length $\geq 1$.
Moreover, in view of Lemma~\ref{cle} this last condition is equivalent to saying that
$\mathcal{L}_{\varepsilon}|_{X_1}\cong \mathcal{I}^{i}|_{X_1}$ for a suitable integer $i$.
In particular, the image $r_{1,s}(H_{s}(k))$ is a finite set of $\mathrm{P}_{[1],s}(k)$, and the kernel of $r_{1,s}$ is contained in $H_s$.
 Let $Z$ be the schematic closure of the set of those points $x\in G_s(k)$ which admit a lifting $\varepsilon\in G(S)$
such that $\mathcal{L}_{\varepsilon}|_{X_1}\cong \OO_{X_1}$. By continuity of $r_{1,s}$, the subgroup scheme $Z\subset G_s$ is a union of
irreducible components of ${H}_{s,\mathrm{red}}$. Denote by $G^{[1]}\rightarrow G$ the dilatation of $G$ with center $Z\subset G_s$.
By definition of $Z$, we have an exact sequence of smooth $k$-group schemes
\begin{equation}\label{eq.ZGPic}
0\longrightarrow Z\longrightarrow G_s\stackrel{r_1}{ \longrightarrow }
\Pic^{0}_{X_1/k}\longrightarrow 0.
\end{equation}
Moreover, according to the universal property of dilatations (\cite{BLR}, 3.2/1), we have the following short exact sequence of abstract groups:
\begin{equation}\label{exact de P}
0\longrightarrow G^{[1]}(S)\longrightarrow G(S)\longrightarrow
\Pic^{0}(X_1)\longrightarrow 0.
\end{equation}

Consider now the morphism $\theta\colon G\rightarrow J$. Denote by $G^{[1]'}$ the dilatation of $G$ with center
$H_{s,\mathrm{red}}=\mathrm{ker}(\theta)_{s,\mathrm{red}}\hookrightarrow G_s$.
Since $H_s$ is the kernel of the canonical map $\theta_s\colon G_s\rightarrow J_s$, the universal property of dilatations
implies that the following sequence is exact
\begin{equation}\label{exact de J}
0\longrightarrow G^{[1]'}(S)\longrightarrow G(S)\longrightarrow
J(S_1)\longrightarrow 0.
\end{equation}
Since $Z\subset H_{s,\mathrm{red}}$ is an open subgroup, $G^{[1]}$ is an open subgroup of $G^{[1]'}$.
From the exact sequences \eqref{exact de P}, \eqref{exact de J}, we obtain a morphism of groups $q_1\colon \Pic^{0}(X_1)\rightarrow J(S_1)$
which makes the external square commute:
\[
\xymatrix{G(S)\ar[r]^{r}\ar@{=}[d]& \Pic^0(X)\ar[d]^{q}\ar[r] & \Pic^{0}(X_1)\ar[d]^{q_1}\\
G(S)\ar[r]^{\theta}& J(S)\ar[r]\ar[r] & J(S_1)}
\]
Moreover, from the fact that $q\circ r=\theta$ and the surjectivity of $r$, the square on the right also commutes.
The morphism of abstract groups $q_1$ is surjective (since all the other maps are), and has kernel generated
by $\mathcal{I}|_{X_1}\in \Pic(X_1)$ (see Corollary~\ref{cle}).
Note that the diagram above fits also into a bigger commutative diagram
\begin{eqnarray}\label{diagram for level 1}
\xymatrix{ G^{[1]}(S)\ar@{->>}[rd]\ar[rdd]\ar@{^(->}[rr] & & G[S]\ar@{->>}[rd]^r\ar@{->>}[rdd]_<<<<<{\theta}\ar@{->>}[rr] & & \Pic^{0}(X_1)\ar@{=}[rd]\ar@{->>}[rdd]_<<<<<{q_1} & \\ & \PP^{[1]}(S)\ar@{^(->}[rr]\ar@{.>}[d]^{\exists ~ q^{[1]}}& & \Pic^{0}(X)\ar@{->>}[rr]\ar@{->>}[d]^{q}& & \Pic^{0}(X_1)\ar@{->>}[d]^{q_1} \\ & J^{[1]}(S)\ar@{^(->}[rr]& & J(S)\ar@{->>}[rr]& & J(S_1)}
\end{eqnarray}
On the level of pro-algebraic groups we have then shown that the morphism $q$ in \eqref{eq.morproalg} induces a map
\begin{equation}\label{eq.morq1}
q_1\colon \Gr( \Pic^{0}_{X_{\psi(1)}/S})\longrightarrow \Gr_1(J)
\end{equation}
because, as we noted in \S~\ref{sec.compfil}, it is sufficient to check the factorization on $k$-points.

In order to proceed with the comparison of the filtrations for higher $n$, let us denote by $\mathcal{M}^{[1]}$ (respectively
by $\mathcal{M}^{[1]'}$) the inverse image of $\mathcal{M}$ over $G^{[1]}$ (respectively over $G^{[1]'}$) via
the morphism $G^{[1]}\rightarrow G$ (respectively via the morphism $G^{[1]'}\rightarrow G$).
Let $H^{[1]}$ (respectively $H^{[1]'}$) be the schematic closure of $H_{K}\hookrightarrow G_{K}^{[1]}=G_{K}$ in $G^{[1]}$ (respectively in $G^{[1]'}$).
Then $\mathcal{M}^{[1]}$ (respectively $\mathcal{M}^{[1]'}$) is a coherent torsion sheaf with support in $H^{[1]}\cup G^{[1]}_s$
(respectively in $H^{[1]'}\cup G^{[1]'}_s$), which admits a resolution of length $1$ by locally free $\mathcal{O}_{G^{[1]}}$-modules
(respectively $\mathcal{O}_{G^{[1]'}}$-modules).
In particular, since the schemes $G^{[1]}$ and $G^{[1]'}$ are regular, $\mathcal{M}^{[1]}$ and $\mathcal{M}^{[1]'}$ are Cohen-Macaulay as modules.
On the other hand, by the universal property of dilatations, the composed morphism $G^{[1]}\rightarrow G\rightarrow J$
(respectively $G^{[1]'}\rightarrow G\rightarrow J$) factors through $J^{[1]}\rightarrow J$. We denote by
$\theta^{[1]}\colon G^{[1]}\rightarrow J^{[1]}$ (respectively by $\theta^{[1]'}\colon G^{[1]'}\rightarrow J^{[1]}$) the morphism obtained in this way.

\begin{lemma}\label{pf rec}
\begin{itemize}
\item[(i)] Let $\xi_1'$ be a generic point of $G^{[1]'}_s$, then the $\OO_{G^{[1]'},\xi_1'}$-module $\mathcal{M}^{[1]'}_{\xi_1'}$ is of length $1$.
\item[(ii)] The scheme $H$ is normal.
\item[(iii)]The morphism $\theta^{[1]}\colon G^{[1]}\rightarrow J^{[1]}$ induces a surjection $G^{[1]}(S)\rightarrow J^{[1]}(S)$.
In particular, $\theta^{[1]}$ is a faithfully flat morphism of $S$-group schemes, with $\ker(\theta^{[1]})=H^{[1]}$.
\end{itemize}
\end{lemma}

\begin{proof} Observe first that we have a commutative diagram with exact rows, where the first row is \eqref{exact de P}:
\begin{equation*}\label{dia.GPic}
\xymatrix{0\ar[r]& G^{[1]}(S)\ar[r]\ar[d]& G(S)\ar[r]\ar[d]^{r_2} &
\Pic^{0}(X_1)\ar[r]\ar@{=}[d] & 0 \\ 0\ar[r] &
\PP^{[1,\psi(2)]}(S)\ar[r]& \PP_{[\psi(2)]}(S)\ar[r]&
\PP_{[1]}(S)\ar[r]& 0 }
\end{equation*}
The morphism $G^{[1]}(S)\rightarrow \PP^{[1,\psi(2)]}(S)$ is surjective, since the map $r_2$ is surjective. Moreover, by Corollary \ref{Pic_S},
 the group $\PP^{[1,\psi(2)]}(S)$ is an $\OO_K$-module of length $1$. Hence, it is an infinite group. Therefore, the composed morphism
\[
G^{[1]}(S)\longrightarrow G(S)\longrightarrow \PP_{[\psi(2)]}(S)=\Pic^{0}({X_{\psi(2)}}),
\]
has infinite image. Hence, the composed morphism
\begin{equation}\label{eq.GGPS2}
G^{[1]'}(S)\longrightarrow G(S)\longrightarrow \PP_{[\psi(2)]}(S)=\Pic^{0}({X_{\psi(2)}})
\end{equation}
also has infinite image because $G^{[1]}$ is an open subgroup of $G^{[1]'}$.

Next, we consider the map of functors $r_{2}'\colon G^{[1]'}\rightarrow \mathrm{P}_{[\psi(2)]}=\Pic_{X_{\psi(2)}/S}^{0}$
obtained as the composition of $G^{[1]'}\rightarrow G$ with $r_{2}\colon G\rightarrow \mathrm{P}_{[\psi(2)]}$.
This map induces a morphism of pro-algebraic groups over $k$, again denoted by $r_2'$:
\[
r_2'\colon \Gr(G^{[1]'})\longrightarrow \Gr(\mathrm{P}_{[\psi(2)]}).
\]
We claim that this morphism factors through $G_{s}^{[1]'}=\Gr_{1}(G^{[1]'})$:
\begin{equation}\label{dia.GrGP2}
\xymatrix{\Gr(G^{[1]'})\ar[r]^{r_2'}\ar[d] & \Gr(\mathrm{P}_{[\psi(2)]}) \\ G^{[1]'}_{s}\ar@{.>}[ru]_{\exists~ r_{2,s}'} & }
\end{equation}
Let $\mathcal{K}$ be the kernel of the morphism $\mathrm{Gr}(G^{[1]'})\rightarrow G_{s}^{[1]'}$; it is pro-smooth and connected.
Let $\varepsilon\in G^{[1]'}(S)$ be such that $\varepsilon(s)\in G^{[1]'}_{s}$ is the unit element. In particular, the support of
the torsion module $\mathcal{M}^{[1]'}$ at $\varepsilon(s)$ has two irreducible components, which implies, according to
Proposition \ref{intersection} (1), that the $\mathcal{O}_K$-module $\varepsilon^{\ast}\mathcal{M}^{[1]'}=\mathrm{H}^1(X,\mathcal{L}_{\varepsilon})$
has length at least $2$ (here, $\mathcal{L}_{\varepsilon}=(\mathrm{id}\times \varepsilon)^{\ast}\left(\mathfrak{L}|_{X\times G^{[1]'}}\right)$).
In particular, by Corollary~\ref{cle}, the restriction $\mathcal{L}_{\varepsilon}|_{X_{\psi(2)}}$ is a power of $\mathcal{I}|_{X_{\psi(2)}}$.
Since the invertible sheaf $\mathcal{I}$ is of finite order and since $k$ is algebraically closed, this implies that the
induced map $\mathcal{K}\rightarrow \Gr(\mathrm{P}_{[\psi(2)]})$ has finite image, in particular, it is trivial since $\mathcal K$
is pro-smooth and connected. This fact ensures the existence of the factorization in \eqref{dia.GrGP2}.

In order to prove (i), let us denote by $\ell_1$ the length of the $\OO_{G^{[1]'},\xi_1'}$-module $\mathcal{M}^{[1]'}_{\xi_1'}$.
By definition of $G^{[1]'}$, we have $\ell_1\geq 1$. Suppose $\ell_1\geq 2$. Let $\varepsilon\in G^{[1]'}(S)$ be a section of $G^{[1]'}$
such that $\varepsilon(s)\in \overline{\{\xi_1'\}}\subset G^{[1]'}_s$ and denote by $\mathcal{L}_{\varepsilon}$
the associated rigidified invertible sheaf on $X$. According to Proposition~\ref{intersection} (1), the
$\OO_K$-module $\varepsilon^{\ast}\mathcal{M}^{[1]'}\cong \HH^{1}(X,\mathcal{L}_{\varepsilon})$ is of length $\geq \ell_1\geq 2$.
Hence, by Corollary~\ref{cle}, we have $\mathcal{L}_{\varepsilon}|_{X_{\psi(2)}}\cong \mathcal{I}^{i}|_{X_{\psi(2)}}$ for a suitable integer $i$.
Thus, $r_{2,s}'(\overline{\{\xi_{1}'\}})\subset \Gr(\mathrm{P}_{[\psi(2)]})$ consists of a single element because $r_{2,s}'$ is continuous
and the set $\{\mathcal{I}^{j}:j\in\mathbb{Z}\}$ is finite. Using the fact that $r_{2,s}'$ is a morphism of groups and
that $\overline{\{\xi_{1}' \}}$ is an irreducible component of $G^{[1]'}_{s}$, we deduce that the morphism $r_{2,s}'$ and
hence the map $r_2'\colon G^{[1]'}(S)\rightarrow \PP_{[\psi(2)]}(S)=\Pic^{0}(X_{\psi(2)})$ has finite image.
This contradicts the assertion on the infinity of the image of \eqref{eq.GGPS2} proved above. Hence $\ell_1=1$, and
this concludes the proof of (i).

Assertion (ii) is just a corollary of (i). In fact, for $Y_1$ an irreducible component of $H_s$, let $\xi_1'$ denote the
generic point of $G_{s}^{[1]'}$ lying above $Y_1$.
Let $x'\in \overline{\{\xi_1'\}}\subset G^{[1]'}_{s}$ be a closed point not contained in $H^{[1]'}_{s}$, and
$\varepsilon'\colon S\rightarrow G^{[1]'}$ a section lifting $x'$, which also gives a section $\varepsilon\colon S\rightarrow G$
by composition with $G^{[1]'}\to G$.
 Let $x=\varepsilon(s)\in H_{s}$.
Since $x'\notin H^{[1]'}_{s}$, and $\ell(\mathcal{M}^{[1]'}_{\xi_{1}'})=1$ by assertion (i) of this Lemma, the $\mathcal{O}_K$-module
$\varepsilon'^{\ast}\mathcal{M}^{[1]'}=\varepsilon^{\ast}\mathcal{M}$ is of length $1$ (see Proposition~\ref{intersection} (1)).
According to Proposition~\ref{intersection} (2), this last condition implies that $H$ is regular at $x$.
 Hence $H$ is regular at the generic point of the irreducible component $Y_1$ of $H_s$ because $Y_1$ contains $x$.
Since this can be done for any generic point of $H_s$, one finds that $H$ is normal by using Serre's criterion of normality
(recall that the generic fibre $H_K$ of $H$ is regular, and the scheme $H$, being a divisor of a regular scheme, is Cohen-Macaulay).

For (iii), recall that the composed morphism $G(S)\rightarrow \Pic^{0}(X)\rightarrow J(S)$ is surjective (see \S~\ref{Rappels Pic}).
Since $G^{[1]'}$ is the dilatation of $G$ along $H_{s,\mathrm{red}}$, the surjectivity of the last map implies that the
map $G^{[1]'}(S)\rightarrow J^{[1]}(S)$ is also surjective.
Since $G^{[1]}\subset G^{[1]'}$ is an open subgroup, with non-empty special fibre, and the abstract group $G^{[1]'}(S)/G^{[1]}(S)$ is a
finite group, according to \cite{BLR}, 9.2/6, the morphism $\theta^{[1]}$ also induces a surjection $G^{[1]}(S)\rightarrow J^{[1]}(S)$.
In particular, using the fact that the two $S$-group schemes $G^{[1]}$ and $J^{[1]}$ are smooth, we find that the morphism $\theta^{[1]}$
is faithfully flat, and hence $\ker(\theta^{[1]})=H^{[1]}$ since both $\ker(\theta^{[1]})$ and $H^{[1]}$ are flat closed
subgroup schemes of $G^{[1]}$ having the same generic fibre.
\end{proof}

By abuse of notation, let us denote by $r_2$ both the following composition of morphisms
\begin{equation*}\label{eq.beta2}
\xymatrix{r_2\colon G^{[1]}\ar[r] & G\ar[r] & \PP_{[\psi(2)]}}
\end{equation*}
and the induced morphism of pro-algebraic groups over $k$:
$\mathrm{Gr}(G^{[1]})\rightarrow \mathrm{Gr}(\PP_{[\psi(2)]})$.
As we have seen in the proof of Lemma~\ref{pf rec}, diagram \eqref{dia.GrGP2}, (since $G^{[1]}\subset G^{[1]'}$ is an open subgroup),
the map $r_2$ factors through the canonical surjection $\Gr(G^{[1]})\rightarrow G^{[1]}_{s}$:
\[
\xymatrix{\mathrm{Gr}(G^{[1]})\ar[r]^{r_2}\ar[d]& \mathrm{Gr}(\PP_{[\psi(2)]}) \\ G^{[1]}_s\ar@{.>}[ru]_{ \exists~ r_{2,s}}
}
\]
 Next, define $Z_1:=\ker(r_{2,s})_{\mathrm{red}}\hookrightarrow
G^{[1]}_s$.
Then the same argument used for $Z$ in \eqref{eq.ZGPic}, in the case $n=1$, implies that $Z_1$ is a union of connected components of
$H^{[1]}_{s,\mathrm{red}}$.

Now we use constructions similar to those used in the comparison at the first level.
Let $G^{[2]}$ (respectively $G^{[2]'}$) be the dilatation of $G^{[1]}$ along the closed smooth subgroup $Z_1$ of $G^{[1]}_s$
(respectively along $H^{[1]}_{s,\mathrm{red}}\hookrightarrow G^{[1]}_s$), and let $\alpha^{[1]}$ be the
composed morphism $G^{[2]}\rightarrow G^{[1]}\rightarrow G$.
According to \cite{BLR}, 3.2/3, $G^{[2]}$ is a smooth $S$-group scheme, and we have an exact sequence:
\begin{equation*}
0\longrightarrow G^{[2]}(S)\longrightarrow G^{[1]}(S)\longrightarrow
\PP^{[1,\psi(2)]}(S)\longrightarrow 0.
\end{equation*}
On the other hand, since $H^{[1]}_s$ is the kernel of the morphism $G^{[1]}_s\rightarrow J^{[1]}_s$, according to the universal property
of dilatations, we have an exact sequence of abstract groups
\[
0\longrightarrow G^{[2]'}(S)\longrightarrow G^{[1]}(S)\longrightarrow
J^{[1]}(S_1)\longrightarrow 0.
\]
Since $Z_1\subset H^{[1]}_{s,\mathrm{red}}$ is an open subgroup scheme, $G^{[2]}\subset G^{[2]'}$ is an open subgroup.
Hence we obtain a morphism $\alpha\colon \PP^{[1,\psi(2)]}(S)\rightarrow J^{[1]}(S_1)$ which makes the following diagram
\begin{eqnarray}\label{level 2}
\xymatrix{G^{[1]}(S)\ar@{->>}[r]\ar@{->>}[rd]& \PP^{[1]}(S)\ar[r] \ar[d]^{q^{[1]}}& \PP^{[1,\psi(2)]}(S)\ar[d]^{\alpha} \\
&J^{[1]}(S)\ar[r] & J^{[1]}(S_1)}
\end{eqnarray}
commute. Hence we get a diagram with exact rows,
\[
\xymatrix{\PP^{[1]}(S)\ar@{^(->}[drr]\ar[ddr]^{\beta'}\ar@/_1pc/[ddd]_<<<<<<<{q^{[1]}} & &&&&&\\
&& \Pic^{0}(X)\ar@{->>}[dd]_<<<<<<{~~q}\ar@{->>}[dr]_{\beta}\ar@{->>}[drrr]^{|_{X_1}}&&&&\\
0\ar[r]&\PP^{[1,\psi(2)]}(S)\ar[rr]\ar[dd]^<<<<<{\alpha}&&\Pic^{0}(X_{\psi(2)}) \ar@{.>>}[dd]^<<<<<<{q_2}\ar[rr]&&\Pic^{0}(X_{1})\ar[ld]^ {q_1}\ar[dd]^{q_1}\ar[r]&0\\
 J^{[1]}(S)\ar@{^(->}[rr]\ar[rd]& &J(S)\ar@{->>}[dr]_{\gamma}\ar@{->>}[rr]&&J(S_1)\ar@{=}[rd]&&\\
0\ar[r]& J^{[1]}(S_1)\ar[rr]_{\varrho_1} && J(S_2)\ar[rr]&& J(S_1)\ar[r]&0.
}
\]
This diagram (without the existence of $q_2$) is seen to be commutative by combining the commutativity of
diagrams \eqref{diagram for J}, \eqref{diagram for level 1} and \eqref{level 2}.
In order to see that $q_2$ exists, we need only show that $\ker(\beta)\subset \ker(\gamma\circ q )$.
Since the upper horizontal sequence is the push-out of the {\textquotedblleft diagonal\textquotedblright} exact sequence
along $\beta'$, we have $\ker(\beta)=\ker(\beta')$. Hence, to complete the proof, it is sufficient to recognize that the two maps
$\PP^{[1]}(S)\to J(S_2)$, obtained, one following the path through $\Pic^0(X)$ and the other via $\PP^{[1,\psi(2)]}(S)$, coincide.
This fact can be easily checked by diagram chasing.
In this way, we have shown the existence of the morphism $q_2$.

Moreover, by Corollary~\ref{cle}, the kernel of $q_2\colon \Pic^{0}(X_{\psi(2)})\rightarrow J(S_2)$ is generated by
$\mathcal{I}|_{X_{\psi(2)}}\in \Pic^{0}(X_{\psi(2)})$. In particular the morphism $q$ in \eqref{eq.morproalg} induces
a morphism $q_2\colon \Gr(\Pic^{0}_{X_{\psi(2)}/S})\rightarrow \Gr_2(J)$ that is compatible with the morphism $q_1$ in \eqref{eq.morq1}.

The general case can be done by induction on $n$, by using the same argument as before.
Finally, we summarize our results in the following theorem:

\begin{theorem}\label{resultat final} Let $X_K$ be a $K$-torsor under an elliptic curve and $X$ its $S$-proper regular minimal model. Let $J$ be the identity component of the N\'eron model over $S$ of the jacobian $\mathrm{Pic}_{X_K/K}^0$ of $X_K$. Then for any integer $n\geq 1$, the surjective morphism of fppf-sheaves $q\colon \Pic^{0}_{X/S}\rightarrow J$ in \eqref{eq.SurPiczero} induces a morphism of smooth $k$-group schemes
\[
q_n\colon \mathrm{Gr}(\Pic^0_{X_{\psi(n)}/S})\longrightarrow \mathrm{Gr}_n(J)
\]
making the obvious diagram commute.
Moreover, the morphism $q_n$ defined above is an isogeny of algebraic $k$-groups, and the group of $k$-points of the kernel of $q_n$ is given by
\[
\ker(q_n)(k)=\{\mathcal{I}^{i}|_{X_{\psi(n)}}: ~i\in
{\mathbb Z}~\}\subset \Gr(\PP_{[\psi(n)]})(k)\cong
\Pic^{0}(X_{\psi(n)}).
\]
\end{theorem}

\begin{corollary}\label{cor.finalresult} The morphism $\bm{ q}\colon\bm{ \Pic^{0}(X)}\longrightarrow \bm{J(S)}$
in \eqref{q pro-alg} maps $\bm{\PP^{[\psi(n)]}(S)}$ onto $\bm{J^{[n]}(S)}$, thus inducing an isogeny of connected quasi-algebraic groups
\[
\bm{q_{n}}\colon \GGr(\PP_{[\psi(n)]})\longrightarrow \bm{\Gr_n}(J)
\]
whose kernel is generated by the element $\mathcal{I}|_{X_{\psi(n)}}\in \GGr(\PP_{[\psi(n)]})(k)=\Pic^{0}(X_{\psi(n)})$. Furthermore $\bm q$ is an epimorphism in the abelian category of Serre pro-algebraic groups (in particular it is surjective on $k$-sections),  and the kernel of $\bm q$ is isomorphic to $\Z/d\Z$. 
\end{corollary}
\begin{proof} According to the previous theorem, we have $\bm q(\bm{\PP^{[\psi(n)]}(S)})\subset \bm{J^{[n]}(S)}$, hence $\bm q$
induces the isogeny $\bm {q_n}$ with properties as stated in the corollary. In particular, being a projective limit of surjective morphisms of quasi-algebraic groups, $\bm q$ is an epimorphism in the category of Serre pro-algebraic groups.
In order to finish the proof, we need only establish that the last inclusion is in fact an equality. To see this, we consider the
quotient of $\bm{J^{[n]}(S)}$ by $\bm {q}(\bm{\PP^{[\psi(n)]}(S)})$.
By applying the snake Lemma to the following diagram with exact rows
\[
\xymatrix{0\ar[r]& \bm{\PP^{[\psi(n)]}(S)}\ar[r]\ar[d] & \bm{\Pic^0(X)}\ar[r]\ar[d]^{\bm{q}}& \GGr(\PP_{[\psi(n)]})\ar[r]\ar[d]^{\bm{q_n}}& 0 \\0\ar[r]& \bm{J^{[n]}(S)}\ar[r] & \bm{J(S)}\ar[r]& \GGr_n(J)\ar[r]& 0 }
\]
we find that the quotient $\bm{J^{[n]}(S)}$ by $\bm{\PP^{[\psi(n)]}(S)}$ is a finite pro-algebraic group. Since $J^{[n]}/S$
has connected fibres, the pro-algebraic group $\bm{J^{[n]}(S)}$ is connected, hence the cokernel of the left vertical arrow
is necessarily trivial. Thus $\bm q(\bm{\PP^{[\psi(n)]}(S)})= \bm{J^{[n]}(S)}$. The kernel of $\bm q$ is cyclic of order $d$
because the kernel of any $\bm q_n$ is a constant finite group and the kernel of $q$ in \eqref{eq.qP} is isomorphic to $\Z/d\Z$
and is generated by $\mathcal I$ (\cite{Raynaud}, Th\'eor\`eme~6.4.1~(3)).
\end{proof}

\begin{corollary}\label{coro.tametorsor2} Let $X_K/K$ be a torsor under an elliptic curve, and $X/S$ its proper regular minimal model. The following conditions are equivalent:
\begin{enumerate}
\item[(i)] The $S$-scheme $X/S$ is cohomologically flat in dimension $0$.
\item[(ii)] The Picard functor $\Pic^{0}_{X/S}$ is representable, and the canonical map $\Pic^{0}_{X/S}\rightarrow J$ is \'etale.
\item[(iii)] The extension of Serre pro-algebraic groups associated to $X/S$ deduced from \ref{cor.finalresult}
 \[
 0\longrightarrow \mathbb{Z}/d\mathbb{Z}\longrightarrow \bm{\Pic^{0}(X)}\stackrel{\bm q}\longrightarrow \bm{J(S)}\longrightarrow 0
 \]
 lies in the subgroup $\mathrm{Ext}^{1}(\bm{\Gr}_{1}(J),\mathbb{Z}/d\mathbb{Z})\subset \mathrm{Ext}^{1}(\bm{J(S)},\mathbb{Z}/d\mathbb{Z})$.
\end{enumerate}
\end{corollary}
\begin{proof} The equivalence between (i) and (ii) follows from Proposition~5.2 of \cite{Raynaud} and from Corollary~\ref{coro.tametorsor}. To see (i)$\Leftrightarrow$ (iii), suppose first that $X/S$ is cohomologically flat in dimension $0$, namely $\mathcal{I}|_{X_1}$ is of order $d$ (Corollary~\ref{coro.tametorsor}); we then have the following commutative diagram
\[
\xymatrix{0\ar[r]& \mathbb{Z}/d\mathbb{Z}\ar[r]\ar@{=}[d] & \bm{\Pic^0(X)}\ar[r]^{\bm{q}}\ar[d]& \bm{J(S)}\ar[r]\ar[d]& 0 \\0\ar[r]& \mathbb{Z}/d\mathbb{Z}\ar[r] & \GGr(\mathrm{P}_{[1]})\ar[r]^{\bm{q_1}}& \GGr_1(J)\ar[r]& 0 }
\]
In particular, we get (iii). Conversely, if condition (iii) holds, the morphism $\bm{q}$ induces an isomorphism between $\bm{q}^{-1}(\bm{J^{[1]}(S)})$ and $\bm{J^{[1]}(S)}=\ker(\bm{J(S)}\rightarrow \GGr_{1}(J))$. In particular, Serre pro-algebraic group $\bm{q}^{-1}(\bm{J^{[1]}(S)})$ is connected. On the other hand, $\bm{q}^{-1}(\bm{J^{[1]}(S)})$ contains the subgroup $\GGr(\mathrm{P}^{[1]})$ of index $d/d'$ with $d'$ the order of $\mathcal{I}|_{X_1}$. As a result, we find $\bm{q}^{-1}(\bm{J^{[1]}(S)})=\GGr(\mathrm{P}^{[1]})$ by the connectedness of $\bm{J^{[1]}(S)}$. Therefore, $d=d'$, and $\mathcal{I}|_{X_1}$ is of order $d$, hence $X/S$ is cohomologically flat in dimension $0$ (Corollary~\ref{coro.tametorsor}).
\end{proof}

We record the following fact, which will not be used in the rest of this paper.

\begin{remark} For each integer $n\geq 1$, let $N^{[n]}$ be the kernel of the morphism of $S$-group schemes $\theta^{[n]}\colon G^{[n]}\rightarrow J^{[n]}$ obtained inductively by a sequence of dilatations of $\theta\colon G\rightarrow J$ in the proof of Theorem~\ref{resultat final} (and its omitted induction steps). The proof of Theorem~\ref{resultat final} (especially of Lemma~\ref{pf rec} (ii)) shows that the scheme $N^{[n]}$ is normal. Moreover, one can verify that the scheme $N^{[n]}$ is smooth over $S$ for sufficiently large $n$.
\end{remark}

\section{Shafarevich's pairing} \label{ShafaPairing}
{Let $A_K$ be an abelian variety over $K$.} We construct in this section via the rigidified Picard functor a homomorphism $\Xi'\colon {\rm H}^1(K,A_K)\to{\rm Hom} (\pi_1(\GGr(A^\prime)), \Q/\Z)$ which coincides with the restriction of \eqref{eq.sha} to the $n$-parts when $n\in {\Z}_{>0}$ is prime to $p$, and more generally, for all positive integers $n$ in the mixed characteristic case (Theorem \ref{thm.main}).
In Section \ref{sec.comparison}, we will use these constructions to study the morphism $\Phi_d$ in \eqref{eq.phid}.
All group schemes we will work with are assumed to be commutative.

\subsection{The component group of a torus}\label{sec.comptori}
One of the key facts in the construction of Shafarevich's duality is the pro-algebraic structure of the cohomology group ${\HH}^1_\fl(K,\mu_n)$, where $\mu_n$ denotes the finite subgroup scheme of $n$-th roots of unity in the multiplicative group $\G_{m,K}$.
 We first recall this construction. Observe that the N\'eron model $T$ over $S=\spec(\OO_K)$ of a torus $T_K$ is locally of finite type over $S$, but, in general, not of finite type over $S$. It is of finite type over $S$ if and only if $T_K$ does not contain split tori (cf. \cite{BLR}, 10.2/1).
Let $\Lambda_K$ denote the character group of $T_K$.
If $T_K$ has no non-trivial split quotient, \emph{i.e.}, if $\Lambda_K(K)=\{0\}$, then $T$ is of finite type.

\begin{lemma}\label{lem.comptori}
There exists a functorial construction that associates to a finite multiplicative group scheme $F_K$ a Serre pro-algebraic group $\bm{{\rm H}^1(K, F_K)}$ whose group of $k$-rational points is isomorphic to ${\rm H}^1 _\fl(K, F_K)$.
\end{lemma}
\proof (Cf. \cite{Beg}, 4.3.) Let $f\colon T_{1,K}\to T_{2,K}$ be an isogeny of tori with kernel $F_K$. Let $\Lambda_{i,K}$ (respectively $T_i$) be the character group (respectively the N\'eron model) of $ T_{i,K}$, $i=1,2$.
Let $T_{1,K}^{(d)}$ denote the torus (\emph{d\'eploy\'e}) whose character group is the constant free group $\Lambda_{1,K}(K)$, and similarly for $T_{2,K}^{(d)}$. They are split tori with isomorphic component groups, say $\Z^r$. Furthermore the isogeny $f$ induces an isogeny $f^{(d)}\colon T_{1,K}^{(d)}\to T_{2,K}^{(d)} $ that is injective on component groups.
The torus $T_{i,K}'$, defined as the kernel of the quotient map $ T_{i,K}\to T_{i,K}^{(d)}$, admits a N\'eron model of finite type over $S$
because its character group is $\Lambda_{i,K}'\cong \Lambda_{i,K}/ \Lambda_{i,K}(K)$.
Since $K$ is a $(C_1)$-field, tori are cohomologically trivial. Hence there is an isomorphism $T_{i,K}(K)/T_{i,K}'(K) \stackrel{\sim}{\to} T_{i,K}^{(d)}(K)$ and the complexes of component groups
\[\pi_0( T_i')\longrightarrow \pi_0(T_i)\longrightarrow \pi_0(T_i^{(d)})\longrightarrow 0,\quad i=1,2, \] are exact. One deduces from this fact that the kernel and the cokernel of the homomorphism $ \pi_0(T_1)\to \pi_0(T_2)$ are finite groups.

The identity components of the N\'eron models $T_i$ are smooth group schemes of finite type (\cite{BLR}, 10.1). Hence their perfect Greenberg realizations are Serre pro-algebraic groups.
Let $\bm P$ denote the cokernel of the map $\GGr(T_1^0)\to \GGr(T_2^0)$.
Now, the cokernel $\bm{{\rm H}^1(K,F_K)}$ of the map $\GGr(T_1)\to \GGr(T_2)$ is an extension of the finite group $\pi_0(T_2)/ \pi_0(T_1)$ by the quotient of $\bm P$ by a finite constant group, hence it is a Serre pro-algebraic group. By construction the group of $k$-points of $\bm{{\rm H}^1(K,F_K)}$ is ${\rm H}^1_\fl(K,F_K)$. Furthermore, the pro-algebraic group $\bm{{\rm H}^1(K,F_K)}$ does not depend on the isogeny $f$ (\cite{Beg}, 4.3 (b)).

For the functoriality, consider a morphism of finite multiplicative group schemes $g\colon F_K\to F_K'$ and let $f\colon T_{1,K}\to T_{2,K}$ be an isogeny of tori with kernel $F_K$. Then $F_K'$ embeds in $T_{1,K}':=T_{1,K}/\ker(g)$ and we have an isogeny of tori $f'\colon T_{1,K}'\to T_{2,K}$ with kernel $F_K'$. Since $f$ factors through $f'$, we get a morphism $\bm{{\rm H}^1(K, F_K)}\to \bm{{\rm H}^1(K, F_K')}$.
\qed

For our later work we will also need the following result:

\begin{lemma}\label{lem.torustor}
Let $0\to T_K\to G_K\to A_K\to 0$ be an extension of an abelian variety $A_K$ by a torus $T_K$. Let $A$, $G$, $T$, be the N\'eron models of $A_K$, $G_K$, $T_K$, respectively, over $S$. Then the above sequence induces a homomorphism of profinite groups $\pi_1(\GGr (A)) \to \pi_0(T)_\tor$ where the index $\tor$ indicates the torsion subgroup.
\end{lemma}
\proof The sequence of N\'eron models of the above extension is exact on $S$-sections; indeed, by the universal property of N\'eron models, $T_K(K)=T(S), G_K(K)=G(S)$, $A_K(K)=A(S)$, and $T_K$ is cohomologically trivial. Hence the sequence of N\'eron models provides an extension of perfect $k$-schemes
$0\to \GGr(T)\to \GGr(G)\to\GGr(A)\to 0$.
If $T$ is of finite type, $\pi_0(T)_\tor=\pi_0(T)\cong \pi_0(\GGr(T))$ is finite. Furthermore the above sequence is a sequence of Serre pro-algebraic groups, and the desired map comes from the long exact sequence of the $\pi_i$'s (see \S~\ref{sec.proalg-gree},(iii)).

Suppose that $T$ is locally of finite type. Since $k$ is algebraically closed, $\pi_0(T)=\pi_0(T)_\tor\oplus \pi_0(T)_\fr$
with $\pi_0(T)_\fr$ torsion-free.
Let $T^\ft$ be the maximal subgroup of $T$ whose component group is finite. In particular, $T^\ft$ contains the identity component $T^0$ and $\pi_0(T^\ft){\stackrel{\sim}{\to} } \pi_0(T)_\tor$. Then the sequence we started with extends to a sequence $0\to T^\ft\to \tilde G\to A^0\to 0$ that is exact on $S$-sections because all extensions of $A^0$ by $T/T^\ft$ are trivial (\cite{SGA7}, \S~5.7, 5.5). Then we have an exact sequence of Serre pro-algebraic groups $ 0\to \GGr(T^\ft)\to \GGr(\tilde G)\to \GGr(A^0)\to 0$. By the long exact sequence of the $\pi_i$'s we get then a map $ \pi_1(\GGr (A^0))\to  \pi_0(\GGr(T^\ft))$. The conclusion follows using the canonical isomorphisms $  \pi_1(\GGr (A^0)) \stackrel{\sim}{\to} \pi_1(\GGr (A))$ and 
$\pi_0(\GGr(T^\ft))\stackrel{\sim}{\to} \pi_0(T)_\tor$.
\qed
\medskip

{In the case $F_K=\mu_{n}$ we can describe explicitly the component group of $\bm{{\rm H}^1(K, F_K)}$.
\begin{lemma}\label{comp-mun}  The component group of $\bm{{\rm H}^1(K, \mu_n)}$ is canonically isomorphic to $ \Z/n\Z$.
\end{lemma}
\begin{proof} Using the Kummer sequence $0\to \mu_n\to \G_{m,K}\to \G_{m,K}\to 0$, we get that $\bm{{\rm H}^1(K, \mu_n)}$ is the cokernel of the $n$-multiplication on $\GGr(\mathcal G)$ with $\mathcal G$ the N\'eron model of $ \G_{m,K}$ over $S$. By the right exactness of the functor $\pi_0$, we then get the isomorphisms $\pi_0(\bm{{\rm H}^1(K, \mu_n)} \stackrel{\sim}{\leftarrow} \pi_0(\GGr(\mathcal G))/n\pi_0(\GGr(\mathcal G)) \stackrel{\sim}{\rightarrow} \Z/n\Z$, since $\pi_0(\GGr(\mathcal G))$ is canonically isomorphic to $\pi_0(\mathcal G_s)$, and hence to $\Z$ (see \S \ref{sec.proalg-gree} (i)).
\end{proof}
}

\subsection{B\'egueri's construction}\label{sec.beg}
 In this section we assume that $K$ has characteristic $0$.
Given $K$-schemes $Z_K$ and $U_K$, let $Z_{U_K}$ denote the fibred product $Z_K\times_K U_K$, viewed as a scheme over $U_K$. Furthermore, we will identify any commutative $K$-group scheme with the corresponding fppf sheaf so that $\mathrm{Hom}$ and $\mathrm{Ext}$ groups or sheaves are always meant for the fppf topology.

Let $X_K$ be a $K$-torsor under $A_K$ and let $n$ be a positive integer such that $n[X_K]$ is trivial in ${\rm H}^1_\fl(K,A_K)$. Since the order of $X_K$ is defined as the order of $[X_K]$, it is the minimum among such integers. The element $[X_K]$ in ${\rm H}^1_\fl(K,A_K)$ corresponds to an extension of group schemes over $K$
\begin{equation}\label{eq.etaB} 0\longrightarrow A_K\longrightarrow B_K\longrightarrow \Z\longrightarrow 0,
\end{equation}
so that the fibre at $1\in \Z$ is isomorphic to $X_K$ (\cite{SGA7}, VII, \S 1.4). Since the class of $X_K$ in ${\rm H}^1_\fl(K,A_K)$ is $n$-torsion, the exact sequence \eqref{eq.etaB} is the pull-back along $\Z\to \Z/n\Z$ of a, not unique, extension
\begin{equation}\label{eq.eta}\eta\colon 0\longrightarrow A_K\stackrel{\alpha}{\longrightarrow} E_K\longrightarrow \Z/n\Z\longrightarrow 0
\end{equation}
and $X_K$ is isomorphic to the fibre of $E_K$ at $1\in \Z/n\Z$. Let
\begin{equation}\label{eq.etan}
\eta_n \colon 0\longrightarrow{}_n A_K\longrightarrow {}_n E_K\longrightarrow \Z/n\Z\longrightarrow 0
\end{equation}
be the sequence of $n$-torsion subgroups. Consider also the exact sequence
\begin{equation}\label{eq.muvE}
0\longrightarrow \mu_n\longrightarrow V_{{}_nE_K}^*\longrightarrow \underline{\rm Ext}^1(E_K,\G_m)_{{}_nE_K}\stackrel{\tau_E}{\longrightarrow}
 \underline{\rm Ext}^1(E_K,\G_m)\cong A^\prime_K \longrightarrow 0
\end{equation}
(cf. \cite{Beg}, 2.3.2)
where $V_{{}_nE_K}^*$ denotes the torus $\Re_{{}_nE_K/K}(\G_{m, {}_nE_K})$ representing the Weil restriction functor that associates to a $K$-scheme $S'$ the group $\G_{m,K}(S'\times_K {}_nE_K )$ (\cite{BLR}, 7.6). The isomorphism on the right is due to the vanishing of the sheaf $\underline{\rm Ext}^i(\Z/n\Z,\G_m)$ for $i=1,2$; to prove this fact recall that $n\colon \G_{m}\to \G_{m}$ is an epimorphism for the fppf topology and that $\underline{\rm Ext}^i(\Z,\G_m)=0$ for $i>0$. Observe that
\[\mu_n\cong \underline{\rm Hom}(\Z/n\Z,\G_m) \stackrel{\sim}{\longrightarrow} 
 \underline{\rm Hom}(E_K,\G_m).
\]
The second map in \eqref{eq.muvE} sends a homomorphism $f\colon E_K\to \G_{m,K}$ to its restriction to ${}_nE_K$, while
the third arrow sends $g\in \G_{m,K}( {}_nE_K ) $ to (the isomorphism class of) the trivial extension endowed with the section $g$,
and the map $\tau_E$ forgets the rigidification along ${}_nE_K$.

We now describe B\'egueri's construction of Shafarevich's duality following \cite{Beg}.
Let $F_K$ be a finite $K$-group scheme and $F_K^D$ its Cartier dual. There is
 a short exact sequence (cf. \cite{Beg}, 2.2.1)
\begin{equation}\label{eq.muZ2}
0\longrightarrow F_K^D\longrightarrow V_{F_K}^*\longrightarrow \underline{\rm Ext}^1(F_K,\G_m)_{F_K}\longrightarrow 0,
\end{equation}
where the second map forgets the group structure and the third map associates to each $f\in \G_{m,K}(F_K)$ the trivial extension
endowed with the rigidification induced by $f$.
We also recall the following exact sequence (cf. \cite{Beg}, 2.3.1)
\begin{equation}\label{eq.ExtAAprime}
0\longrightarrow V_{{}_nA_K}^*\longrightarrow \underline{\rm Ext}^1(A_K,\G_m)_{{}_nA_K}\stackrel{\tau_A}{\longrightarrow}
 A^\prime_K\longrightarrow 0.
\end{equation}
In \cite{Beg}, 8.2.2, B\'egueri first constructs a map
\begin{equation*}\label{eq.Gamma}
\Gamma\colon {\rm H}^1_\fl(K,{}_nA_K)\longrightarrow {\rm Ext}^1(\GGr(A'),\bm{{\rm H}^1(K,\mu_n) })
\end{equation*}
as follows: any element in ${\rm H}^1_\fl(K,{}_nA_K)$ corresponds to a sequence $\eta_n$ as in \eqref{eq.etan}. Consider now the diagram
\begin{equation*}\label{eq.diaBeg}
\xymatrix{
\mu_n\ar@{=}[r]\ar@{^{(}->}[d]&\mu_n\ar@{^{(}->}[d]&0\ar[d]\\
V_{\Z/n\Z}^*\ar[r]\ar[d]^{v_1}&V_{{}_nE_K}^*\ar[r]\ar[d]^{v_2}&V_{{}_nA_K}^*\ar[d]^{v_3}\\
 \underline{\rm Ext}^1(\Z/n\Z,\G_m)_{\Z/n\Z}\ar[d]\ar[r]& \underline{\rm Ext}^1(E_K,\G_m)_{{}_nE_K}\ar@{->>}[d]^{\tau_E}\ar[r]& \underline{\rm Ext}^1(A_K,\G_m)_{{}_nA_K}\ar@{->>}[d]^{\tau_A}\\
0& A'_K\ar@{=}[r]& A'_K&}
\end{equation*}
where the rows are complexes and the vertical sequences are those in \eqref{eq.muZ2}, for $F_K=\Z/n\Z$,
 \eqref{eq.muvE}, \eqref{eq.ExtAAprime}, respectively.
Since $K$ has characteristic $0$, the second row consists of tori, while the third row consists of semi-abelian varieties.
Hence they all admit N\'eron models, locally of finite type over $S$. On passing to the perfection of the Greenberg realization of the N\'eron models
and considering the cokernels of the maps induced by $v_1,v_2,v_3$, one gets a complex of Serre pro-algebraic groups (cf. Lemma~\ref{lem.comptori})
\begin{equation}\label{eq.ZEA}
 0\longrightarrow \bm{{\rm H}^1(K,\mu_n)} \longrightarrow \bm{{\rm Ext}^1(E_K,\G_m)}\longrightarrow \GGr(A^\prime)\longrightarrow 0;
\end{equation}
this is indeed an exact sequence because on $k$-points it induces the exact sequence
\[
 0\longrightarrow{\rm Ext}^1(\Z/n\Z,\G_m)\longrightarrow {\rm Ext}^1(E_K,\G_m)\longrightarrow {\rm Ext}^1(A_K,\G_m)\longrightarrow 0,
\]
where the group on the left is isomorphic to ${ \rm H}^1_\fl(K,\mu_n)$ and the one on the right is isomorphic to $A'(\OO_K)$. We have thus associated with \eqref{eq.etan} an extension of $\GGr(A')$ by $\bm{\HH^1(K,\mu_n)}$: this is the image of \eqref{eq.etan}
via $\Gamma$.

The homomorphism
\begin{equation*}\label{eq.psid}
\psi_n\colon {\rm H}^1_\fl(K,{}_nA_K)\longrightarrow {\rm Ext}^1(\GGr(A^{\prime 0}),\Z/n\Z)
 \end{equation*}
in \cite{Beg}, 8.2.3, is then obtained by applying first $\Gamma$, then the pull-back along
$\GGr(A^{\prime 0})\to \GGr(A^{\prime})$ and, finally, the push-out along $\bm{\HH^1(K,\mu_n)}\to \pi_0(\bm{\HH^1(K,\mu_n)}) \cong \Z/n\Z $ (cf.~Lemma~\ref{comp-mun}).
Let
\begin{equation}\label{eq.shaX}
\psi_n(\eta_n)\colon \quad 0\longrightarrow \Z/n\Z \longrightarrow W(X_K)\longrightarrow \GGr(A^{\prime 0})\longrightarrow 0
\end{equation}
denote the image of \eqref{eq.etan} via $\psi_n$.
Recall now that (cf. \cite{SerrePro}, 5.4)
\begin{equation}\label{eq.ExtHom-pro}
{\rm Ext}^1(\GGr(A^{\prime 0}), \Q/\Z) \stackrel{\sim}{\longrightarrow}  {\rm Hom}(\pi_1(\GGr(A^{\prime 0})), \Q/\Z) \stackrel{\sim}{\longleftarrow} {\rm Hom}(\pi_1(\GGr(A^\prime)), \Q/\Z) .
\end{equation}
In terms of homomorphisms of profinite groups, the extension \eqref{eq.shaX} then corresponds to a map
 \begin{equation}\label{eq.utau}
u^\tau=u^\tau_{X_K} \colon \pi_1(\GGr(A^{\prime}))\longrightarrow \pi_0(\bm{{\rm H}^1(K,\mu_n)} ) \cong \Z/n\Z\subset \Q/\Z
\end{equation}
deduced from \eqref{eq.ZEA} (or equivalently, from the pull-back of
\eqref{eq.ZEA} along $\GGr(A^{\prime 0})\to \GGr(A')$)
 via the long exact sequence of $\pi_i$'s.

\begin{theorem}[\cite{Beg}, 8.2.3, 8.3.6]
 \begin{itemize}
\item[(i)] The extension class $\psi_n(\eta_n)$ in \eqref{eq.shaX} depends only on the class of the sequence \eqref{eq.etaB}, \emph{i.e.}, on $[X_K]$; furthermore, $\psi_n$ factors through an isomorphism $\Psi_n\colon {}_{n}{\rm H}^1_\fl( K,A_K)\stackrel{\sim}{\to} {\rm Ext}^1(\GGr(A^{\prime 0}),\Z/n\Z)$. 
\item[(ii)] Let $n',n\in \Z_{>0}$, with $n|n'$; then the following diagram
\[\xymatrix{{}_n {\rm H}^1_\fl (K,A_K)\ar[r]\ar[d]^{\Psi_n }& {}_{n'}{\rm H}^1_\fl( K,A_K)\ar[d]^{\Psi_{n'} }\\
{\rm Ext}^1(\GGr(A^{\prime 0}),\Z/n\Z)  \ar[r]& {\rm Ext}^1(\GGr(A^{\prime 0}),\Z/n'\Z)
}\]
 commute, where the upper horizontal arrow is the usual inclusion of torsion subgroups of ${\rm H}^1_\fl (K,A_K)$ and the lower horizontal arrow is the push-out along the inclusion $\Z/n\Z\to \Z/n'\Z$ (in $\Q/\Z$). 
\item[(iii)]Passing to the limit on $n$, the homomorphisms $\Psi_n$ provide an isomorphism $ {\rm H}^1_\fl( K,A_K)\stackrel{\sim}{\to} {\rm Ext}^1(\GGr(A^{\prime 0}), \Q/\Z) $ and hence Shafarevich duality $ {\rm H}^1_\fl( K,A_K)\stackrel{\sim}{\to} {\rm Hom}(\pi_1(\GGr(A^\prime)), \Q/\Z)$ in \eqref{eq.sha} via the isomorphisms \eqref{eq.ExtHom-pro}.
\end{itemize}
\end{theorem}
Hence we can deduce that
\begin{corollary}\label{cor.utau}
Shafarevich's duality $ {\rm H}^1_\fl( K,A_K)\stackrel{\sim}{\to} {\rm Hom}(\pi_1(\GGr(A^\prime)), \Q/\Z)$ in \eqref{eq.sha} maps the class of the torsor $X_K$ to the homomorphism $ u^\tau_{X_K}$ in \eqref{eq.utau}.
\end{corollary}

We give now an alternative construction of the map $\psi_n$ that will be useful for further applications. The kernel of $\tau_E$ in \eqref{eq.muvE} is a torus, which, for brevity, will be denoted by $T^\tau_K$.
Let $T^\tau$ be its N\'eron model over $S$. We have an exact sequence
\begin{equation*}\label{seq.torustau}
0\longrightarrow T_K^\tau\longrightarrow G_K^\tau:=\underline{\rm Ext}^1(E_K,\G_m)_{{}_nE_K}\stackrel{\tau_E}{\longrightarrow}
 A^\prime_K\longrightarrow 0
 \end{equation*}
which extends to a sequence of N\'eron models 
\begin{equation}\label{seq.torustauN}
0\longrightarrow T^\tau\longrightarrow G^\tau{\longrightarrow}
 A^\prime \longrightarrow 0
 \end{equation}
which is exact on $S$-sections because $T_K^\tau$ is cohomologically trivial. On applying the perfection of the Greenberg functor we get an exact sequence
\begin{equation}\label{eq.Grj}
0\longrightarrow\GGr( T^\tau)\longrightarrow \GGr(G^\tau)\stackrel{\bm\tau}{\longrightarrow} \GGr( A^\prime)\longrightarrow 0
\end{equation}
where the first two groups are not Serre pro-algebraic groups in general, because they are projective limits of perfect schemes not necessarily of finite type.
Let $j_*V_{{}_nE_K}^*$ denote the N\'eron model over $S$ of the torus $V_{{}_nE_K}^*$.
Since the map $V_{{}_nE_K}\to T_K^\tau$, deduced from \eqref{eq.muvE}, is an isogeny with kernel $\mu_n$, we have an exact sequence
\begin{equation*}\label{eq.GrH}
\GGr( j_*V_{{}_nE_K}^* )\longrightarrow \GGr(T^\tau)\stackrel{h^\tau}{\longrightarrow}\bm{{\rm H}^1(K,\mu_n)} \longrightarrow 0,
\end{equation*} as explained in the proof of Lemma~\ref{lem.comptori}.
Now take the push-out of \eqref{eq.Grj}
 along $h^\tau$; by construction, the resulting exact sequence
 is the one in \eqref{eq.ZEA}, \emph{i.e.}, the image of 	\eqref{eq.etan} via $\Gamma$.
Hence, if one considers the pull-back of \eqref{eq.Grj} along $\GGr(A^{\prime 0})\to \GGr(A^\prime)$,
\begin{equation}\label{eq.Grj0}
0\longrightarrow\GGr( T^\tau)\longrightarrow U \longrightarrow
\GGr( A^{ \prime 0})\longrightarrow 0,
\end{equation}
and then the push-out of \eqref{eq.Grj0} along the composition of maps
\[ \GGr(T^\tau)\stackrel{h^\tau}{\longrightarrow}\bm{{\rm H}^1(K,\mu_n)}\longrightarrow\pi_0( \bm{{\rm H}^1(K,\mu_n)} ) \cong \Z/n\Z , 
\]
one gets the extension $\psi_n(\eta_n)$ in \eqref{eq.shaX}, \emph{i.e.}, the image of $[X_K]$ via Shafarevich's duality.

Thanks to this new description of Shafarevich's map, we can characterize the map $u^\tau$ in \eqref{eq.utau} as follows:
 let $T^{\tau, \ft}$ be the maximal subgroup scheme of finite type in $T^\tau$ and consider the sequence, exact on $S$-sections,
\[
0\longrightarrow T^{\tau, \ft}\longrightarrow G'\longrightarrow A^{\prime 0}\longrightarrow 0,
\]
and obtained from \eqref{seq.torustauN}, as explained in the proof of Lemma~\ref{lem.torustor}. 
Then $\psi_n(\eta_n)$, \emph{i.e.}, the push-out of \eqref{eq.Grj0} along $h^\tau$ is isomorphic to the push-out of
\begin{equation}\label{eq.zetataup}
0\longrightarrow\GGr( T^{\tau, \ft})\longrightarrow \GGr(G')\longrightarrow \GGr(A^{\prime 0})\longrightarrow 0\end{equation}
along the composition of maps $h^{\tau, \ft}\colon \GGr( T^{\tau, \ft})\to \GGr( T^\tau)\stackrel{h^\tau}{\to} \bm{{\rm H}^1(K,\mu_n)}$.
Hence
\begin{equation}\label{eq.w}
u^\tau_{X_K} =\pi_0(h^{\tau, \ft}) \circ u^{\tau, \ft},
\end{equation}
where the homomorphism $u^{\tau, \ft} \colon \pi_1(\GGr(A^\prime))\to \pi_0(\GGr(T^{\tau, \ft}))$ is deduced from \eqref{eq.zetataup} via the long exact sequence of the $\pi_i$'s.

\subsection{An alternative construction using rigidificators}\label{sec.rig}
Let $X_K$ be a torsor under an abelian variety $A_K$.
We will see in this section how the homomorphism $u^\tau$ in \eqref{eq.utau} (and in \eqref{eq.w}) can be constructed using a rigidificator $x_K$ of the relative Picard functor $\Pic_{X_K/K}$ (\cite{Raynaud}, 2.1.1).
Observe that any closed point of $X_K$ provides a rigidificator of $\Pic_{X_K/K}$.

\begin{lemma}\label{lem.index-period} Let $X_K$ be a torsor under an abelian variety $A_K$, of order $d$.
Let $d'$ be the separable index of $X_K$, \emph{i.e.}, the greatest common divisor of the degrees of its finite separable 
splitting extensions. Then $d|d'$. If $A_K$ is an elliptic curve, then $d=d'$ and the index is indeed the degree of a minimal separable splitting extension.
\end{lemma}
\proof For the first assertion see \cite{LT}, comments at pp. 663-664 and Proposition 5.
For the latter assertion on elliptic curves see \cite{GLL}, Thm. 9.2.
\qed

\begin{remark}\label{rem.tori} Let $x_K=\spec(K')$ with $K'/K$ a finite separable extension of degree $n$. Then the torus $V_{x_K}^*:=\Re_{K'/K}(\G_{m,K'})$ has component group isomorphic to $\Z$ and the closed immersion $\G_{m,K}\to V_{x_K}^*$ (\cite{BLR}, p.~197 last lines), that is the inclusion $K^*\subset K^{\prime *}$ on $K$-sections, induces the $n$-multiplication $n\colon \Z\to \Z$ on component groups of N\'eron models over $S$.
\end{remark}

The main idea here is to use in \eqref{eq.muvE} a rigidificator $x_K$ of $\Pic_{X_K/K}$ in place of ${}_nE_K$.
The advantage is that the new construction works even for $K$ of positive characteristic; in the latter case we choose $x_K$ \'etale so that $V_{x_K}^*:=\Re_{x_K/K}(\G_{m, x_K})$ is still a torus.

Observe that a rigidificator $x_K$ is a closed subscheme of $E_K$ and the homomorphisms
 \begin{equation*}\label{eq.muvclosed}
\underline{\rm Hom}(E_K,\G_m)\longrightarrow V_{x_K}^*
\end{equation*}
is still a closed immersion.
Indeed any homomorphism $f\colon E_K\to \G_m$ factors through $\rho\colon E_K\to \Z/n\Z$ and if $f_{|x_K}=0$ then $f_{|X_K}=0$ because $x_K$ is a
rigidificator. However $X_K$ is the fibre at $1$ of $\rho$ and hence also $f=0$. 
We then have an exact sequence
\begin{equation}\label{eq.muvEx}
0\longrightarrow \mu_n \longrightarrow V_{x_K}^*\longrightarrow \underline{\rm Ext}^1(E_K,\G_m)_{x_K}\longrightarrow
 A^\prime_K\longrightarrow 0 
\end{equation}
after recalling the isomorphisms $\mu_n\cong {\rm Hom}(\Z/n\Z,\G_m) \stackrel{\sim}{\to} {\rm Hom}(E_K,\G_m) $. 
More generally, for any finite \'etale subscheme $Z_K$ of $E_K$ which satisfies the following property
\begin{verse}
 $ (*)$\quad the canonical map
$ \underline{\rm Hom}(E_K,\G_m) \longrightarrow V_{Z_K}^*$ is a closed immersion,
\end{verse}
 we can construct an exact sequence as in \eqref{eq.muvEx}.

Let $T_K$ denote the torus $T_K^x:=V_{x_K}^*/\mu_n $.
The sequence \eqref{eq.muvEx} induces an exact sequence
\begin{equation}\label{eq.TExtA}
0\longrightarrow T_K\longrightarrow \underline{\rm Ext}^1(E_K,\G_m)_{x_K}\longrightarrow A^\prime_K\longrightarrow 0,
\end{equation}
and hence a sequence which is exact on $S$-sections (see proof of Lemma~\ref{lem.torustor})
\begin{equation}\label{eq.TExtAN}
0\longrightarrow T^\ft\longrightarrow G'' \stackrel{ }{\longrightarrow}
 A^{\prime 0} \longrightarrow 0,
\end{equation}
where $T^\ft$ is the maximal subgroup of finite type over $S$ of the N\'eron model $T$ of $T_K$.
Now consider the cokernel
 \begin{equation}\label{eq.hproalg}
\GGr(j_*V_{x_K}^*)\stackrel{g^x}{\longrightarrow}\GGr(T)\stackrel{h}{\longrightarrow} \bm{{\rm H}^1(K,\mu_n)}\longrightarrow 0
\end{equation}
of the homomorphism between the perfect Greenberg realizations of the N\'eron models of $V_{x_K}^*$ and $T_K$; by Lemma~\ref{lem.comptori} it
is a Serre pro-algebraic group whose group of $k$-points is $ {\rm H}^1_\fl(K,\mu_n)$.

In order to provide a more useful description of the map $u^\tau$ in \eqref{eq.utau}, consider the perfect Greenberg realization
of \eqref{eq.TExtAN}
\begin{equation}\label{eq.TExtANG}
0\longrightarrow \GGr(T^\ft)\longrightarrow \GGr(G'')\longrightarrow
\GGr( A^{\prime 0})\longrightarrow 0,
\end{equation}
and then its push-out along the composition of maps
\begin{equation}\label{eq.hprime}
h^\ft\colon \GGr(T^\ft)\longrightarrow \GGr(T)\stackrel{h}{\longrightarrow} \bm{{\rm H}^1(K,\mu_n)}.
\end{equation}
We obtain an exact sequence
\begin{equation*}\label{eq.zeta}
\zeta\colon \quad 0\longrightarrow \bm{{\rm H}^1(K,\mu_n)}\longrightarrow W' \longrightarrow \GGr( A^{\prime 0})\longrightarrow 0
\end{equation*}
and hence a homomorphism
\[u_{X_K}=u\colon \pi_1(\GGr(A^\prime))\longrightarrow \pi_0( \bm{{\rm H}^1(K,\mu_n)} ) \cong \Z/n\Z\subset \Q/\Z\]
such that
\begin{equation}\label{eq.u} u=\pi_0(h^\ft)\circ u^\ft,\end{equation}
where $ u^\ft\colon \pi_1(\GGr(A^\prime))\to \pi_0(\GGr(T^\ft)) \stackrel{\sim}{\to} \pi_0(T)_\tor$
is deduced from the long exact sequence of the $\pi_i$'s of \eqref{eq.TExtANG}.

\begin{proposition}\label{pro.sha0}
{The map $\Xi\colon {\rm H}^1(K, A_K)\to {\rm Hom}(\pi_1(\GGr(A^\prime)),\Q/\Z)$, with $[X_K]\mapsto u_{X_K}$, is a group} homomorphism.
If $char(K)=0$ the homomorphism $u_{X_K}$ in \eqref{eq.u} coincides with the homomorphism $u^\tau_{X_K}$ in \eqref{eq.utau}.
In particular, the homomorphism $\Xi$ is Shafarevich's duality in \eqref{eq.sha}.
\end{proposition}
\proof
We start by showing that, once $X_K$ has been fixed, the construction of $u\colon \pi_1(\GGr(A'))\to \Q/\Z$ in \eqref{eq.u}
does not depend on the choices of $x_K$, $n$ and $\eta\in {\rm Ext}^1(\Z/n\Z,A_K)$ above $X_K$.

First we see that $u$ does not depend on the \'etale finite closed subscheme $x_K$ of $E_K$ satisfying $(*)$.
Let $x_K\subseteq y_K$ be two \'etale subschemes of $E_K$ satisfying $(*)$.
Let $T_K^x, h^x, h^{\ft, x}, u^{\prime x}, u^x$ denote, respectively, the torus in \eqref{eq.TExtA}, the maps in \eqref{eq.hproalg},
\eqref{eq.hprime} and \eqref{eq.u} for $x_K$, and similarly for $y_K$.
The canonical morphism of tori $ T_K^y\to T_K^x$ induces a morphism $\beta\colon T^{y,\ft}\to T^{x,\ft }$ between the maximal subgroups of finite type of the N\'eron models over $S$. Denote by $\beta^\prime\colon \GGr(T^{y, \ft})\to \GGr(T^{x,\ft})$ the corresponding map
on perfect Greenberg realizations. One then has $\beta^\prime \circ h^{x,\ft}= h^{y,\ft}$ and
 $\pi_0(h^{y,\ft})=\pi_0(h^{x,\ft})\circ \pi_0(\beta')$.
Furthermore the sequence \eqref{eq.TExtANG} for $x_K$ is the push-out along $\beta^\prime$ of the sequence \eqref{eq.TExtANG} for $y_K$.
Hence $u^{x,\ft}=\pi_0(\beta^\prime )\circ u^{y,\ft}$. We conclude then that
\begin{equation}\label{eq.uxuy} u^x= \pi_0(h^{x,\ft}) \circ u^{x,\ft} = \pi_0(h^{x,\ft}) \circ \pi_0(\beta^\prime )\circ u^{y,\ft}=\pi_0(h^{y,\ft}) \circ u^{y,\ft}=u^y.
\end{equation}
Let now $n,\hat n$ be positive integers such that $n[X_K]=0$ and $n|\hat n$. We can consider the pull-back $\hat \eta$ of $\eta$
in \eqref{eq.eta} along the projection $\Z/\hat n\Z\to \Z/n\Z$. If we proceed with $\hat \eta$ as we have done for $\eta$,
we get a map $\hat u\colon \pi_1(\GGr(A^\prime))\to \Q/\Z$. Observe that the $2$-fold extension \eqref{eq.muvEx} for $\hat\eta$
is the push-out along $\mu_n\to \mu_{\hat n}$ of \eqref{eq.muvEx} and that the map
$ \pi_0 ( \bm{{\rm H}^1(K,\mu_n)} )\to \pi_0 ( \bm{{\rm H}^1(K,\mu_{\hat n})})$ is the inclusion $\Z/n\Z\to \Z/\hat n\Z$. It is now
immediate to check that the maps $\hat u$ and $u$ coincide.

We have thus obtained a map
\begin{equation}\label{eq.etau}
{\rm Ext}^1(\Z/n\Z, A_K)\longrightarrow{\rm Hom}(\pi_1(\GGr(A')),\Q/\Z),\qquad \eta\mapsto u.
\end{equation}
To check that this map is indeed a homomorphism, observe that it
 is functorial in $A_K$. Furthermore we could repeat the construction with any finite constant group $F_K$ in place of $\Z/n\Z$
obtaining in this way a map
 \[{\rm Ext}^1(F_K,A_K)\longrightarrow {\rm Hom}( \pi_1(\GGr(A^\prime)),\pi_0(\bm{{\rm H}^1(K,F_K^D)}) )
\]
with $F_K^D$ the Cartier dual of $F_K$. This construction is functorial in $F_K$.
The functoriality results are sufficient to conclude that the map in \eqref{eq.etau} is a homomorphism, because the Baer sum
of two extensions as in \eqref{eq.eta} is found by first taking the direct sum of the two extensions, then applying the push-out along
the multiplication of $A_K'$ and finally applying the pull-back along the diagonal map $\Z/n\Z\to \Z/n\Z\oplus \Z/n\Z$.

Suppose now that $n$ and $x_K$ are fixed.
We show that the map $u$ is trivial if $X_K$ is trivial, \emph{i.e.}, the map in \eqref{eq.etau} factors
through ${\rm H}^1_\fl(K,A_K)\cong {\rm Ext}^1(\Z,A_K)$.
Suppose that $X_K$ is trivial and choose a $K$-point $x_K$ of $X_K$. In particular, $V_{x_K}^*=\G_{m,K}$, $T_K \cong \G_{m,K}$ and $\pi_0(T) \cong \Z$. Hence $T^\ft=\G_{m,\OO_K}$, the homomorphism $u^\ft\colon \pi_1( \Gr(A'))\to \pi_0(\GGr(T^\ft))=0$ is the zero map and $u=0$.

Suppose now that $char(K)=0$. To see that the homomorphism $[X_K]\mapsto u_{X_K}$ is Shafarevich's duality, it is sufficient to check that the homomorphisms $u^\tau$ in \eqref{eq.utau} and $u$ in \eqref{eq.u} coincide (see Corollary~\ref{cor.utau}).
Consider then a finite separable extension $K'/K$ splitting \eqref{eq.etan} and a point $x_K=\spec(K')$ of ${}_n E_K$ above $1$. In particular $x_K$ is a rigidificator of $\Pic_{X_K/K}$.
Set $y_K={}_n E_K$. Then, $u^y$ coincides with the map $u^\tau$ in \eqref{eq.w} and one can repeat the arguments used in \eqref{eq.uxuy} to showing that $u^\tau=u^y=u^x$.
\qed

\begin{remark}
The original construction by B\'egueri works only for $K$ of characteristic zero because in the case of positive characteristic the scheme $ V_{{}_n E_K}^*$ (and hence $T^\tau_K$) need not be a torus; in particular it might not admit a N\'eron model over $S$.
The construction via rigidificators described in this section works in any characteristic. For $char(K)=p$ it is not clear that the homomorphism $\Xi$ in Proposition~\ref{pro.sha0} is Shafarevich's duality in \eqref{eq.sha} (see also \cite{Ber}, Theorem 3). We will see in Proposition~\ref{pro.shap} that this is the case on the prime-to-$p$ parts. 
\end{remark}

\subsection{A construction via the Picard functor}\label{sec.picardconstruction}
{Let $A_K$ be an abelian variety.}
In this section we present a third possible construction of a homomorphism ${\rm H}^1(K, A_K)\to {\rm Hom}(\pi_1(\GGr(A^\prime)),\Q/\Z)$, this one making use of the relative Picard functor.
We will see in Theorem~\ref{thm.main} that the new construction always coincides with the one in Proposition~\ref{pro.sha0} and hence with Shafarevich's duality in the characteristic $0$ case.

Let $X_K$ be a torsor under $A_K$ and $x_K=\spec(K')$ a closed point of $X_K$ with $K'/K$ a finite separable extension of degree $n$; it exists by the smoothness of $X_K$, and $n[X_K]=0$ by Lemma~\ref{lem.index-period}. No assumption on the characteristic of $K$ is made.

Consider the exact sequence (cf. \cite{Raynaud}, 2.4.1)
 \begin{equation}\label{eq.GVPic}
0\longrightarrow V_{X_K}^*\longrightarrow V_{x_K}^*\longrightarrow ({\rm Pic}_{X_K/K},x_K)^0\longrightarrow A^\prime_K\longrightarrow 0 
\end{equation}
where we use that $A_{K}'{\stackrel{\sim}{\to} \Pic^{0}}_{X_K/K}$ (Remark \ref{rem.J_K et A_K toute dim}).
Observe that $ V_{X_K}^*:=\Re_{X_K/K}(\G_{m,X_K})\cong \G_{m,K}$ (\cite{Raynaud}, 2.4.3), $ V_{x_K}^*$ is a torus and hence so too is $N_K:=V_{x_K}^*/\G_{m,K}$. Let $\mathcal N$ denote its N\'eron model. 
Observe that it follows from Remark~\ref{rem.tori} that the component group of $\mathcal N$ is cyclic of order $n$, hence its perfect Greenberg realization is a Serre pro-algebraic group.

We proceed as in the previous section, first by passing to N\'eron models and then applying the perfect Greenberg realization to the sequence
\begin{equation}\label{eq.N}
0\longrightarrow N_K\longrightarrow (\Pic_{X_K/K},x_K)^0 \stackrel{h_K}{\longrightarrow} A^\prime_K \longrightarrow 0,
\end{equation}
so that we obtain an exact sequence of Serre pro-algebraic groups
\begin{equation*}\label{eq.GrN}
0\longrightarrow \GGr(\mathcal N)\longrightarrow \GGr(j_* (\Pic_{X_K/K},x_K)^0 )\stackrel{h}{\longrightarrow} \GGr(A^\prime)\longrightarrow 0,
\end{equation*}
and hence a homomorphism
\begin{equation}\label{eq.v}
v=v_{X_K}\colon \pi_1(\GGr(A^\prime))\longrightarrow \pi_0(\GGr(\mathcal N)) \cong \Z/n\Z\subset \Q/\Z.
\end{equation}

In order to compare this construction with the (modified) B\'egueri construction of the previous section,
\emph{i.e.}, in order to compare the maps $u$ in \eqref{eq.u} and $v$ in \eqref{eq.v}, we consider the following diagram
\begin{equation}\label{eq.diagramTNA}
\xymatrix{
0\ar[r]& T_K\ar[r]\ar[d]^{t_K}& \underline{\rm Ext}^1(E_K,\G_m)_{x_K}\ar[d]^{f_K}
    \ar[r]& A^\prime_K \ar@{=}[d] \ar[r]& 0\\
0\ar[r]& N_K\ar[r]& (\Pic_{X_K/K},x_K)^0 \ar[r]^(0.6){h_K}& A^\prime_K \ar[r]& 0}
\end{equation}
where the upper sequence is \eqref{eq.TExtA}, the lower one is \eqref{eq.N}, $t_K\colon T_K{ :=} V_{x_K}^*/ \mu_n  \to   V_{x_K}^*/ V_{X_K}^* \stackrel{\sim}{\to} N_K$ is induced by the identity on $V_{x_K}^*$,
and $f_K$ associates to a $\G_m$-extension $L_K$ of $E_K$ endowed with a
$x_K$-section $\sigma$ its restriction (as torsor) to $X_K$ endowed with the trivialization along $x_K$ induced by $\sigma$. 
 The morphism $t_K$ is surjective and its kernel is $V_{X_K}^*/\mu_n\cong \G_{m,K}/\mu_n \cong \G_{m,K}$.

Consider now the induced diagram on N\'eron models (with exact rows when restricting to $S$-sections)
\begin{equation}\label{dia.TTN}
\xymatrix{
0\ar[r]& T^\ft\ar[r]\ar[d]^{}\ar@/_1pc/[dd]_{t^\ft}& G''\ar[d]^{}
    \ar[r]& A^{\prime 0} \ar[d] \ar[r]& 0\\
& T\ar@{^{(}->}[r]\ar[d]^{t}& j_* \underline{\rm Ext}^1(E_K,\G_m)_{x_K}\ar[d]^{f}
    \ar[r]& A^\prime \ar@{=}[d] \ar[r]& 0\\
0\ar[r]&{\mathcal N}\ar[r]& j_*(\Pic_{X_K/K},x_K)^0 \ar[r]& A^\prime \ar[r]& 0}
\end{equation}
where the first row is \eqref{eq.TExtAN}. Here $j_*H_K$ is just a notation for the N\'eron model of $ H_K$ when it exists. 
The homomorphism $u$ in \eqref{eq.u} is the composition of the homomorphism
$u^\ft\colon \pi_1(\GGr(A^{\prime 0})) \to \pi_0(\GGr(T^\ft))$ (deduced from the upper sequence) with the homomorphism
\[\pi_0(h^\ft)\colon \pi_0(T^\ft)= \pi_0(\GGr(T^\ft))\longrightarrow \pi_0(\bm{{\rm H}^1(K,\mu_n)} ),\]
where $h^\ft$ was introduced in \eqref{eq.hprime}.
It now follows form the above diagram that the map $v\colon \pi_1(\GGr(A^\prime))\to \pi_0(\GGr(\mathcal N))$ in \eqref{eq.v},
obtained from the lower exact sequence, satisfies
\begin{equation}\label{eq.vuft} v= \pi_0(t^\ft)\circ u^\ft.
 \end{equation}
We are going to check that $u$ and $v$ coincide up to sign, by showing that, up to canonical identifications we have $\pi_0(h^\ft)=-\pi_0(t^\ft)$.
To see this fact, consider the following diagram with exact rows and columns
\begin{equation*}
\xymatrix{
0\ar[r]& \mu_n\ar[r] \ar@{=}[d] &V_{X_K}^*\cong \G_{m,K} \ar[r]^n\ar[d]& \G_{m,K} \ar[r]\ar[d]& 0\\
0\ar[r]& \mu_n\ar[r] &V_{x_K}^* \ar[r]\ar[d]& T_K \ar[r]\ar[d]^{t_K}& 0\\
&&N_K \ar@{=}[r]&N_K }
\end{equation*}
where the middle horizontal sequence is deduced from \eqref{eq.muvEx} while the middle vertical sequence comes from \eqref{eq.GVPic}.
Consider the induced diagram of component groups of N\'eron models
\begin{equation*}
\xymatrix{
\Z \ar[r]^n\ar@{^{(}->}[d]& \Z\ar@{^{(}->}[d]  &\pi_0(T^\ft)=\pi_0(T)_\tor\ar@{_{(}.>}[dl]_\iota
\ar@{.>}[ddl]^{\pi_0(t^\ft)} \\
\pi_0(j_*V_{x_K}^*) \ar[r]\ar@{->>}[d]&\pi_0 (T )\ar@{->>}[d]&\\
\pi_0(\mathcal N) \ar@{=}[r]&\pi_0(\mathcal N) &}
\end{equation*}
where $\iota$ is the inclusion map and the vertical sequences are left exact because $\Z$ is torsion free (cf.~\cite{SGA7}, VIII 5.5).
We complete the diagram by inserting the cokernels of the horizontal maps
\begin{equation*}\label{eq.diagrambig}
\xymatrix{ 
0\ar[r]& \Z\ar@{^{(}->}[d]\ar[r]^n& \Z\ar[r]\ar[d]& \Z/n\Z\ar[r]\ar[d]^{{\wr}}&0\\
 0\ar[r]&\pi_0({j_*}V_{x_K}^* )\ar[r]\ar@{->>}[d]^{}& \pi_0(T)\ar[r] ^{{\pi_0(h)}}\ar@{->>}[d]_{{\pi_0(t)}}&\pi_0( \bm{{\rm H}^1(K,\mu_n)}) \ar[r]&0\\
& \pi_0(\mathcal N) \ar@{=}[r] & \pi_0(\mathcal N )& \pi_0(T^\ft)\ar@{.>}[u] _{\pi_0(h^\ft)}\ar@{.>}[l]^{\pi_0(t^\ft)}\ar@{_{(}.>}[ul]_\iota&
}
\end{equation*}
where $\pi_0(h)\circ \iota= \pi_0(h^\ft)$ (see \eqref{eq.hprime}) and $\pi_0(t)\circ \iota=\pi_0(t^\ft)$ (see \eqref{dia.TTN}).
By Remark~\ref{rem.tori}, $\pi_0(j_*V_{x_K}^* )$ is isomorphic to $\Z$, $\pi_0(\mathcal N)$ is isomorphic to $ \Z/n\Z$ and, under these identifications, the left vertical sequence coincides with the upper horizontal sequence. More precisely, we have 
\[ \pi_0(\mathcal N)\stackrel{\sim}{\longrightarrow} \Z/n\Z\stackrel{\sim}{\longrightarrow} \pi_0 ( \bm{{\rm H}^1(K,\mu_n)})\]
 where the first isomorphism maps the image of the class of a uniformizer $\pi'\in K^{\prime *}=V_{x_K}^* (K)$ to the class of $1$,
while the second isomorphism maps the class of $1$ to the image of the cohomology class corresponding to a uniformizer $\pi\in K^*=\G_{m,K}(K)$. Furthermore, the middle vertical sequence splits as does the middle horizontal sequence, and we may identify $\pi_0(T)$ with $\Z\oplus \Z/n\Z$ so that $\pi_0(h)(a\oplus \bar b)=\bar a -\bar b$ and $\pi_0(t)(a\oplus \bar b)=\bar b$ with $\bar b$ the class of $b\in \Z$ modulo $n\Z$.
Let $\sigma$ be the section of $\pi_0(t)$ mapping $\bar b$ to $0\oplus \bar b$. Then $\pi_0(h)\circ \sigma=-\id_{\Z/n\Z}$.
Furthermore $\sigma\circ\pi_0(t) \circ \iota=\iota$ because $\sigma\circ \pi_0(t)\circ \iota-\iota$ factors through $\Z$ and thus is trivial because $\pi_0(T^\ft)$ is torsion. Hence
\[ \pi_0(h^\ft)=\pi_0(h)\circ \iota= \pi_0(h)\circ \sigma \circ \pi_0(t)  \circ \iota= -t\circ \iota= -\pi_0(t^\ft), \]
and thanks to \eqref{eq.vuft} and \eqref{eq.u}, we get
$v=\pi_0(t^\ft)\circ u^\ft= -\pi_0(h^\ft)\circ u^\ft=-u$.
We can then state the main result which is an immediate consequence of what we have just proved and Proposition~\ref{pro.sha0}:

\begin{theorem}\label{thm.main}
Let $A_K$ be an abelian variety over $K$. The homomorphism
\[ \Xi'\colon {\rm H}^1(K,A_K)\longrightarrow{\rm Hom} (\pi_1(\GGr(A^\prime)), \Q/\Z)\]
mapping the class $[X_K]$ of the torsor $X_K$ to the homomorphism $-v_{X_K}\colon \pi_1(\GGr(A'))\to \Q/\Z$ (see \eqref{eq.v}) coincides with the homomorphism $\Xi$ in Proposition~\ref{pro.sha0} mapping $[X_K]$ to the homomorphism $u_{X_K}$ in \eqref{eq.u}. If furthermore the characteristic of $K$ is zero, then $-v_{X_K}=u_{X_K}$ coincides with B\'egueri's homomorphism $u_{X_K}^\tau$ in \eqref{eq.utau} and hence $\Xi'$ is Shafarevich's duality in \eqref{eq.sha}.
\end{theorem}

\subsection{Comparison on the prime-to-$p$ parts}

We have seen in Theorem~\ref{thm.main} that the homomorphisms $\Xi$ and $\Xi'$ always coincide, and in the mixed characteristic case, that they coincide with Shafarevich's duality in \eqref{eq.sha}; in particular they are isomorphisms.
For $K$ of characteristic $p$, it is not clear in general either that they are isomorphisms or that they correspond to Shafarevich's duality in \eqref{eq.sha} (see also \cite{Ber}, Theorem 3).
However, we have a partial result on the prime-to-$p$ parts where Shafarevich's duality is quite easy to describe.

We recall here what Shafarevich's duality looks like on the prime-to-$p$ parts.
Let $n=l^r$ be a positive integer, prime to $p$, and large enough to kill the $l$-primary parts of the component groups of $A_K$ and $A_K'$.
Consider the perfect cup product pairing
\[
\langle\ ,\ \rangle\colon {\rm H}^1(K,{}_nA_K)\times {}_nA^\prime_K(K)\longrightarrow {\rm H}^1(K,\mu_n) \cong \Z/n\Z
\]
on the (\'etale or flat) cohomology groups of the $n$-torsion subgroups of $A_K$ and $A_K^\prime$.
Given an extension $\eta_n$ as in \eqref{eq.etan} (which corresponds to the torsor $X_K$) and a point $a\in {}_nA^\prime_K(K)$,
then $\langle \eta_n,a \rangle$ is the class of the pull-back along $a\colon \Z\to {}_nA^\prime_K$ of the Cartier dual of $\eta_n$,
\[ \eta_n^D\colon 0\longrightarrow \mu_n \longrightarrow {}_nE_K^D \longrightarrow {}_n A_K^\prime\longrightarrow 0, \]
and it corresponds to the image of $a$ along the boundary map $ \partial\colon {}_n A_K^\prime(K) \to {\rm H}^1(K,\mu_n)$.
Furthermore, if $ {}_nA^{\prime 0}$ denotes the quasi-finite subgroup of $n$-torsion sections of $A^{\prime 0}$, we have
\[
\pi_1(\GGr(A^\prime))/n\pi_1(\GGr(A^\prime))
\stackrel{\sim}{\longrightarrow}
{}_nA^{\prime 0}(\OO_K)\stackrel{\sim}{\longleftarrow}
{}_{n^2}A^{\prime}(\OO_K)/ {}_nA^{\prime}(\OO_K),
\]
\[{\rm H}^1(K,{}_{n^2}A_K)/{\rm H}^1(K,{}_nA_K)
\stackrel{\sim}{\longrightarrow}
{}_n{\rm H}^1(K,A_K) 
\]
and hence a perfect pairing 
\begin{equation}\label{eq.shad}{}_n{\rm H}^1(K,A_K) \times
\pi_1(\GGr(A^\prime))/n\pi_1(\GGr(A^\prime)) \longrightarrow {\rm H}^1(K,\mu_n) \cong\Z/n\Z. \end{equation}
Now, using the isomorphism \[ {\rm Hom}( \pi_1(\GGr(A^\prime) )/n\pi_1(\GGr(A^\prime) ),\Z/n\Z )\stackrel{\sim}{\longrightarrow} {\rm Hom}( \pi_1(\GGr(A^\prime)) ,\Z/n\Z ),\] the pairing in \eqref{eq.shad} provides an isomorphism 
\begin{equation}\label{eq.shad2}
{}_n{\rm H}^1(K,A_K)\stackrel{\sim}{\longrightarrow} 
{\rm Hom}( \pi_1(\GGr(A^\prime)) ,\Z/n\Z ),
 \end{equation}  
that is the restriction of \eqref{eq.sha} to the $n$-parts (cf. \cite{Ber} \S~1). 

Now, we will show that \eqref{eq.shad2} coincides with the restriction to the $n$-parts of the homomorphism $\Xi$ in Proposition~\ref{pro.sha0}.
The map $u^\tau_{X_K}\colon \pi_1(\GGr(A^\prime))\to \pi_0(\bm{{\rm H}^1(K,\mu_n)}) \cong{\rm H}^1(K,\mu_n)$ in \eqref{eq.utau} can also be viewed as the composition
\begin{equation}\label{eq.pi1cup}
\pi_1(\GGr(A^\prime))\stackrel{\delta}{\longrightarrow}
 {}_nA^\prime(\OO_K)\stackrel{\sim}{\longrightarrow}  {}_nA^\prime_K(K)\stackrel{\partial}{\longrightarrow}
{\rm H}^1(K,\mu_n) \cong\Z/n\Z \subset \Q/\Z
\end{equation}
where the first map is deduced from
the exact sequence \begin{equation*}
0\longrightarrow {}_nA^\prime_K \longrightarrow A^\prime_K\stackrel{n}{\longrightarrow} A^\prime_K\longrightarrow 0 \end{equation*}
on passing to N\'eron models. More precisely we have
\begin{equation*}
0\longrightarrow {}_nA^\prime\longrightarrow A^\prime\stackrel{n}{\longrightarrow} nA^\prime \longrightarrow 0
\end{equation*}
where $nA^\prime$ is a subgroup scheme of $A^\prime$ that contains $ A^{\prime 0}$.
In particular, on applying the perfect Greenberg realization functor, we get a homomorphism
\begin{equation} \label{eq.pi1A} \delta\colon 
\pi_1(\GGr(A^\prime) )\longrightarrow {}_nA^\prime (\OO_K)
\end{equation}
via the canonical isomorphisms $\pi_1(\GGr( nA^\prime ))  \stackrel{\sim}{\to}\pi_1(\GGr( A^\prime ))$, ${}_nA^\prime (\OO_K)\cong \GGr( {}_nA^\prime ) \stackrel{\sim}{\to} \pi_0({}_nA^\prime )$.


Let now $X_K$ be a torsor under $A_K$ of order $d$ with $d$ a power of a prime integer $l$, $l\neq p$.
Let $n=l^r$ be a multiple of $d$ large enough to kill the $l$-primary parts of the component groups of $A_K$ and $A_K'$.
Fix an extension corresponding to $X_K$ as in \eqref{eq.eta}, and let $x_K=\spec(K')$ be a rigidificator of $\Pic_{X_K/K}$
contained in ${}_nE_K$, \emph{i.e.}, a point of ${}_nE_K$ above $1\in \Z/n\Z$ in \eqref{eq.etan}.
We show that the composition of the maps in \eqref{eq.pi1cup} coincides with the map $u$ in \eqref{eq.u}. This is sufficient
to conclude that our construction via rigidificators (or equivalently via the relative Picard functor) is Shafarevich's
duality on the prime-to-$p$ parts.

With notation as in \eqref{eq.eta}, observe that the $n$-multiplication on $A_K$ factors through $E_K$ so that we have a
homomorphism $\gamma \colon E_K\to A_K$, with kernel ${}_n E_K$ such that $\gamma\circ \alpha=n$.
Consider the sequence in
\eqref{eq.muvEx}.
 We have a diagram with exact rows
\begin{equation*}
\xymatrix{
0\ar[r]& \mu_n \ar[r]\ar@{=}[d] & {}_nE_K^D \ar[r]\ar[d]& \underline{\rm Ext}^1(A_K,\G_m) \ar[r]^{~ ~ n} \ar[d]^{\gamma^*}&
 A^\prime_K\ar[r]\ar@{=}[d]& 0\\
0\ar[r]& \mu_n \cong\underline{\rm Hom}(E_K,\G_m) \ar[r]& V_{x_K}^*\ar[r]& \underline{\rm Ext}^1(E_K,\G_m)_{x_K}\ar[r]&
 A^\prime_K\ar[r]& 0 .
}
\end{equation*}
Indeed ${}_nE_K^D \cong\underline{\rm Hom}({}_nE_K,\G_m) $ maps canonically to $ V_{x_K}^* \cong\underline{\rm Mor}(x_K,\G_{m,K})$;
hence ${}_nA^\prime_K$ maps to the torus $T_K :=V_{x_K}^*/\mu_n$ in \eqref{eq.TExtA}. The push-out of the exact sequence
$0\to {}_nA^\prime_K\to A^\prime_K\to A^\prime_K\to 0$ along ${}_nA^\prime_K\to T_K$ provides the sequence \eqref{eq.TExtA}
and the homomorphism $\gamma^*$ sends a $\G_m$-extension of $A_K$ to its pull-back along $\gamma$ endowed with its
canonical trivialization along $x_K$, induced by the canonical trivialization along ${}_nE_K$.
Moreover, the boundary map $ \partial\colon {}_n A_K'(K) \to {\rm H}^1(K,\mu_n)$ (of finite groups) is the composition
of $ \nu\colon {}_n A_K^\prime(K)\to T_K(K)$ with the boundary map $h\colon T_K(K)\to {\rm H}^1(K,\mu_n)$, \emph{i.e.},
\begin{equation}\label{eq.partial}
 \partial=h\circ \nu.
 \end{equation}
Recall furthermore that the kernel of the $n$-multiplication on $A'$ is a quasi-finite group scheme over $\OO_K$
whose finite part is an \'etale finite group scheme over $\OO_K$ of order prime to $p$, hence constant, because $\OO_K$
is strictly henselian.
On the level of Serre pro-algebraic groups we then have a diagram with exact rows
\begin{equation*}
\xymatrix{
0\ar[r]& {}_n A^\prime(\OO_K)\ar[r]\ar[d]^{ \nu}& \GGr(A^\prime) \ar[r]\ar[d]^{\alpha^*}&
 \GGr(nA^\prime) \ar[d]\ar[r]& 0\\
0\ar[r] & \GGr(T) \ar[r]& \GGr(j_* \underline{\rm Ext}^1(E_K,\G_m)_{x_K}) \ar[r]&
\GGr( A^\prime) \ar[r]& 0
}
\end{equation*}
Since the vertical map on the left factors through a map $\nu^\ft\colon {}_n A^\prime(\OO_K)\to \GGr(T^\ft)$, the homomorphism
$u^\ft\colon \pi_1(\GGr (A^\prime))\to \pi_0 ( \GGr(T^\ft) )=\pi_0(\GGr(T))_\tor $ in \eqref{eq.u} factors through
the map $\delta\colon \pi_1(\GGr(A'))\to {}_n A^\prime(\OO_K)$ in \eqref{eq.pi1A} and hence
\[
u^\tau\stackrel{\eqref{eq.pi1cup}}= \partial\circ \delta\stackrel{\eqref{eq.partial}}{=}h\circ \nu \circ \delta=\pi_0(h^\ft)\circ \pi_0(\nu^\ft)\circ \delta=\pi_0(h^\ft)\circ u^\ft=u,
\]
 \emph{i.e.}, the homomorphism
$u\colon \pi_1(\GGr (A^\prime))\to \Z/n\Z$ in \eqref{eq.u} coincides with that in \eqref{eq.pi1cup}.
Thus we have
\begin{proposition}\label{pro.shap} Let $K$ be a complete discrete valued field with algebraically closed residue field and $A_K$ an abelian variety over $K$. Then the homomorphism $\Xi$ in Proposition~\ref{pro.sha0} coincides with Shafarevich's duality ${\rm H}^1_\fl(K,A_K)\stackrel{\sim}{\to} {\rm Hom}(\pi_1(\GGr(A^{\prime})), \Q/\Z)$ (see \eqref{eq.sha} and \cite{Ber}, Theorem~3), on the prime-to-$p$ parts.
\end{proposition}

The comparison for the $p$-parts in the equal positive characteristic case is still open.

\section{Comparison between \eqref{eq.phid} and \eqref{eq.sha}}\label{sec.comparison}
In this last section, we return to the study of torsors under an elliptic curve $A_K$, and we examine the relation
between the fundamental short exact sequence (\ref{eq.ZPicJ}) of Serre pro-algebraic groups with Shafarevich's duality of abelian varieties.
Let $n\geq 1$ be an integer and $X_K$ a torsor under the elliptic curve $A_K$ of order $d$ dividing $n$. Let $X$ denote the $S$-proper minimal regular model of $X_K$. Thanks to Corollary~\ref{cor.finalresult} we are provided with a short exact sequence of Serre
pro-algebraic groups
\begin{eqnarray}\label{eq.d0}
0\longrightarrow \mathbb{Z}/d\mathbb{Z}\longrightarrow {\bm {\Pic^{0}(X)}}\stackrel{\bm q}{\longrightarrow} {\bm {J(S)}}\longrightarrow 0
\end{eqnarray}
by sending $\bar{1}\in \mathbb{Z}/d\mathbb{Z}$ to $\mathcal{O}_{X}(D)\in \Pic^{0}(X)$. If we push out this short exact sequence
by the canonical map
\begin{eqnarray}\label{transition}
\mathbb{Z}/d\mathbb{Z}\longrightarrow \mathbb{Z}/n\mathbb{Z}, ~~~~\bar{1} \mapsto \frac{n}{d}\cdot \bar{1},
\end{eqnarray}
we get an element, denoted by $\Phi_n(X_K)$, of the group $\mathrm{Ext}^{1}(\mathbf{Gr}(J),\mathbb{Z}/n\mathbb{Z})$
of extensions of the pro-algebraic group $\GGr(J)=\bm{J(S)}$ by the constant group $\mathbb{Z}/n\mathbb{Z}$. Recall that by Lemma~\ref{J_K et A_K}~(i), there is a canonical isomorphism of elliptic curves $\iota\colon A_{K}'\xrightarrow{\sim} J_K$, thus the group $\mathbf{Gr}(J)$ does not depend on the torsor $X_K$.
In this way we get the following canonical map of sets:
\begin{eqnarray}\label{Phi_d}
\Phi_n\colon {}_n\mathrm{H}^{1}_{\mathrm{fl}}(K,A_K)\longrightarrow \mathrm{Ext}^{1}(\mathbf{Gr}(J),\mathbb{Z}/n\mathbb{Z}).
\end{eqnarray}
Motivated by the isomorphism \eqref{thetastar}, one might ask if this morphism is always an isomorphism.
Our strategy in studying this question is to relate the above construction to Shafarevich's pairing in \eqref{eq.sha} by using our new construction in \S~\ref{sec.picardconstruction} as an intermediate bridge.

\subsection{Some set-theoretical considerations}
 We begin with the following lemma, which follows from Lemma \ref{lemme cohomologique}.

\begin{lemma}\label{exofrig}
The schematic closure $Y$ in $X$ of any closed point $x_K$ of $X_K$ provides a rigidificator of the Picard functor $\Pic_{X/S}$.
\end{lemma}
\begin{proof}
We need only verify the injectivity of the map
$\HH^{0}(X_s, \OO_{X_s})\rightarrow \HH^{0}(Y_s,\OO_{Y_s})$ (Corollary~2.2.2 of \cite{Raynaud}).
More generally, we will prove by induction on $n$ that the canonical morphism
$\HH^{0}(X_n,\OO_{X_n})\rightarrow \HH^{0}(Y_n,\OO_{Y_n})$ is injective, where $Y_n:=Y\times_X X_n$.
Let $Y$ be defined by the ideal sheaf $\mathcal J$.
Then $Y_n$ is defined by the ideal sheaf $\mathcal{I}^{n}+\mathcal J$.
Let us begin with the case $n=1$: by Lemma~\ref{lemme cohomologique}, we know that $\HH^{0}(X_1,\OO_{X_1})=k$.
Let $\varepsilon\in \HH^{0}(X_1,\OO_{X_1})$; then $\varepsilon$ is a global function on $X_1$ and so is constant.
As a result, the image of $\varepsilon$ in $\HH^{0}(Y_1,\OO_{Y_1})$ is zero if and only if $\varepsilon=0$ that is, the morphism
$\HH^{0}(X_1,\OO_{X_1})\rightarrow \HH^{0}(Y_{1},\OO_{Y_1})$ is injective.
In order to complete the induction, consider the following diagram of sheaves over $X$
\[
\xymatrix{0\ar[r]& \frac{\mathcal{I}^{n}}{\mathcal{I}^{n+1}} \ar[r]\ar[d]&
 \frac{\OO_X}{\mathcal{I}^{n+1}} \ar[r]\ar[d]&
\frac{\OO_X}{\mathcal{I}^{n}} \ar[r]\ar[d]& 0\\
0\ar[r]& \frac{\mathcal{I}^{n}+\mathcal J}{\mathcal{I}^{n+1}+\mathcal J} \ar[r]&
 \frac{\OO_X}{\mathcal{I}^{n+1}+\mathcal J} \ar[r]&
\frac{\OO_X}{\mathcal{I}^{n}+\mathcal J} \ar[r]& 0
}
\]
where ${\rm H}^0(X, \frac{\OO_X}{\mathcal{I}^{m}})={\rm H}^0(X_m, \OO_{X_m})$ and ${\rm H}^0(X, \frac{\OO_X}{\mathcal{I}^{m}+\mathcal J})={\rm H}^0(Y_m, \OO_{Y_m})$.
Hence we need only establish the injectivity of the morphism
\begin{equation*}\label{H0}
\HH^{0}\left(X,\frac{\mathcal{I}^{n}}{\mathcal{I}^{n+1}}\right)\longrightarrow \HH^{0}\left(X,\frac{\mathcal{I}^{n}+\mathcal J}{\mathcal{I}^{n+1}+\mathcal J}\right).
\end{equation*}
Observe first that $\frac{ \mathcal{I}^{n} }{ \mathcal{I}^{n+1} }\cong \mathcal{I}^{n}\otimes_{\OO_X} \frac{\OO_X}{\mathcal{I}}$. Hence
\[
\HH^{0}\left( X,\frac{ \mathcal{I}^{n} }{ \mathcal{I}^{n+1} } \right)\cong \HH^{0}\left( X,\mathcal{I}^{n}\otimes_{\OO_X} \frac{\OO_X}{\mathcal I}
 \right)= \HH^{0}\left(X_1, \mathcal{I}^{n}|_{X_1}\right).
\]
Furthermore, consider the map
$ \frac{ \mathcal{I}^{n} }{ \mathcal{I}^{n+1} } \otimes_{\OO_X} \frac{\OO_X}{ \mathcal{I}+\mathcal J }\to
\frac{ \mathcal{I}^{n}+\mathcal J }{ \mathcal{I}^{n+1}+\mathcal J } $ that, on sections, maps
$\bar a\otimes\bar b$ to $\overline{ab}$. It is well defined and surjective.
Since $\mathcal I$ is invertible, $Y$ is integral and $Y\not\subseteq X_1$, so our map is also injective.
In particular, we find
\begin{equation*}\label{eq.IJ}
\frac{ \mathcal{I}^{n}+\mathcal J }{ \mathcal{I}^{n+1}+\mathcal J }
\cong
\frac{ \mathcal{I}^{n} }{ \mathcal{I}^{n+1} } \otimes_{\OO_X} \frac{\OO_X}{ \mathcal{I}+\mathcal J }
\cong
 \mathcal{I}^{n} \otimes_{\OO_X}
\frac{\OO_X}{ \mathcal{I} }
\otimes_{ \OO_X }
\frac{\OO_X}{ \mathcal{I}+\mathcal J }
\cong
 \mathcal{I}^{n} \otimes_{\OO_X} \frac{\OO_X}{\mathcal{I}+\mathcal J} .
\end{equation*}
Hence
\[ \HH^{0}\left( X,\frac{ \mathcal{I}^{n}+\mathcal J }{ \mathcal{I}^{n+1} +\mathcal J} \right)\cong \HH^{0}\left( X,\mathcal{I}^{n}\otimes_{\OO_X} \frac{\OO_X}{\mathcal I+\mathcal J} \right)= \HH^{0}\left(Y_1, \mathcal{I}^{n}|_{Y_1}\right).
\]
We are then reduced to proving that the restriction map
\[
\HH^{0}\left(X_1, \mathcal{I}^{n}|_{X_1}\right) \longrightarrow \HH^{0}\left(Y_1, \mathcal{I}^{n}|_{Y_1}\right)
\]
is injective. Since $\mathcal I$ is an invertible sheaf, according to Lemma~\ref{lemme cohomologique}, the first group is trivial or its consists of constant functions; hence the result follows.
\end{proof}

\begin{corollary}\label{cor.exofrig} Let $x_K=\spec(K')$ be a closed point of $X_K$ with $K'/K$ a finite \emph{separable} extension of degree $d$ (see Lemma~\ref{lem.index-period}), and let $Y:=\overline{\{x_K\}}\subset X$ be the schematic closure of $x_K$. Then the subscheme $Y\hookrightarrow X$ of $X$ is a rigidificator of the Picard functor $\mathrm{Pic}_{X/S}$. Moreover, the scheme $Y$ is regular. 
\end{corollary}
\begin{proof} The first statement follows directly from Lemma~\ref{exofrig}. Now, $Y$ is a local scheme since $\OO_K$ is complete. Let $y$ be the closed point of $Y$. Since  $Y$ is of degree $d$ over $S$ we have $Y\cdot X_s=d$. Furthermore, $X_s=d D$ as divisor of $X$, hence $Y\cdot D=1$ and $Y$ cuts $D$ transversally at  $y$.
 Let $A$ be the local ring of $X$ at $y$. Let $f$ (respectively $g$) be a local equation around the point $y$  which defines $Y$ (respectively $D$). 
Then the local ring $A/(f,g)$ is of length $1$, hence it is isomorphic to $k$. As a result, $(f,g)$ is a system of parameters of the two dimensional regular local ring $A$. Therefore, $Y$ and $D$ both are regular at the point $y$. Thus the scheme $Y$ is regular.
\end{proof}
In the following, let $Y$ be the rigidificator given in Corollary~\ref{cor.exofrig}. We will use notation as in \S~\ref{Rappels Pic} and \S~\ref{filtrations}.
In particular, $G=(\Pic_{X/S},Y)^0$ is the identity component of the rigidified Picard scheme $(\Pic_{X/S},Y)$, and we have
the following canonical map
\[
r\colon G=(\mathrm{Pic}_{X/S},Y)^{0}\longrightarrow \mathrm{Pic}^{0}_{X/S}
\]
that forgets the rigidification. Let $N$ be the kernel of the morphism $r$. In general, this fppf-sheaf $N$ is not representable, but it has representable fibres. Following \S~\ref{Not}, let $H=\overline{N_K} \hookrightarrow (\mathrm{Pic}_{X/S},Y)^{0}=G$ denote the schematic closure of $N_K$ in $G$; it is representable by a flat $S$-group scheme of finite type. Then the fppf quotient $G/H$ gives us the identity component $J$ of
the $S$-N\'eron model of the Jacobian $J_{K}=\Pic^{0}_{X_K/K}$ of the curve $X_K/K$, and one has the following exact sequence
of $S$-group schemes:
\[
0\longrightarrow H\longrightarrow G\stackrel{\theta}{\longrightarrow} J\longrightarrow 0.
\]
which induces an exact sequence of abstract groups (\S~\ref{Rappels Pic}):
\begin{eqnarray}\label{eq.SPtsofJ}
0\longrightarrow H(S)\longrightarrow G(S)\longrightarrow J(S)\longrightarrow 0.
\end{eqnarray}
On the other hand, by definition, we have another exact sequence of sheaves, which is exact for the \'etale topology
(since \eqref{eq.picrig} in \S~\ref{Rappels Pic} is exact for the \'etale topology):
\[
0\longrightarrow N\longrightarrow G \stackrel{r}{\longrightarrow} \Pic_{X/S}^{0}\longrightarrow 0.
\]
Since $S$ is strictly henselian, the latter sequence induces the following short exact sequence of abstract groups:
\begin{eqnarray}\label{eq.SPtsofPic}
0\longrightarrow N(S)\longrightarrow G(S)\longrightarrow \mathrm{Pic}^{0}(X)\longrightarrow 0.
\end{eqnarray}
On combining \eqref{eq.SPtsofJ} and \eqref{eq.SPtsofPic}, we get the following commutative diagram of abstract groups with exact rows:
\begin{equation}\label{eq.diaHN}
\xymatrix{
0\ar[r]& N(S)\ar[r]\ar@{^{(}->}[d]& G(S)\ar[r]\ar@{=}[d]&
\mathrm{Pic}^{0}(X)\ar[r]\ar@{->>}[d]& 0\\
 0\ar[r]& H(S)\ar[r]\ar@{^{(}->}[d]& G(S)\ar[r]\ar@{^{(}->}[d]&
J(S)\ar[r]\ar@{^{(}->}[d]& 0 \\
0\ar[r] & N(K)\ar[r]& G(K)\ar[r]& \mathrm{Pic}^{0}_{X/S}(K)\ar[r]& 0 \\
}
\end{equation}
where the lower sequence is exact on the right because $N_K$ is a torus, the upper vertical map on the right is surjective (\cite{BLR}, 9.5/2)
and the remaining vertical maps are all injective.

\subsection{The pro-algebraic nature of diagram \eqref{eq.diaHN}}
For $n\in \mathbb{Z}_{\geq 1}$, as in \S \ref{sec.proalg-gree} we put $S_n=\mathrm{Spec}(\mathcal{O}_{K,n})=\mathrm{Spec}(\mathcal{O}_{K}/ \pi^n\OO_K)$ and let $\mathbb{R}_n$ denote the Greenberg algebra associated with $ \mathcal{O}_{K,n}$ (Appendix A of \cite{Lipman}). The aim of this subsection
is to show, with the help of Greenberg realization functors, that the diagram \eqref{eq.diaHN} is pro-algebraic in nature.

First, the sheaf $H$ is representable by an $S$-group scheme separated of finite type.
Hence its Greenberg realization $\mathrm{Gr}_n(H)$ is representable by a $k$-scheme of finite type (\S~\ref{sec.proalg-gree})
and we have the following short exact sequence:
\begin{equation*}\label{eq.GrHGJ}
0\longrightarrow \mathrm{Gr}_n(H)\longrightarrow \mathrm{Gr}_n(G)\longrightarrow \mathrm{Gr}_n(J)\longrightarrow 0.
\end{equation*}
For the right exactness, we need only prove that the map
$\mathrm{Gr}_n(G)\rightarrow \mathrm{Gr}_n(J)$
induces a surjective map on the groups of $k$-rational points, \emph{i.e.}, that the morphism of group $G(S_n)\rightarrow J(S_n)$ is surjective.
This last statement follows from the surjectivity of the maps $\theta(S)\colon G(S)\rightarrow J(S)$ (see \S~\ref{Rappels Pic} last paragraph)
and $J(S)\to J(S_n)$.
On passing to the projective limit of the associated perfect group schemes, one obtains an extension of Serre pro-algebraic groups
\begin{equation*}\label{eq.NPQ}
0\longrightarrow \GGr(H) \longrightarrow \GGr( G) \longrightarrow \GGr(J)\longrightarrow 0
\end{equation*}
which says that \eqref{eq.SPtsofJ} is pro-algebraic in nature.

Next, we consider the fppf sheaf $N$. Let us first remark that for any
$k$-algebra ${\mathsf A}$, by considering the $\mathcal{O}_{K,n}$-algebra $\mathbb{R}_{n}({\mathsf A})$, we have the following exact sequence of groups:
\[
0\longrightarrow N(\mathbb{R}_n({\mathsf A}))\longrightarrow G(\mathbb{R}_n({\mathsf A}))\longrightarrow \mathrm{Pic}_{X/S}^{0}(\mathbb{R}_n({\mathsf A})) .
\]
Let $\mathrm{Gr}_n(N)$ be the fppf sheaf associated with the pre-sheaf $\mathsf{A}\mapsto N(\mathbb{R}_n({\mathsf A}))$. By taking the
associated fppf sheaves, we get the following exact complex of algebraic $k$-groups (where the representability of $\mathrm{Gr}_n(N)$
follows from the representability of the last two functors by smooth $k$-group schemes):
\begin{equation}\label{eq.GrPic}
0\longrightarrow \mathrm{Gr}_{n}(N)\longrightarrow \mathrm{Gr}_n\left(G\right)\longrightarrow
\mathrm{Gr}_n(\mathrm{Pic}_{X/S}^{0}).
\end{equation}
By taking the $k$-rational points, we get the usual exact sequence
\[
0\longrightarrow N(S_n)\longrightarrow G(S_n)\longrightarrow \mathrm{Pic}_{X/S}^0(S_n)\longrightarrow 0
\]
which is exact on the right since $ \mathcal{O}_{K,n}$ is strictly henselian.
So the complex \eqref{eq.GrPic} is in fact a short exact sequence of algebraic $k$-groups.
 Now, by taking the projective limit with respect to $n$ in the sequence of perfect group schemes associated with \eqref{eq.GrPic} we get a
short exact sequence of Serre pro-algebraic groups:
\begin{equation*}\label{eq.proAlgPic}
0\longrightarrow \bm{N(S)}\longrightarrow \GGr(G)\longrightarrow
\bm{\Pic^{0}(X)}\longrightarrow 0.
\end{equation*}

Finally, the group scheme $N_K$ is a torus, $\mathrm{Pic}^{0}_{X_K/K}\cong A_K'$ is
an elliptic curve and $G_K=(\mathrm{Pic}_{X/S},Y)_{K}^{0}$ is a semi-abelian variety; hence they all admit N\'eron models over $S$, which will be denoted by $\mathcal{N}$, $ A'$, and $\mathcal{G}$ respectively; in particular they are smooth group schemes over $S$.
Moreover, according to Remark~\ref{rem.tori}, the $S$-group schemes $\mathcal{N}$, ${A}'$ are of finite type over $S$, and hence
the same holds for $\mathcal{G}$.

 By the N\'eron mapping property we have the following two canonical maps
\[
f_G\colon G\longrightarrow \mathcal{G}, \ \ \ \ \ \textrm{and} \ \ f_A\colon J\longrightarrow A'.
\]
As a consequence, the morphisms
\[
G(S)\longrightarrow G(K)=\mathcal{G}(S), \ \ \ \ \ \textrm{and} \ \ J(S)\longrightarrow \mathrm{Pic}^{0}_{X/S}(K) \cong{A}'(S)
\]
in diagram \eqref{eq.diaHN} come from the morphisms of Serre pro-algebraic groups
\[
\GGr(G)\longrightarrow\GGr({\mathcal{G}}), \quad \text{and} \quad \GGr(J)\longrightarrow \GGr({A}')
\]
induced by $f_G, f_A$.
This implies the existence of a morphism of pro-algebraic groups
\[
\alpha\colon
\bm{H(S)}\longrightarrow \bm{N(K)}=\GGr({\mathcal N})
\]
which realizes the lower left vertical inclusion in \eqref{eq.diaHN}.
Summarizing, we have that \eqref{eq.diaHN} comes from a commutative diagram (with exact rows) of Serre pro-algebraic groups 
\begin{equation}\label{eq.diaHNpro}
\xymatrix{
0\ar[r]& \bm{N(S)}\ar[r]\ar[d]& \GGr(G)\ar[r]\ar@{=}[d]& \bm{\Pic^0(X) } \ar[r]\ar@{->>}[d]& 0\\
0\ar[r]& \bm{H(S)}\ar[r]\ar[d]^{\alpha}& \GGr(G) \ar[r]\ar[d]& \GGr(J) \ar[r]\ar[d]& 0 \\
0\ar[r] & \bm{N(K)}\ar[r]& \GGr({\mathcal G}) \ar[r]&
\bm{ \Pic^0_{X/S}(K) } \cong \GGr({A}')\ar[r]& 0}
\end{equation}

\subsection{Comparison}
We deduce from \eqref{eq.diaHNpro} a commutative diagram of profinite groups:
\begin{equation}\label{dia.pi10}
\xymatrix{\pi_1 \left(\GGr(J)\right)\ar[r]^<<<<<{\sim }\ar[d]^{}& \pi_1 \left( \GGr(A') \right)\ar[d]^{} \\
\pi_0\left(\bm{H(S)}\right)\ar[r]^{\pi_0(\alpha)} & \pi_0\left(\bm{N(K)}\right)}
\end{equation}
The upper arrow is an isomorphism because by Lemma~\ref{J_K et A_K} we have a canonical isomorphism ${ A}^{\prime 0} \stackrel{\sim}{\to} J$, hence $ \GGr({ A}^\prime)^{0} \stackrel{\sim}{\to} \GGr(J) $.

In order to give an explicit description of the morphism $\pi_0(\alpha)$, recall first of all that the group of
connected components $\pi_0\left(\bm{N(K)}\right)\cong \pi_0(\mathcal N)$ is isomorphic to $\mathbb{Z}/d\mathbb{Z}$
(cf. Remark~\ref{rem.tori} and \eqref{eq.diagramTNA}), with identification given by
\[
\beta\colon \pi_0\left(\bm{N(K)}\right)\longrightarrow \mathbb{Z}/d\mathbb{Z}, \ \ \ \ \textrm{class of }\pi'\textrm{ in~~} N_K(K)=k(x_K)^{\ast}/K^{\ast}\mapsto \bar 1
\]
where $\pi'\in k(x_K)$ is a uniformizer. Furthermore, the class of $\pi'$ in $N_K(K)$,
viewed as element of $G(K)=(\Pic_{X/S},Y)^{0}(K)$ (see \S~\ref{Rappels Pic}), is the trivial line bundle on $X_K$ with the
rigidification on $Y_K$ given by the multiplication by $\pi'$.

Second, the component group of $\bm{H(S)}$ is also $\Z/d\Z$.
Indeed our group scheme $H$ coincides with the one denoted by $H_1$
in \cite{LLR}, pp.~18--21.
We then have the following exact sequence
\[
V_{Y}^{\ast}(S)\longrightarrow H(S)\longrightarrow \mathbb{Z}/d\mathbb{Z}\longrightarrow 0,
\]
(\emph{loc. cit.}, Theorem 3.5) where the first map is the natural factorization of
$V_Y^*\rightarrow N\hookrightarrow G:=(\mathrm{Pic}_{X/S},Y)^{0}$ through $H
\hookrightarrow G$ (since $V_Y$ is flat over $S$), and the second map is defined by
\[
\gamma\colon H(S)\longrightarrow \mathbb{Z}/d\mathbb{Z}, \ \ \ \
\left(\mathcal{O}_{X}\left(\frac{m}{d}X_s\right), a\right)\mapsto \bar{m}\in \mathbb{Z}/d\mathbb{Z}.
\]
(see \cite{LLR}, 3.5).
Applying the perfect Greenberg functor to the morphisms $V_X^*\to V_Y^*\to G$ one sees that the map $\gamma$
is of pro-algebraic nature and we write:
\[
\pi_0(\gamma)\colon \pi_0\left(\bm{H(S)}\right)\stackrel{\sim}{\longrightarrow} \mathbb{Z}/d\mathbb{Z}
\]

\begin{lemma}
The following diagram \begin{equation}\label{comparePi0}
\xymatrix{\pi_0\left(\bm{H(S)}\right)\ar[r]^{\pi_0(\alpha) }\ar[d]^{\pi_0(\gamma)}_\wr& \pi_0\left(\bm{N(K)}\right)\ar[d]^{\beta}_\wr
\\ \mathbb{Z}/d\mathbb{Z} \ar[r]^{\bar 1\mapsto -\bar 1}& \mathbb{Z}/d\mathbb{Z}}
\end{equation}
commutes. In particular, the morphism $\pi_0(\alpha)$ is an isomorphism.
\end{lemma}

\begin{proof} Recall that $D$ denotes the vertical divisor $\frac{1}{d}X_s$. Since the map $H(S)\to N(K)$ sends $(\mathcal{O}_{X}(-D), a)$ to its generic fibre, and the generic fibre of $\mathcal{O}_X(-D)$ is trivial,
we are reduced to verifying that a rigidification on $Y_K$ can be given by multiplication by the uniformizer $\pi'$, and that this rigidification extends to a rigidification of $\mathcal{O}_X(-D)$ on $Y$.

Now consider $\mathcal{O}_X(-D)$.
This gives us an ideal sheaf of $\mathcal{O}_{X}$.
Recall that $Y$ is regular (Corollary~\ref{cor.exofrig}), hence $Y=\mathrm{Spec}(R')$ with $R'$ a complete discrete valuation ring whose field of fractions is $K':=k(x_K)$.
Next, we claim that the intersection of $Y$ and $D$, viewed as a divisor of $Y$, is defined by the equation $\pi'=0$.
In fact, let $y\in Y$ be its closed point, and consider the local ring $\mathcal{O}_{X,y}$ which is regular of dimension $2$.
Let $r\in \mathcal{O}_{X,y}$ (respectively $t\in \mathcal{O}_{X,y}$) be a defining equation of $Y$ (respectively of $D$) around
$y\in X$. Since $X_s=dD$ as divisor of $X$, we have $(\pi)=(t^d)\subset \mathcal{O}_{X,y}$.
 By definition, $Y/S$ is of degree $d$; it
follows that the intersection number in $X$
\begin{equation*}\label{eq.intersection}
Y\cdot X_k=\ell(\mathcal{O}_{X,y}/(r,\pi))=\ell(\mathcal{O}_{X,y}/(r,t^d))=\ell(R'/(t')^d)
\end{equation*}
is equal to $d$, with $t'$ the image of $t$ in $R'=\mathcal{O}_{X,y}/(r)$.
This implies that $t'$ is an uniformizer of $R'$ (and the maximal ideal of $\mathcal{O}_{X,y}$ is generated by $r$ and $t$).
Hence $t'=u'\pi'$ with $u'\in R'$ a unit.
As a result, the intersection $Y\cap D$ is defined by the equation $t'=0$, or equivalently, by the equation $\pi'=0$ in $Y$.

So by the claim, we have $\mathcal{O}_{Y}(-D\cap Y)=(\pi')$. We get in this way a rigidification of $\mathcal{O}_{X}(-D)$ along $Y$
\begin{equation*}\label{eq.a}
a\colon \widetilde { R'}=\mathcal{O}_Y\longrightarrow \mathcal{O}_{X}(-D)|_{Y}=(\pi'), \ \ \ \ \ \ 1\mapsto \pi',
\end{equation*}
with $\widetilde{R'}$ the coherent module associated with $R'$.
Now, if we restrict to the generic point, we get
\begin{equation*}\label{eq.aK}
a_K \colon \widetilde { K'}=\mathcal{O}_{Y_K}\longrightarrow \mathcal{O}_{X}(-D)|_{Y_K}=(\pi')=\widetilde { K'}, \ \ \ \ \ \ 1\mapsto \pi'.
\end{equation*}
\end{proof}

Now, on forgetting the rigidifications, we have the following exact sequence of Serre pro-algebraic groups
\[
\xymatrix{0\ar[r]& \bm{H(S)}\ar[r]\ar[d]^{\gamma'}& \GGr(G) \ar[r]\ar[d]& \GGr(J)\ar[r]\ar@{=}[d]& 0
\\ 0\ar[r]& < \mathcal{O}_{X}(D)> \ar[r]& \bm{\mathrm{Pic}^{0}(X)}\ar[r]^{\bm q}& \GGr(J) \ar[r]& 0},
\]
where $\gamma'$ is given on $k$-rational sections by $\left(\mathcal{O}_{X}({\frac md}D),a\right)\mapsto \mathcal{O}_{X}({\frac md}D)$,
and the vertical map in the middle is given by $(\mathcal{L},a)\mapsto \mathcal{L}$.
Hence, if we identify the kernel $ <\mathcal{O}_{X}(D)>$ of $\bm q$ with $\mathbb{Z}/d\mathbb{Z}$ by sending $\mathcal{O}_{X}(D)$ to $\bar 1\in \mathbb{Z}/d\mathbb{Z}$, we get the extension of pro-algebraic groups:
\begin{equation}\label{eq.extPicQ}
0\longrightarrow \mathbb{Z}/d\mathbb{Z}\longrightarrow \bm{\mathrm{Pic}^{0}(X)}\longrightarrow \GGr(J)\longrightarrow 0,
\end{equation}
and this will give us an element of $
\mathrm{Ext}^{1}\left( \GGr(J),\mathbb{Z}/d\mathbb{Z}\right)$, which is canonically isomorphic to $\mathrm{Ext}^1\left(\GGr({A}^{\prime 0}),\mathbb{Z}/d\mathbb{Z}\right).$ By construction the map $\gamma'\colon \bm{H(S)}\to <\mathcal{O}_{X}(D)>\cong\Z/d\Z$ is the composition of $ \bm{H(S)} \to \pi_0( \bm{H(S)} )$ with $\pi_0(\gamma)\colon \pi_0( \bm{H(S)} )\to \Z/d\Z$.
Hence the commutativity of \eqref{comparePi0} and \eqref{dia.pi10} implies that the extension \eqref{eq.extPicQ} is the opposite of the extension obtained from lower exact sequence in \eqref{eq.diaHNpro} by taking first the push-out along $\bm{N(K)}\to \pi_0( \bm{N(K)})\cong \Z/d\Z$ and then the pull-back along the canonical map $\GGr(J)\rightarrow \GGr({A}')$. We can summarize these facts as follows:

\begin{proposition} The extension \eqref{eq.extPicQ} is the opposite of the extension obtained from the homomorphism \eqref{eq.v} via the canonical isomorphisms
\[\mathrm{Hom}\left(\pi_1\left( \GGr({A}^{\prime}\right), \mathbb{Z}/d\mathbb{Z} \right)\stackrel{\sim}{\longleftarrow}
\mathrm{Ext}^1\left(\GGr({A}^{\prime 0}),\mathbb{Z}/d\mathbb{Z}\right) \stackrel{\sim}{\longleftarrow}\mathrm{Ext}^{1}\left(\GGr(J),\mathbb{Z}/d\mathbb{Z}\right). \]
\end{proposition}
The minus sign depends on the choice of the isomorphism $<\OO_X(D)>\cong \Z/d\Z$. 

\begin{corollary}\label{cor.phin} Let $A_K$ be an elliptic curve and $J$ the identity component of the N\'eron model $A'$ of the dual elliptic curve $A_K'$ over $S$. Let $n\in \mathbb{Z}_{\geq 1}$. The map $\Phi_n\colon {}_n\mathrm{H}^{1}_{\mathrm{fl}}(K,A_K)\to \mathrm{Ext}^{1}(\mathbf{Gr}(J),\mathbb{Z}/n {\mathbb{Z}})$ in \eqref{Phi_d} is an injective morphism of groups,
which is an isomorphism if one of the following conditions is verified:
\begin{itemize}
\item The field $K$ has characteristic $0$.
\item The integer $n$ is prime to $p$.
\end{itemize}
If one of the above conditions is satisfied, then the composition of $\Phi_n$ with the isomorphism $\mathrm{Ext}^{1} \left( \GGr(J),\mathbb{Z}/n\mathbb{Z} \right) \stackrel{\sim}{\to} \mathrm{Hom}\left( \pi_1 \left( \GGr(A^{\prime}) \right),\mathbb{Z}/n\mathbb{Z}\right)$ coincides with the restriction of Shafarevich's duality \eqref{eq.sha} to the $n$-parts. 
\end{corollary}
\begin{proof}
 In view of the previous Proposition and results in \S~\ref{ShafaPairing} (Theorem~\ref{thm.main} and Proposition~\ref{pro.sha0})
only the injectivity requires verification. Since, there is no non-zero morphism from the connected pro-algebraic group $\mathbf{Gr}(J)$
to a constant finite group, the canonical map (\ref{transition}) induces an injective maps between the group of extensions
\[
\mathrm{Ext}^{1}(\mathbf{Gr}(J),\mathbb{Z}/n'\mathbb{Z})\longrightarrow \mathrm{Ext}^{1}(\mathbf{Gr}(J),\mathbb{Z}/n\mathbb{Z})
\]
when $n'|n$. Hence, we only need to show that, for $X_K$ a torsor under $A_{K}$ of order $d$, the extension (\ref{eq.d0})
is non-zero in $\mathrm{Ext}^{1}(\mathbf{Gr}(J),\mathbb{Z}/d\mathbb{Z})$ unless $d=1$.
Since the pro-algebraic group $\bm{\mathrm{Pic}^{0}(X)}$ is also connected, extension (\ref{eq.d0}) is split
if only if $d=1$, and this fact implies that the torsor $X_K$ is in fact trivial.
\end{proof}

\begin{remark}
The problem of extending the above Corollary to the $p$-parts in the equal characteristic case reduces to showing that $\Xi$ in Proposition~\ref{pro.sha0} is always an isomorphism, for example, by checking that it coincides with \eqref{eq.sha} on the $p$-parts too. Although we have partial results in this direction, for example in the case of abelian varieties with totally degenerate reduction, a full answer is not yet at hand.
\end{remark}

\thanks{\emph{Acknowledgements:}
We thank M. Raynaud for suggesting this subject to us and for useful discussions. We thank the referee for the very detailed review of our paper. The preprint \cite{To} was done when the second author was a post-doc in Essen, and he thanks H. Esnault for the hospitality. He also thanks D.~Lorenzini, W.~Zheng and C.~P\'epin for their remarks. He wants to thank especially M. Raynaud for allowing him to use his unpublished work \cite{Raynaud4}, and for his generous help during the preparation of \cite{To}. Both authors thank Progetto di Eccellenza Cariparo 2008-2009 ``Differential methods in Arithmetics, Geometry and Algebra'' for financial support.


\end{document}